\documentclass[11pt]{amsart}

\usepackage{color}
\usepackage{comment}
\usepackage[margin=1in]{geometry}
\usepackage{graphicx}
\usepackage{yfonts}

\def\eq{\begin{equation}}
\def\endeq{\end{equation}}
\def\bbm{\begin{bmatrix}}
\def\ebm{\end{bmatrix}}
\def\bpm{\begin{pmatrix}}
\def\epm{\end{pmatrix}}
\newcommand{\C}{\mathbb{C}}
\newcommand{\Hankel}{\mathcal{H}}
\newcommand{\bigO}{\mathcal{O}}
\newcommand{\Prob}{\mathbb{P}}
\newcommand{\realR}{\mathbb{R}}
\newcommand{\T}{{\mathbb T}}
\newcommand{\Z}{\mathbb{Z}}

\newcommand{\ga}{\gamma}
\newcommand{\de}{\delta}
\newcommand{\ep}{\epsilon}
\newcommand{\Sg}{\Sigma}
\newcommand{\sg}{\sigma}

\newcommand{\Om}{\Omega}
\newcommand{\om}{\omega}
\newcommand{\hsgn}[1]{\epsilon(#1)}
\newcommand{\lattice}{L_{n, \tau}}

\newtheorem{theo}{Theorem}[section]
\newtheorem{prop}[theo]{Proposition}
\newtheorem{rhp}[theo]{Riemann--Hilbert Problem}
\newtheorem{lem}[theo]{Lemma}

\numberwithin{equation}{section}

\title[The $k$-tacnode process]{The $k$-tacnode process}

\author{Robert Buckingham}

\author{Karl Liechty}

\thanks{The authors thank Dong Wang, who was involved in this project at an early stage. Robert Buckingham was supported by by the National Science Foundation 
through grant DMS-1615718 and by the Charles Phelps Taft Research Center 
through a Faculty Release Fellowship.  Karl Liechty was supported by the 
Simons Foundation through grant \#357872.}

\begin{document}

\begin{abstract} 
The tacnode process is a universal behavior arising in nonintersecting 
particle systems and tiling problems.  For Dyson Brownian bridges, the 
tacnode process describes the grazing collision of two packets of walkers.  
We consider such a Dyson sea on the unit circle with drift.  For any 
$k\in\mathbb{Z}$, we show that an appropriate double scaling of the 
drift and return time leads to a generalization of the tacnode process in 
which $k$ particles are expected to wrap around the circle.  We derive 
winding number probabilities and an expression for the correlation kernel 
in terms of functions related to the generalized Hastings--McLeod solutions to 
the inhomogeneous Painlev\'e-II equation.  The method of proof is asymptotic 
analysis of discrete orthogonal polynomials with a complex weight.

\end{abstract}


\maketitle

\tableofcontents

\section{Introduction}

Ensembles of nonintersecting particles, growth and tiling processes, random 
matrix models, and related probabilistic systems exhibit a 
variety of non-classical universal stochastic processes.  These include the 
sine process describing bulk spacing \cite{Mehta:2004}, the Airy process at a 
soft edge \cite{TracyW:1994}, the $k$-Airy process for the initial opening of 
a new band at a soft edge 
\cite{BaikBP:2005,Baik:2006,AdlerDv:2009,BertolaBLP:2012,BaikW:2013}, the 
Bessel process 
at a hard edge \cite{TracyW:2007,DesrosiersF:2008,KuijlaarsMW:2011}, and the 
Pearcey process with or without inliers at the merger of two bands 
\cite{TracyW:2006,BleherK:2007,AdlerOv:2010,AdlerDvV:2011,LiechtyW:2016}.  Here we consider 
the \emph{tacnode process} describing the statistics when two band endpoints 
collide and separate without merging.  

The tacnode process first appeared in the paper \cite{AdlerFv:2013} of Adler, Ferrari, and van Moerbeke 
in the analysis of continuous-time 
nonintersecting random walks on $\mathbb{Z}$ with fixed starting and ending 
points.  Their results were expressed in terms of resolvents and Fredholm 
determinants of the Airy integral operator acting on a semi-infinite 
interval.  Very shortly thereafter, Delvaux, Kuijlaars, and Zhang \cite{DelvauxKZ:2011} studied the 
tacnode process for nonintersecting Brownian bridges on the line with two 
fixed starting and ending points.  They analyzed a certain $4\times 4$ 
Riemann--Hilbert problem (see also 
\cite{DuitsG:2013,Kuijlaars:2014,LiechtyW:2016b}) associated to the 
\emph{Hastings--McLeod solution} to the homogeneous Painlev\'e-II equation 
defined below in \eqref{PII-hom} and \eqref{uhm-gen-plus-inf}.  Johansson 
\cite{Johansson:2013} subsequently computed a resolvent form of the tacnode 
kernel for nonintersecting Brownian motions, but it was quite different from the one found in \cite{AdlerFv:2013}. The equivalence between the Airy resolvent formulas appearing in  \cite{AdlerFv:2013} and \cite{Johansson:2013} was shown in \cite{AdlerJv:2014} by computing a scaling limit of the statistics in the domino tiling of the double Aztec 
diamond in two different ways. Ferrari and Vet\H o \cite{Ferrari:2012}  extended the results of Johansson  \cite{Johansson:2013} to the non-symmetric case when there are a different number of 
walkers in the two groups. Then the equivalence of the 
Riemann--Hilbert form of the kernel \cite{DelvauxKZ:2011} and resolvent form of Johansson \cite{Johansson:2013} was shown by Delvaux 
\cite{Delvaux:2013}, who was also able to extend the Riemann--Hilbert formulation to the non-symmetric tacnode kernel.  Liechty and Wang \cite{LiechtyW:2017} showed that 
the even (respectively, odd) parts of the tacnode kernel (with respect to the 
spatial variables) arise in appropriate scaling limits of nonintersecting 
Brownian bridges on an interval with reflecting (respectively, absorbing) 
boundary conditions.  
Transitions between 
the tacnode kernel and the Airy and Pearcey kernels were studied by 
Bertola and Cafasso \cite{BertolaC:2013}, Girotti \cite{Girotti:2014}, and 
Geudens and Zhang \cite{GeudensZ:2015}.

Liechty and Wang \cite{LiechtyW:2016} showed the tacnode process occurs 
in a Dyson sea of nonintersecting Brownian bridges on the unit circle with a 
single starting and ending point.  Suppose $n$ Brownian particles on $\T = \{ e^{i\theta} |\, \theta \in \realR\}$ with diffusion 
parameter (i.e. standard deviation) $n^{-1/2}$ are conditioned to start at 
angle $\theta=0$ at time $t=0$, to return to the starting position at a fixed 
return time $T$, and to not intersect for $0<t<T$.  In the large-$n$ limit,  
if $T<\pi^2$ then there is a portion of 
the unit circle no particles are expected to visit (Figure 
\ref{no-drift-plots}, left panel), whereas if $T>\pi^2$ then there is a time 
interval in which the bulk of particles is expected to cover the entire 
circle (Figure \ref{no-drift-plots}, right panel).  In the critical case 
$T=\pi^2$, particles are expected to visit the far side of the circle, 
but only at the halfway time $t=T/2=\pi^2/2$ (Figure \ref{no-drift-plots}, 
center panel).  The behavior when $T\approx\pi^2$ in a neighborhood 
of $\theta=-\pi$ around the time $t=T/2$ is described by the tacnode 
process.  In the current work, by adding a small amount of drift we 
introduce a generalization of 
the tacnode process (the \emph{$k$-tacnode process}) in which two band 
endpoints collide and separate, after which $k$ particles are expected to 
switch from one endpoint to the other (Figure \ref{random-walk-plots}, right 
panel).  
\begin{figure}[h]
\begin{center}
\includegraphics[height=1.8in]{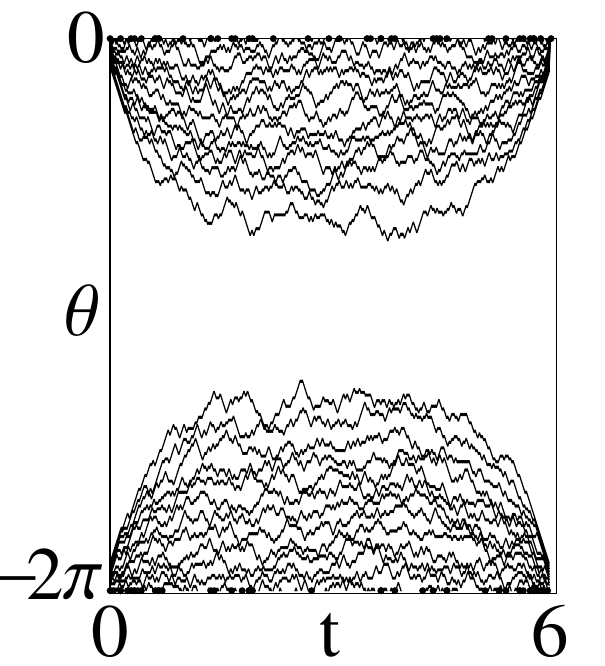}
\includegraphics[height=1.8in]{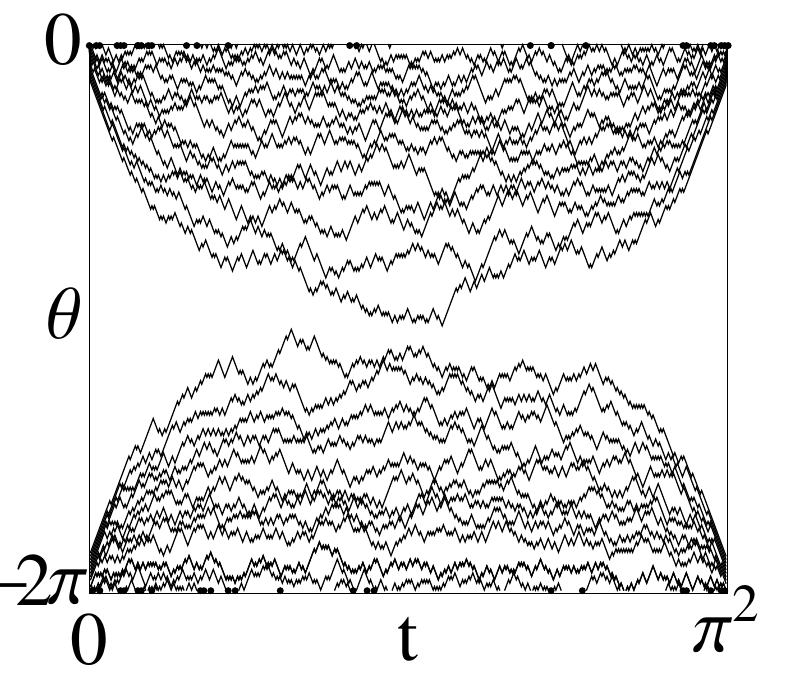}
\includegraphics[height=1.8in]{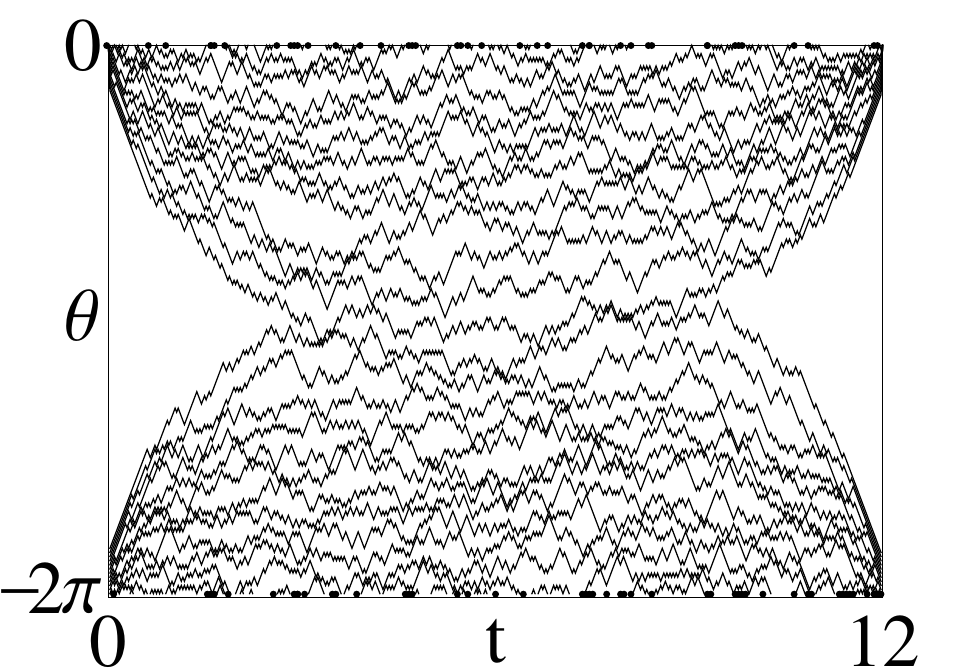}
\caption{Random walk simulations of a Dyson Brownian bridge with 24 walkers 
in subcritical (left), critical (center), and supercritical 
(right) cases.  The tacnode process describes the statistics of the critical 
case.}
\label{no-drift-plots}
\end{center}
\end{figure}

\begin{figure}[h]
\begin{center}
\includegraphics[height=2in]{random-walks-crit2}
\includegraphics[height=2in]{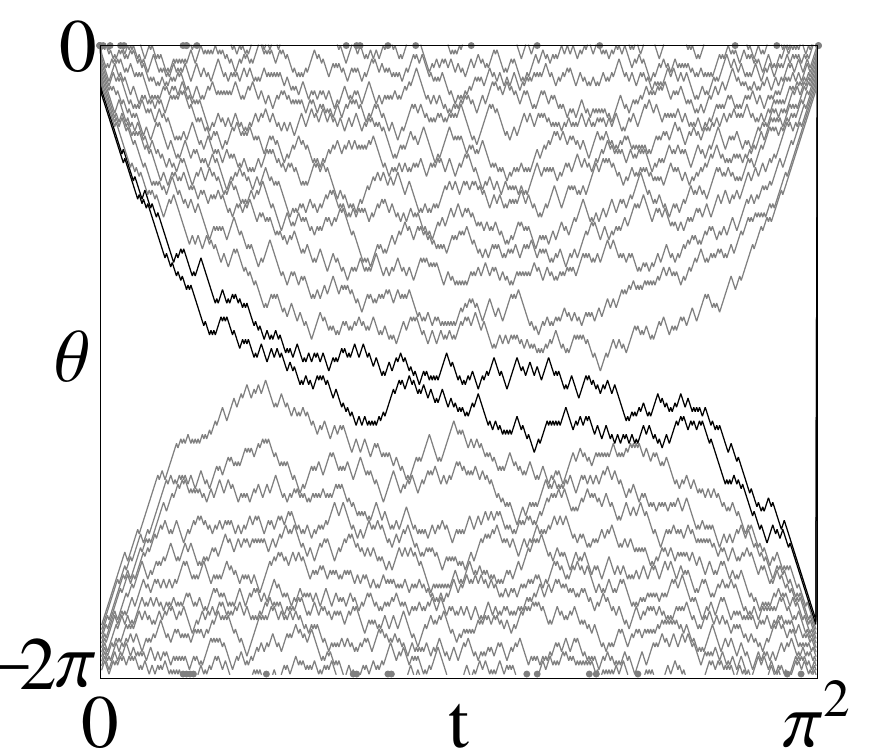}
\caption{Random walk simulations of a Dyson Brownian bridge with 24 walkers in 
the critical case $T=\pi^2$ with total winding 0 (left) and $-2$ (right).}
\label{random-walk-plots}
\end{center}
\end{figure}

More precisely, consider the determinantal process $\text{NIBM}_{0\to T}^\mu$ 
introduced in \cite{BuckinghamL:2017}.  This process describes $n$ 
nonintersecting Brownian bridges on the circle with the addition of a (real) 
drift $\mu$ and is defined as follows.  Given a positive integer 
$n$ and $\tau\in[0,1]$, define the lattice 
\begin{equation}
\label{lattice-def}
  \lattice := \left\{ \frac{m + \tau}{n} \mid m \in \Z \right\}.
\end{equation}
Define $p^{(T,\mu,\tau)}_{n, j}(x)$ to be the monic polynomial of degree $j$ satisfying
\begin{equation} \label{eq:defn_of_discrete_Gaussian_OP}
 \frac{1}{n} \sum_{x \in \lattice} p^{(T,\mu,\tau)}_{n, j}(x) p^{(T,\mu,\tau)}_{n, k}(x)e^{-\frac{Tn}{2}(x^2-2i\mu x)} = h_{n,j}^{(T,\mu,\tau)}\de_{jk},
\end{equation}
where $\{h_{n,j}^{(T,\mu,\tau)}\}_{j=0}^\infty$ are the normalizing constants. These polynomials are defined to be orthogonal with respect to a non-Hermitian weight, and so their existence is not guaranteed.  In this paper we analyze the Riemann--Hilbert problem for these orthogonal polynomials and prove their existence for large enough $n$. Then the existence of the full system of polynomials for small $|\mu|$ follows via a perturbation argument from the $\mu=0$ case, in which the orthogonality is Hermitian. 

Now define the $\tau$-deformed correlation kernel 
\eq
\label{tau-kernel-ext}
K_{t_i, t_j}(\varphi,\theta) := \widetilde{K}_{t_i, t_j}(\varphi,\theta) - {\mathop{W}^\circ}_{[i,j)}(\varphi,\theta),
\endeq
where
\eq
\label{tau-kernel}
\widetilde{K}_{t_i, t_j}(\varphi,\theta):= \frac{n}{2\pi}\sum_{j=0}^{n-1}\frac{1}{h_{n,j}^{(T,\mu,\tau)}}S_{j,T-t}(\varphi)S_{j,t}(-\theta),
\endeq
wherein
\eq
S_{j,a}(\varphi;T,\mu,\tau,n)\equiv S_{j,a}(\varphi):=\frac{1}{n}\sum_{x\in L_{n,\tau}}p_{n,j}^{(T,\mu,\tau)}(x)e^{-an(x^2-2i\mu x)/2}e^{i\varphi nx},
\endeq
and
\begin{equation} \label{eq:W_ij_exact_formula}
{\mathop{W}^\circ}_{[i,j)}(\varphi,\theta) :=\left\{
  \begin{aligned}
&0 , \quad &t_i\ge t_j, \\
&   \frac{1}{2\pi} \sum_{s \in \lattice} e^{-\frac{(t_j - t_i)ns^2}{2} - in(\theta - \varphi)s}, \quad &t_i<t_j.
  \end{aligned}\right.
\end{equation}
Then $\text{NIBM}_{0\to T}^\mu$ is the determinantal process on $\T$ defined by the multi-time extended kernel \eqref{tau-kernel-ext} with the parameter $\tau$ set to $\tau=0$ if $n$ is odd, and $\tau=1/2$ if $n$ is even. That is, if we fix $m$ times $0<t_1\le t_2\le \dots \le t_m<T$, then the $m$-point correlation function for the positions of the particles at times $t_1, \dots, t_m$ is given by 
\begin{equation} \label{eq:defn_Rmelon_special}
  R_{0\to T}^{(n)}(\theta_1, t_1;  \dotsc; \theta_m, t_m ; \mu) = \det \left( K_{t_i, t_j}(\theta_i,\theta_j)\bigg|_{\tau=\hsgn{n}}  \right)_{i, j = 1}^m\,,
\end{equation}
where
\begin{equation}\label{eq:def_hsgn}
  \hsgn{n} :=
  \begin{cases}
    0 & \text{if $n$ is odd}, \\
    \frac{1}{2} & \text{if $n$ is even}.
  \end{cases}
\end{equation}
We note that \eqref{tau-kernel} is not a correlation kernel unless 
$\tau=\epsilon(n)$.  The Fourier-type parameter $\tau$ is used for computing 
winding numbers, as we will see below.  

To analyze the $k$-tacnode process, we consider the process $\text{NIBM}_{0\to T}^\mu$ as the number of particles $n$ approaches infinity when the total time $T$ is close the critical time $T_c=\pi^2$, and the drift $\mu$ is of order $\bigO((\log n)/n)$. More specifically, we define the {\it $k$-tacnode regime} to be the scaling regime such that
\begin{itemize}
\item As $n\to\infty$, $(\pi^2-T)n^{2/3}$ is bounded.
\item For some fixed non-negative integer $k$, $\mu$ satisfies
\eq
\label{mu-range}
\left(k-\frac{1}{2}\right)\frac{\log n}{3\pi n} < \mu \leq \left(k+\frac{1}{2}\right)\frac{\log n}{3\pi n}.
\endeq
\end{itemize}
 In order to state our results, we define a parameter $s$ in terms of $n$ and $T$ as
\eq\label{def-s}
s \equiv s(T, n) := -\left[(3\pi n)\left(z_1 -i\mu- \int_{i\mu}^{z_1} \rho(w)\,dw\right)\right]^{2/3}
\endeq
where 
\eq\label{def-rho}
\rho(w):=\frac{T}{2\pi} \sqrt{\frac{4}{T} - (w-i\mu)^2}, \quad z_1 := i\mu + \frac{2}{T}\sqrt{T-\pi^2}.
\endeq
Here $\rho(w)$ is positive on the interval $(-\frac{2}{\sqrt{T}}+i\mu, \frac{2}{\sqrt{T}}+i\mu)$ and extends analytically into a neighborhood of compact subsets of that interval. One can check that as $T\to \pi^2$,
\eq\label{T-scaling}
s=\frac{2^{2/3}n^{2/3}}{\pi^2}\left((\pi^2 - T)+\frac{4}{5\pi^2}(\pi^2-T)^2+\bigO((\pi^2 - T)^3)\right),
\endeq
so $s$ is bounded as $n\to\infty$ in the $k$-tacnode scaling.
We choose to write \eqref{def-s} in a form that involves $\mu$ because the density $\rho(w)$ on the interval $(-\frac{2}{\sqrt{T}}+i\mu, \frac{2}{\sqrt{T}}+i\mu)$ naturally appears in our analysis, but it is not hard to see that $s$ does not depend on  $\mu$. In fact, $s$ is the same as was defined in \cite[Equation (1.43)]{Liechty:2012}. 

\subsection{Results on winding numbers}
We first review the known results in the tacnode regime in the zero-drift 
case $\mu=0$ \cite{LiechtyW:2016} and then state our results in the 
$k$-tacnode regime.  Let $u_{_\textnormal{HM}}^{(0)}(s)$ be the 
\emph{Hastings--McLeod function}, that is, the unique solution to the 
homogeneous Painlev\'e-II equation 
\eq
\label{PII-hom}
u''(s) = 2u(s)^3 + su(s)
\endeq
with asymptotics
\eq
\label{uhm-gen-plus-inf}
u_{_\textnormal{HM}}^{(0)}(s) = \frac{1}{2\sqrt{\pi}}s^{-1/4}e^{-\frac{2}{3}s^{3/2}}\left(1+\mathcal{O}\left(\frac{1}{s^{3/4}}\right)\right) \text{ as } s\to+\infty.
\endeq
One can also write $u_{_\textnormal{HM}}^{(0)}(s)\sim\text{Ai}(s)$ as 
$s\to+\infty$, where $\text{Ai}(s)$ is the Airy function.  
Let $\mathcal{W}_n(T,\mu)$ be the random variable counting the total winding 
number of particles in the process $\text{NIBM}_{0\to T}^\mu$.  The first 
result from the literature describes the distribution of winding numbers 
in the zero-drift case.
\begin{prop}{\rm(Winding numbers in the zero-drift tacnode regime
\cite[Theorem 1.2(b)]{LiechtyW:2016})}.  Set $\mu=0$.  Suppose $(\pi^2-T)n^{2/3}$ remains bounded as $n\to\infty$, and define $s$ by \eqref{def-s}.  Then, as $n\to\infty$,
\eq
\label{zero-drift-winding}
\Prob(\mathcal{W}_n(T,0) = \om)  = \begin{cases} \vspace{.1in} 1 \displaystyle - \frac{u_{_\textnormal{HM}}^{(0)}(s)}{(2n)^{1/3}} + \mathcal{O}\left(\frac{1}{n^{2/3}}\right), & \omega=0, \\ \vspace{.1in} \displaystyle \frac{1}{2}\frac{u_{_\textnormal{HM}}^{(0)}(s)}{(2n)^{1/3}} + \mathcal{O}\left(\frac{1}{n^{2/3}}\right), & \omega=\pm 1, \\ \displaystyle \mathcal{O}\left(\frac{1}{n^{2/3}}\right), & \text{otherwise}. \end{cases}
\endeq
\end{prop}
Our results for the winding numbers in the $k$-tacnode regime are expressed 
in terms of certain functions $\mathcal{U}_k$ and $\mathcal{V}_k$ that solve 
a coupled Painlev\'e-II system and whose logarithmic derivatives are 
\emph{generalized Hastings--McLeod functions}, which we now define.  
For any $\alpha>-\frac{1}{2}$, the inhomogeneous Painlev\'e-II equation
\eq
\label{PII}
u''(s) = 2u(s)^3 + su(s) -\alpha
\endeq
has a unique solution \cite{ItsK:2003,FokasIKN:2006,ClaeysKV:2008}, denoted 
$u_{_\textnormal{HM}}^{(\alpha)}$ and called the generalized Hastings--McLeod 
solution, satisfying both
\eq
\label{uhm-minus-inf}
u_{_\textnormal{HM}}^{(\alpha)}(s) = \sqrt{\frac{-s}{2}}\left(1+\mathcal{O}\left(\frac{1}{(-s)^{3/2}}\right)\right) \text{ as } s\to-\infty,
\endeq
and
\eq
\label{uhm-plus-inf}
u_{_\textnormal{HM}}^{(\alpha)}(s) = \frac{\alpha}{s}\left(1+\mathcal{O}\left(\frac{1}{s^3}\right)\right) \text{ as } s\to+\infty \quad (\alpha\neq 0).
\endeq
For $\alpha=0$ the asymptotic condition \eqref{uhm-gen-plus-inf} is used 
instead of \eqref{uhm-plus-inf}, giving the standard Hastings--McLeod 
function $u_{_\textnormal{HM}}^{(0)}$.

Now let $\mathcal{U}_k$ and $\mathcal{V}_k$ be the solutions of the 
coupled Painlev\'e-II system
\eq
\label{coupled-PII}
\begin{split}
\mathcal{U}_k''(s) & = 2\mathcal{U}_k(s)^2\mathcal{V}_k(s) + s\mathcal{U}_k(s), \\
\mathcal{V}_k''(s) & = 2\mathcal{U}_k(s)\mathcal{V}_k(s)^2 + s\mathcal{V}_k(s), \\
\end{split}
\endeq
with asymptotic behavior 
\eq
\label{Uk-Vk}
\begin{split}
\mathcal{U}_k(s) & = \begin{cases} \displaystyle \frac{-ik!}{2\cdot 8^k\sqrt{\pi}s^{(2k+1)/4}}e^{-\frac{2}{3}s^{3/2}}\left(1+\mathcal{O}\left(\frac{1}{s^{3/4}}\right)\right), & s\to+\infty, \\ \displaystyle -\frac{i(-s)^{(2k+1)/2}}{8^k\sqrt{2}}\left(1+\mathcal{O}\left(\frac{1}{(-s)^{3/2}}\right)\right), & s\to-\infty, \end{cases} \\
\mathcal{V}_k(s) & = \begin{cases} \displaystyle \frac{8^k\sqrt{\pi}is^{(2k-1)/4}}{(k-1)!}e^{\frac{2}{3}s^{3/2}}\left(1+\mathcal{O}\left(\frac{1}{s^{3/4}}\right)\right), & s\to+\infty, \\ \displaystyle \frac{2^{3k-1}\sqrt{2}i}{(-s)^{(2k-1)/2}}\left(1+\mathcal{O}\left(\frac{1}{(-s)^{3/2}}\right)\right), & s\to-\infty.
\end{cases}
\end{split}
\endeq
When $k=0$, these functions are simply multiples of the Hastings--McLeod 
function for $\alpha=0$:
\eq
\label{U0-V0-ito-HM}
\mathcal{U}_0(s) = -iu_{_\textnormal{HM}}^{(0)}(s), \quad \mathcal{V}_0(s) = iu_{_\textnormal{HM}}^{(0)}(s).
\endeq
For $k>0$, if we define 
\eq
\label{log-derivatives}
p_k(s):=\frac{\mathcal{U}_k'(s)}{\mathcal{U}_k(s)}, \quad q_k(s):=\frac{\mathcal{V}_k'(s)}{\mathcal{V}_k(s)},
\endeq
and
\eq
\label{scaled-P-Q}
P_k(s):=2^{-1/3}p_k(-2^{-1/3}s), \quad Q_k(x):=2^{-1/3}q_k(-2^{-1/3}s),
\endeq
then
\eq
\label{Pk-Qk-ito-hm}
P_k(x) \equiv -u_{_\textnormal{HM}}^{(k+\frac{1}{2})}(x), \quad   Q_k(x) \equiv u_{_\textnormal{HM}}^{(k-\frac{1}{2})}(x).
\endeq
Proofs of Equations \eqref{U0-V0-ito-HM} and \eqref{Pk-Qk-ito-hm} are given in \S\ref{Id-Uk-Vk}.

\begin{theo}[Winding numbers in the $k$-tacnode regime]
\label{thm-tacnode-winding}
Fix a non-negative integer $k$ and choose $\mu$ according to 
\eqref{mu-range}.  Also suppose $(\pi^2-T)n^{2/3}$ remains bounded as 
$n\to\infty$, and define $s$ by \eqref{def-s}.  Then 
\eq
\label{tacnode-winding}
\Prob(\mathcal{W}_n(T,\mu) = \om) = \begin{cases} 
\displaystyle \frac{F_\mathcal{V}}{1+F_\mathcal{U}+F_\mathcal{V}} + \mathcal{O}\left(\frac{1}{n^{2/3}}\right), & \omega=k-1, \\
\displaystyle \frac{1}{1 + F_\mathcal{U} + F_\mathcal{V}} + \mathcal{O}\left(\frac{1}{n^{2/3}}\right), & \omega=k, \\ 
\displaystyle \frac{F_\mathcal{U}}{1 + F_\mathcal{U} + F_\mathcal{V}} + \mathcal{O}\left(\frac{1}{n^{2/3}}\right), & \omega=k+1, \vspace{.1in} \\ 
\displaystyle \mathcal{O}\left(\frac{1}{n^{2/3}}\right), & \text{otherwise}, \end{cases}
\endeq
where
\eq
\label{FU-FV-def}
\begin{split}
F_\mathcal{U}\equiv F_\mathcal{U}(s,\mu,n) & := \frac{ie^{2n\pi\mu}}{2(2n)^{(2k+1)/3}}\mathcal{U}_k(s), \\
F_\mathcal{V}\equiv F_\mathcal{V}(s,\mu,n) & := \frac{-i(2n)^{(2k-1)/3}}{2e^{2n\pi\mu}}\mathcal{V}_k(s).
\end{split}
\endeq
Here $F_\mathcal{U}$ and $F_\mathcal{V}$ are non-negative real quantities.
\end{theo}
Theorem \ref{thm-tacnode-winding} is proven in 
\S\ref{subsec-winding-thm-proof} (assuming Lemma \ref{lemma-cnnnm1}, which is 
proven in \S\ref{subsec-lemma-proofs}).
Note that when $\mu=0$ (implying $k=0$), \eqref{tacnode-winding} reduces to \eqref{zero-drift-winding}
using \eqref{U0-V0-ito-HM}. When $\mu$ is in the interior of the interval \eqref{mu-range}, both $F_\mathcal{U}$ and $F_\mathcal{V}$ are $o(1)$. These terms are smallest when $\mu = \frac{k\log n}{3\pi n}$, i.e, $\mu$ is in the center of the interval, at which point they are $\bigO(n^{-1/3})$. They are largest when $\mu$ is at the endpoints of the interval, at which point they are $\bigO(1)$. 

The generalized Hastings--McLeod functions $u_{_\textnormal{HM}}^{(\alpha)}(s)$
have previously appeared in at least three different random matrix and 
nonintersecting particle 
problems.  Claeys, Kuijlaars, and Vanlessen \cite{ClaeysKV:2008} showed these 
functions arise in the double-scaling limit of the eigenvalue correlation 
kernel near the origin for $n\times n$ unitary random matrix ensembles of the 
form 
$$Z_{n,N}^{-1}|\det M|^{2\alpha}e^{-N\text{Tr} V(M)}dM, \quad \alpha>-\frac{1}{2}.$$
Here $V(x)$ is the potential and $Z_{n,N}$ is the partition function.  They study an associated $2\times 2$ 
Riemann--Hilbert problem with jumps on four rays and a pole of order $\alpha$ 
at the origin.  For $\alpha=0$, this Riemann--Hilbert problem agrees with 
Riemann--Hilbert Problem \ref{rhp-jm} below with $k=0$, a fact we will use to 
identify the functions $\mathcal{U}_k$ and $\mathcal{V}_k$.  
A different random matrix setting in which 
$u_{_\textnormal{HM}}^{(\alpha)}(s)$ appears is the chiral two-matrix model 
for $n\times(n+\alpha)$ rectangular matrices $\Phi_1$ and $\Phi_2$ with 
distribution 
$$Z_n^{-1}\exp(-n\text{Tr}(V(\Phi_1^*\Phi_1)+W(\Phi_2^*\Phi_2)-\tau(\Phi_1^*\Phi_2+\Phi_2^*\Phi_1)))d\Phi_1 d\Phi_2, \quad \tau\in\mathbb{R}.$$
Here $V$ and $W$ are polynomial potentials, $Z_n$ is the partition 
function, and $\tau$ is a coupling constant.  In the case where $V$ is linear 
and $W$ is quadratic, Delvaux, Geudens, and Zhang \cite{DelvauxGZ:2013} 
analyzed the problem using a $4\times 4$ Riemann--Hilbert problem with jumps 
on ten rays and a pole of order $\alpha$ at the origin whose solution can be 
expressed in terms of $u_{_\textnormal{HM}}^{(\alpha)}(s)$.  This same
Riemann--Hilbert problem was previously derived by Delvaux 
\cite{Delvaux:2013b} in the study of the so-called hard-edge tacnode.  Here 
$n$ nonintersecting squared Bessel paths start and end at the same point, 
chosen so the paths just osculate against the hard edge.  The squared Bessel 
process is a $\realR_+$-valued stochastic process with a transition probability density defined in terms of Bessel functions. It depends on a parameter $\alpha>-1$, which is the parameter of the Bessel function (in the special case $\alpha=\frac{d}{2} - 1$ for a positive integer $d$, the squared Bessel process behaves like the square of the distance to the origin of a $d$-dimensional Brownian motion). When a group of nonintersecting squared Bessel paths is tuned to approach and then separate from the hard edge at 0, the hard-edge tacnode process appears and is given in terms of the generalized Hastings--McLeod 
functions $u_{_\textnormal{HM}}^{(\alpha+\frac{1}{2})}(s)$.  For certain Bessel parameters, the 
hard-edge tacnode kernel was related to the even and odd parts of the 
(standard) tacnode kernel by Liechty and Wang \cite{LiechtyW:2017}.  Note 
that all of these processes are different from the $k$-tacnode process we 
study here, as they involve the generalized Hastings--McLeod functions 
directly as opposed to the tau functions $\mathcal{U}_k(s)$ and 
$\mathcal{V}_k(s)$.  To the best of our knowledge, this is the first time the 
tau functions for the generalized Hastings--McLeod functions have appeared in 
the literature.  Interestingly, it is the tau functions for the rational 
solutions to the inhomogeneous Painlev\'e-II equation that arise in the 
analysis of a librational-rotational transition for the semiclassical 
sine-Gordon equation \cite{BuckinghamM:2012}.  

\subsection{Results on the correlation kernel}
Next, we define the zero-drift tacnode kernel and describe results from the 
literature on the local convergence of the kernel \eqref{tau-kernel} with 
$\mu=0$ before giving our results for nonzero drift.  Let 
${\bf \Psi}(\zeta;s)$ be the $2\times 2$ matrix-valued function satisfying 
the differential equation 
\eq
\label{zero-drift-lax-eqn}
\frac{d}{d\zeta}{\bf \Psi}(\zeta;s) = \bbm -4i\zeta^2-i(s+2u_{_\textnormal{HM}}^{(0)}(s)^2) & 4\zeta u_{_\textnormal{HM}}^{(0)}(s)+2iu_{_\textnormal{HM}}^{(0)\,\prime}(s) \\ 4\zeta u_{_\textnormal{HM}}^{(0)}(s)-2iu_{_\textnormal{HM}}^{(0)\,\prime}(s) & 4i\zeta^2 + i(s+2u_{_\textnormal{HM}}^{(0)}(s)^2) \ebm {\bf \Psi}(\zeta;s)
\endeq
and the asymptotic condition 
\eq
\label{zero-drift-psi-asymptotics}
{\bf\Psi}(\zeta;s)e^{i(\frac{4}{3}\zeta^3+s\zeta)\sigma_3} = \mathbb{I} + \mathcal{O}(\zeta^{-1}), \quad \zeta\to\pm\infty
\endeq
(here the Pauli matrix $\sigma_3$ is defined in \eqref{pauli-matrices}).  
Equation \eqref{zero-drift-lax-eqn} is part of the Flaschka--Newell Lax pair 
for the homogeneous Painlev\'e-II equation (specialized to the 
Hastings--McLeod solution).  Then define 
\eq
f_0(u;s):=\begin{cases} -[{\bf\Psi}(u;s)]_{12}, & \Im u>0, \\ [{\bf\Psi}(u;s)]_{11}, & \Im u <0,\end{cases} 
\qquad 
g_0(u;s):=\begin{cases} -[{\bf\Psi}(u;s)]_{22}, & \Im u>0, \\ [{\bf\Psi}(u;s)]_{21}, & \Im u <0,\end{cases}
\endeq
and
\eq\label{kernel-non-ess}
\phi_{\tau_i,\tau_j}(\xi,\eta):=\begin{cases} 0, & \tau_i\ge \tau_j, \\ \displaystyle \frac{1}{\sqrt{2\pi(\tau_j-\tau_i)}}e^{-(\xi-\eta)^2/(2(\tau_j-\tau_i))}, & \tau_i<\tau_j. \end{cases}
\endeq
Also define $\Sigma_T$ to be the oriented contour consisting of two 
unbounded components, the first composed of three straight segments from 
$e^{i\pi/6}\cdot\infty$ to $\sqrt{3}+i$ to $-\sqrt{3}+i$ to 
$e^{5\pi i/6}\cdot\infty$, and the second composed of three straight 
segments from $e^{-5\pi i/6}\cdot\infty$ to $-\sqrt{3}-i$ to $\sqrt{3}-i$ to 
$e^{-i\pi/6}\cdot\infty$.  Then set
\eq
\widetilde{K}_{\tau_i,\tau_j}^{(0)}(\xi,\eta;s):=\frac{1}{2\pi}\oint_{\Sigma_T}du\oint_{\Sigma_T}dv\,e^{\tau_iu^2/2-\tau_jv^2/2}e^{-i(u\xi-v\eta)}\frac{f_0(u;s)g_0(v;s)-g_0(u;s)f_0(v;s)}{2\pi i(u-v)}
\endeq
and define the (zero-drift) tacnode kernel as 
\eq
K_{\tau_i,\tau_j}^{(0)}(\xi,\eta;s) := \widetilde{K}_{\tau_i,\tau_j}^{(0)}(\xi,\eta;s) - \phi_{\tau_i,\tau_j}(\xi,\eta).
\endeq
The following result from the literature states that the tacnode kernel is the 
limiting behavior of the kernel \eqref{tau-kernel} in the tacnode regime.
\begin{prop}{\rm(The correlation kernel in the zero-drift tacnode regime
\cite[Theorem 1.3(b)]{LiechtyW:2016})}.  Set $\mu=0$.  Suppose $(\pi^2-T)n^{2/3}$ remains bounded as $n\to\infty$, and define $s$ by \eqref{def-s}.  Scale 
\eq\label{tacnode-scaling}
t_i=\frac{T}{2} + \frac{2^{-10/3}\pi^2}{n^{1/3}}\tau_i, \quad t_j=\frac{T}{2} + \frac{2^{-10/3}\pi^2}{n^{1/3}}\tau_j, \quad \varphi=-\pi-\frac{2^{-5/3}\pi}{n^{2/3}}\xi, \quad \theta=-\pi-\frac{2^{-5/3}\pi}{n^{2/3}}\eta.
\endeq
Then
\eq
\label{zero-drift-kernel}
\lim_{n\to\infty}K_{t_i,t_j}(\varphi,\theta)\left|\frac{dy}{d\eta}\right| = K_{\tau_i,\tau_j}^{(0)}(\xi,\eta;s).
\endeq
\end{prop}

We now present our results for the correlation kernel in the $k$-tacnode 
regime. 
 Let $\widetilde{\bf L}_k(\zeta;s)$ be the $2\times 2$ matrix-valued function satisfying the differential equation 
\eq
\frac{\partial}{\partial\zeta}\widetilde{\bf L}_k(\zeta;s)  = \bbm -4i\zeta^2 - i(s+2\mathcal{U}_k(s)\mathcal{V}_k(s)) & 4i\zeta\mathcal{U}_k(s)-2\mathcal{U}_k'(s) \\ -4i\zeta\mathcal{V}_k(s) - 2\mathcal{V}_k'(s) & 4i\zeta^2 + i(s+2\mathcal{U}_k(s)\mathcal{V}_k(s)) \ebm \widetilde{\bf L}_k(\zeta;s),
\endeq
where $\mathcal{U}_k(s)$ and $\mathcal{V}_k(s)$ are defined in \eqref{coupled-PII} and \eqref{Uk-Vk}, with the boundary condition
\eq
\widetilde{\bf L}_k(\zeta;s)\zeta^{-k\sigma_3}e^{i(\frac{4}{3}\zeta^3+s\zeta)\sg_3} = \mathbb{I} + \mathcal{O}\left(\frac{1}{\zeta}\right) \qquad \textrm{as} \ \zeta\to\pm\infty;
\endeq
compare to \eqref{zero-drift-lax-eqn} and 
\eqref{zero-drift-psi-asymptotics}.  
%
Then define the functions 
\eq
\label{fk-gk-def}
f_k(u;s):=\begin{cases} -[\widetilde{\bf L}_k(u;s)]_{12}, & \Im u>0, \\ [\widetilde{\bf L}_k(u;s)]_{11}, & \Im u <0,\end{cases} 
\qquad 
g_k(u;s):=\begin{cases} -[\widetilde{\bf L}_k(u;s)]_{22}, & \Im u>0, \\ [\widetilde{\bf L}_k(u;s)]_{21}, & \Im u <0.\end{cases}
\endeq
Set
\eq
\label{kernel-Kktilde}
\widetilde{K}_{\tau_i,\tau_j}^{(k)}(\xi,\eta;s) :=
\frac{1}{2\pi} \int_{\Sg_T} du \int_{\Sg_T } dv\, e^{(\frac{\tau_i}{2} u^2 - \frac{\tau_j}{2} v^2) - i(\xi u - \eta v)}\frac{f_k(u; s)g_k(v;s) - g_k(u; s)f_k(v; s)}{2\pi i (u - v)}
\endeq
and define the $k$-tacnode kernel to be
\eq
K_{\tau_i,\tau_j}^{(k)}(\xi,\eta;s) := \widetilde{K}_{\tau_i,\tau_j}^{(k)}(\xi,\eta;s) - \phi_{\tau_i,\tau_j}(\xi,\eta),
\endeq
where $\phi_{\tau_i,\tau_j}(\xi,\eta)$ is defined in \eqref{kernel-non-ess}.
\begin{theo}[The correlation kernel in the $k$-tacnode regime]
\label{thm:k-tacnode-kernel}
Fix a non-negative integer $k$ and let $\mu=\frac{k\log n}{3\pi n} $.  Also 
suppose $(\pi^2-T)n^{2/3}$ remains bounded as $n\to\infty$, and define $s$ by 
\eqref{def-s}.  Scale $t_i$, $t_j$, $\varphi$, and $\theta$ as in 
\eqref{tacnode-scaling}.
Then
\eq
\label{k-tacnode-kernel}
\lim_{n\to\infty}K_{t_i,t_j}(\varphi,\theta)\left|\frac{dy}{d\eta}\right| = K_{\tau_i,\tau_j}^{(k)}(\xi,\eta;s).
\endeq
\end{theo}

Theorem \ref{thm:k-tacnode-kernel} is proven in \S\ref{subsec-kernel-proof} 
(assuming Lemmas \ref{lemma:asymptotics_of_OPs_critical} and \ref{lemma:asymptotics_of_OPs_origin}, which are proven in \S\ref{subsec-lemma-proofs}).  
Note that in Theorem \ref{thm:k-tacnode-kernel} we have chosen $\mu$ to be exactly at the midpoint of the interval \eqref{mu-range}. In fact the theorem holds with slower convergence for any $\mu$ in the interior of this interval, provided it does not approach the endpoints. For ease of the analysis we choose the $\mu$ that gives the best convergence.

\subsection{Outline and notation}  We begin in \S\ref{sec-proofs} with the 
proofs of Theorems \ref{thm-tacnode-winding} and \ref{thm:k-tacnode-kernel}, 
assuming three lemmas describing the asymptotic behavior of the discrete 
Gaussian orthogonal polynomials with a complex weight.  The remainder of the 
paper is dedicated to proving these lemmas.  In \S\ref{sec-rhp} we recall the 
meromorphic Riemann--Hilbert problem associated to the discrete orthogonal 
polynomials derived in \cite{BuckinghamL:2017} and carry out a series of 
transformations in order to arrive at a controllable problem.   The resulting 
(sectionally analytic) Riemann--Hilbert problem has jumps that are close to 
the identity matrix everywhere except in a neighborhood of a certain line 
segment (the \emph{band}).  Discarding the jumps except on the band leads to 
the \emph{outer model problem}, which is constructed in 
\S\ref{sec-parametrices} in a nonstandard way due to a pole of order $k$ at 
the point $i\mu$ on the band.  In \S\ref{sec-parametrices} we also 
construct \emph{parametrices} near the point $i\mu$ and the band endpoints.  
The parametrix near $i\mu$ is constructed in terms of the Painlev\'e-II tau 
functions $\mathcal{U}_k$ and $\mathcal{V}_k$, and for certain values of 
$\mu$ contributes to the leading-order behavior of the solution to the full 
Riemann--Hilbert problem.  This contribution is captured through the 
construction of a \emph{parametrix for the error} that is carried out in 
\S\ref{sec-error}.  In this section we also complete the error 
analysis and prove the three lemmas on discrete orthogonal polynomials.  

\

\noindent
{\it Notation}.  With the exception of 
\eq
\label{pauli-matrices}
\mathbb{I}:=\bbm 1 & 0 \\ 0 & 1 \ebm, \quad \sigma_1:=\bbm 0 & 1 \\ 1 & 0 \ebm, \quad \sigma_3:=\bbm 1 & 0 \\ 0 & -1 \ebm, \quad {\bf 0}:=\bbm 0 & 0 \\ 0 & 0 \ebm,
\endeq
matrices are denoted by bold capital letters.  
We denote the $(jk)$th entry of a matrix ${\bf M}$ by $[{\bf M}]_{jk}$.
In reference to a smooth, oriented contour $\Sigma$, for $z\in\Sigma$ we denote by 
$f_+(z)$ (respectively, $f_-(z)$) the non-tangential limit of $f(\zeta)$ as 
$\zeta$ approaches $\Sigma$ from the left (respectively, the right).

\section{Proofs of the main theorems assuming results on orthogonal polynomials}
\label{sec-proofs}

In this section we prove Theorems \ref{thm-tacnode-winding} and \ref{thm:k-tacnode-kernel} using asymptotic results for the orthogonal polynomials \eqref{eq:defn_of_discrete_Gaussian_OP}. These asymptotic results are stated here and are proved using the Riemann--Hilbert method in the remaining sections of the paper. 

\subsection{Proof of Theorem \ref{thm-tacnode-winding} (winding numbers)}
\label{subsec-winding-thm-proof}
Define the Hankel determinant
\begin{equation}
\Hankel_n(T,\mu,\tau) := \det \left( \frac{1}{n} \sum_{x \in \lattice} x^{j + m - 2} e^{-\frac{Tn}{2}(x^2-2i\mu x)} dx \right)^n_{j, m = 1}
\end{equation}
(recall $L_{n,\tau}$ from \eqref{lattice-def}).  
Then the total winding number $\mathcal{W}_n(T,\mu)$ has the following probability 
mass function (see \cite[Equation (1.16)]{BuckinghamL:2017} and 
\cite[Equation (185)]{LiechtyW:2016}):
\begin{equation}
\label{eq:total_offset_formula}
\Prob(\mathcal{W}_n(T,\mu)=\om) = e^{2\pi i\omega\hsgn{n}} \int_0^1  \frac{\Hankel_n(T,\mu,\tau)}{\Hankel_n(T,\mu,\hsgn{n})} e^{-2\pi i \om \tau}d\tau,
\end{equation}
where $\epsilon(n)$ was defined in \eqref{eq:def_hsgn}.
We exploit the natural connection between Hankel determinants for discrete 
weights and discrete orthogonal polynomials.  
Define $c^{(T,\mu,\tau)}_{n, m,j}$ as the coefficient of the term 
$x^j$ in $p^{(T,\mu,\tau)}_{n, m}(x)$:
\eq
p^{(T,\mu,\tau)}_{n, m}(x) = x^m + \sum_{j=0}^{m-1} c^{(T,\mu,\tau)}_{n, m,j} x^j.
\endeq
It was shown in \cite[Proposition 4.2]{BuckinghamL:2017} that 
\eq
\label{log-Hn-integral}
\log \left(\frac{\Hankel_n(T,\mu,\tau) }{\Hankel_n(T,\mu,\hsgn{n})}\right) = \int_{\hsgn{n}}^\tau \left(inT\mu + Tc^{(T,\mu,v)}_{n, n,n-1}\right)\,dv.
\endeq
Combining \eqref{eq:total_offset_formula} and \eqref{log-Hn-integral} gives
\eq
\label{winding-probs-ito-cnnnm1}
\Prob(\mathcal{W}_n(T,\mu)=\om) = e^{2\pi i\omega\hsgn{n}} \int_0^1  
\exp\left(\int_{\hsgn{n}}^\tau \left(inT\mu + T c_{n,n,n-1}^{(T,\mu,v)}
\right)\,dv\right)
e^{-2\pi i \om \tau}d\tau.
\endeq
Define 
\eq
\label{RV-RU-def}
R_\mathcal{V}\equiv R_\mathcal{V}(s, \mu, n):=\frac{(-1)^n(2n)^{(2k-1)/3}}{\pi e^{2n\pi\mu}}\mathcal{V}_k(s), \quad R_\mathcal{U}\equiv R_\mathcal{U}(s, \mu, n):=\frac{(-1)^n e^{2n\pi\mu}}{\pi (2n)^{(2k+1)/3}}\mathcal{U}_k(s).
\endeq
The following lemma gives the expansion of $c_{n,n,n-1}^{(T,\mu,\tau)}$ 
we need.
\begin{lem}
\label{lemma-cnnnm1}
Fix a non-negative integer $k$ and choose $\mu$ by \eqref{mu-range}.  Also 
suppose $(\pi^2-T)n^{2/3}$ remains bounded as $n\to\infty$, and define $s$ by 
\eqref{def-s}.  Then 
\eq
c_{n,n,n-1}^{(T,\mu,\tau)} = \frac{2ik}{\pi} + \frac{R_\mathcal{V}e^{-2i\pi\tau}}{1-\frac{\pi}{2i}R_\mathcal{V}e^{-2\pi i\tau}} + \frac{R_\mathcal{U}e^{2i\pi\tau}}{1+\frac{\pi}{2i}R_\mathcal{U}e^{2\pi i\tau}} - in\mu + \mathcal{O}\left(\frac{1}{n^{2/3}}\right),
\endeq
where the error term is uniformly bounded in $\tau$.
\end{lem}

Lemma \ref{lemma-cnnnm1} is proven in \S\ref{subsec-lemma-proofs}.  
Assuming this, we are now ready to prove our main result on winding numbers.  

\begin{proof}[Proof of Theorem \ref{thm-tacnode-winding}]
Combining \eqref{log-Hn-integral}, Lemma \ref{lemma-cnnnm1}, and 
\eqref{T-scaling} (and using the uniform boundedness of the error term to 
integrate with respect to $\tau$) shows 
\eq
\begin{split}
\log &\left(\frac{\Hankel_n(T,\mu,\tau) }{\Hankel_n(T,\mu,\hsgn{n})}\right) = \int_{\hsgn{n}}^\tau \left(inT\mu + T c_{n,n,n-1}^{(T,\mu,\tau)} \right)\,dv \\ 
& \hspace{.2in} = \int_{\hsgn{n}}^\tau 2ik\pi + \frac{\pi^2 R_\mathcal{V}}{1-\frac{\pi}{2i}R_\mathcal{V}e^{-2\pi i\tau}} e^{-2\pi iv} + \frac{\pi^2R_\mathcal{U}}{1+\frac{\pi}{2i}R_\mathcal{U}e^{2\pi i\tau}} e^{2\pi iv} \,dv + \mathcal{O}\left(\frac{1}{n^{2/3}}\right) \\
& \hspace{.2in} = 2ik\pi(\tau-\hsgn{n}) + 
\log\left(1-\frac{\pi}{2i}R_\mathcal{V}e^{-2\pi i\tau}\right) - \log\left(1-\frac{\pi}{2i}R_\mathcal{V}e^{-2\pi i\epsilon(n)}\right) 
\\ 
& \hspace{.45in} + 
\log\left(1+\frac{\pi}{2i}R_\mathcal{U}e^{2\pi i\tau}\right) - \log\left(1+\frac{\pi}{2i}R_\mathcal{U}e^{2\pi i\epsilon(n)}\right) 
+ \mathcal{O}\left(\frac{1}{n^{2/3}}\right).
\end{split}
\endeq
Taking the exponent of both sides gives
\eq
\frac{\Hankel_n(T,\mu,\tau) }{\Hankel_n(T,\mu,\hsgn{n})} = \frac{(-1)^{-(n+1)k}}{1 + \frac{(-1)^{n+1}\pi}{2i}(R_\mathcal{U}-R_\mathcal{V})}  \left(1 + \frac{\pi}{2i}R_\mathcal{U}e^{2\pi i\tau} - \frac{\pi}{2i}R_\mathcal{V}e^{-2\pi i\tau} \right)e^{2ik\pi\tau} + \mathcal{O}\left(\frac{1}{n^{2/3}}\right).
\endeq
Here we have used $e^{2\pi i\hsgn{n}}=(-1)^{n+1}$ from \eqref{eq:def_hsgn} and 
$R_\mathcal{U}R_\mathcal{V}=\mathcal{O}(n^{-2/3})$ from \eqref{RV-RU-def}.  The last step is to use 
\eqref{eq:total_offset_formula} to determine the winding numbers.
\eq
\begin{split}
& \Prob(\mathcal{W}_n(T,\mu) = \om) \\
& = \frac{(-1)^{(n+1)(\omega-k)}}{1 + \frac{(-1)^{n+1}\pi}{2i}(R_\mathcal{U}-R_\mathcal{V})} \int_0^1 \left(1 + \frac{\pi}{2i}R_\mathcal{U}e^{2\pi i\tau} - \frac{\pi}{2i}R_\mathcal{V}e^{-2\pi i\tau} \right)e^{2\pi i\tau(k-\omega)}d\tau + \mathcal{O}\left(\frac{1}{n^{2/3}}\right) \\
& = \begin{cases} 
\displaystyle \frac{\frac{\pi}{2i}(-1)^nR_\mathcal{V}}{1 + \frac{\pi}{2i}(-1)^{n+1}R_\mathcal{U} + \frac{\pi}{2i}(-1)^nR_\mathcal{V}} + \mathcal{O}\left(\frac{1}{n^{2/3}}\right), & \omega=k-1, \vspace{.05in} \\ 
\displaystyle \frac{1}{1 + \frac{\pi}{2i}(-1)^{n+1}R_\mathcal{U} + \frac{\pi}{2i}(-1)^nR_\mathcal{V}} + \mathcal{O}\left(\frac{1}{n^{2/3}}\right), & \omega=k, \vspace{.05in} \\ 
\displaystyle \frac{\frac{\pi}{2i}(-1)^{n+1}R_\mathcal{U}}{1 + \frac{\pi}{2i}(-1)^{n+1}R_\mathcal{U} + \frac{\pi}{2i}(-1)^nR_\mathcal{V}} + \mathcal{O}\left(\frac{1}{n^{2/3}}\right), & \omega=k+1, \vspace{.1in} \\ 
\displaystyle \mathcal{O}\left(\frac{1}{n^{2/3}}\right), & \text{otherwise}. \end{cases} 
\end{split}
\endeq
Along with the definitions of $F_\mathcal{U}$ and $F_\mathcal{V}$ in 
\eqref{FU-FV-def}, this completes the proof of Theorem 
\ref{thm-tacnode-winding}.
\end{proof}

\subsection{Proof of Theorem \ref{thm:k-tacnode-kernel} (the correlation kernel)}
\label{subsec-kernel-proof}
We will compute the aymptotics of the kernel \eqref{tau-kernel-ext} in the scalings \eqref{T-scaling} and \eqref{tacnode-scaling}.
Our proof closely follows that of \cite[Section 5.3]{LiechtyW:2016}.  
We begin with two results on discrete orthogonal polynomials that are proven 
in \S\ref{subsec-lemma-proofs}.  In order to state these two lemmas we define 
the so-called $g$-function and Lagrange multiplier $\ell$:
\eq\label{eq:def-g-function}
g(z)\equiv g(z;T,\mu):=g_0(z-i\mu;T), \quad \ell \equiv \ell(T,\mu) := \ell_0(T) -\frac{T}{2} \mu^2,
\endeq
where
\eq
\label{eq:LWg-definition}
g_0(z;T):=\frac{T}{4}z\left(z-\sqrt{z^2-\frac{4}{T}}\right)-\log\left(z-\sqrt{z^2-\frac{4}{T}}\right) - \frac{1}{2} + \log\frac{2}{T}, \qquad \ell_0(T) := -1-\log T,
\endeq
with the logarithm and square root indicating principal branches.  More 
details concerning the function $g$ are given in \S\ref{sec-rhp}.

The following lemma is the extension of \cite[Proposition 3.7]{LiechtyW:2016} to include the parameter $\mu$.
\begin{lem} \label{lemma:asymptotics_of_OPs_critical}
Fix a non-negative integer $k$ and let $\mu=\frac{k\log n}{3\pi n}$.  Also 
suppose $(\pi^2-T)n^{2/3}$ remains bounded as $n\to\infty$, and define $s$ by 
\eqref{def-s}.  Then the discrete Gaussian orthogonal polynomials 
\eqref{eq:defn_of_discrete_Gaussian_OP} satisfy the estimates 
\eq
  p^{(T,\mu, \tau)}_{n, n}(z) = e^{ng(z)}     [{\bf M}_k^{\rm (out)}(z)]_{11} (1+\bigO(n^{-1/3})), \quad \frac{p^{(T,\mu, \tau)}_{n, n-1}(z)}{h^{(T,\mu,\tau)}_{n, n-1}} = e^{n(g(z)-\ell)}     [{\bf M}_k^{\rm (out)}(z)]_{21} (1+\bigO(n^{-1/3}))  \label{asub1} \\
\endeq
as $n\to\infty$  in the domain $\{ z \mid \lvert \Im z - \mu \rvert > \ep \}$ for any fixed $\ep>0$. Here the function $g(z)$ is defined in \eqref{eq:def-g-function} and the matrix function ${\bf M}_k^{\rm (out)}(z)$ is defined in \eqref{M-k-out}. The errors are uniform in $z$.
    \end{lem}
The following lemma extends \cite[Proposition 3.8]{LiechtyW:2016} to include 
the parameter $\mu$.
 \begin{lem} \label{lemma:asymptotics_of_OPs_origin}
Fix a non-negative integer $k$ and a real number $\delta>0$.  Set 
$\mu=\frac{k\log n}{3\pi n}$.  Also suppose $(\pi^2-T)n^{2/3}$ remains 
bounded as $n\to\infty$, and define $s$ by \eqref{def-s}.  Then there exists 
$\ep>0$ such that for all $z,w\in \{z \in \C \lvert \ |z|<\ep n^{-\delta}\}$ 
the following asymptotic formula holds:
  \begin{multline}\label{critical_CD_kernel}
      e^{-\frac{nT}{4}(z^2-2i\mu z+w^2-2i\mu w)} \frac{p^{(T,\mu, \tau)}_{n, n}(z)p^{(T,\mu, \tau)}_{n, n-1}(w)-p^{(T\mu, \tau)}_{n, n-1}(z)p^{(T,\mu, \tau)}_{n, n}(w)}{h^{(T,\mu, \tau)}_{n, n-1}(z-w)} =  \frac{1}{2\pi i (z-w)}  \\
    \times
      \begin{bmatrix}
        -e^{-i\pi(nz-\tau)} \\
        e^{i\pi(nz-\tau)}
      \end{bmatrix}^\mathsf{T}
      \widetilde{\bf L}_k(d n^{\frac{1}{3}} z ; s)^{-1} \widetilde{\bf L}_k(d n^{\frac{1}{3}} w ; s)
      \begin{bmatrix}
        e^{i\pi(nw-\tau)} \\
        e^{-i\pi(nw-\tau)}
      \end{bmatrix}
      \left(1+\bigO(n^{-1/3})+\bigO(n^{1/3-2\delta})+\bigO(n^{-\delta})\right),
  \end{multline}
  where $d = 2^{-5/3}\pi$. Also, the following estimates hold uniformly in $\{z \in \C \lvert \ |z|<\ep \}$:
  \begin{equation} \label{eq:rough_estimate_critical_zero}
    p^{(T,\mu,\tau)}_{n, n}(z) = \bigO\left(n^{\frac{2k}{3}} e^{ng(z)+rn^{1/3}}\right), \quad \frac{p^{(T,\mu,\tau)}_{n, n - 1}(z)}{h^{(T,\mu,\tau)}_{n, n - 1}} = \bigO\left(n^{\frac{2k}{3}} e^{n(g(z) - \ell)+rn^{1/3}}\right)
  \end{equation}
  for some constant $r \in \realR$.
    \end{lem}

Assuming these two results, we can now prove Theorem 
\ref{thm:k-tacnode-kernel}.  

\begin{proof}[Proof of Theorem \ref{thm:k-tacnode-kernel}]
From \eqref{tau-kernel-ext}, we need to understand 
$\widetilde{K}_{t_i,t_j}(\varphi,\theta)$ and 
$\displaystyle{\mathop{W}^\circ}_{[i,j)}(\varphi,\theta)$.  The second 
function is exactly the one in \cite[Equation (133)]{LiechtyW:2016}, and in 
this scaling it was shown in \cite[Equation (270)]{LiechtyW:2016} that
\eq\label{non-ess-limit}
\lim_{n\to\infty} \frac{\pi}{2^{5/3} n^{2/3}} {\mathop{W}^\circ}_{[i,j)}(\varphi,\theta) = \phi_{\tau_i,\tau_j}(\xi,\eta).
\endeq
We now compute $\widetilde{K}_{t_i, t_j}(\varphi,\theta)$.  
Set 
\eq
d:=2^{-5/3} \pi
\endeq
and define $\Sigma$ to be the contour that in the upper half-plane connects 
$\infty\cdot e^{i0}$ to $\sqrt{3}+i$ to $(\sqrt{3}+i)d^{-1}n^{-1/3}$ to 
$(-\sqrt{3}+i)d^{-1}n^{-1/3}$ to $-\sqrt{3}+i$ to $\infty\cdot e^{i\pi}$ 
(oriented right-to-left), and in the lower half-plane is the reflection of 
the contour in the upper half-plane through the origin 
(oriented left-to-right).  
Our starting point 
for asymptotics is the formula 
\begin{equation} \label{tac8}
\widetilde{K}_{t_i, t_j}(\varphi,\theta)= \frac{n}{2\pi} \int_{\Sg} dz \int_{\Sg+(i/2d)n^{-2/3}} dw J(z, w),
\end{equation}
where
\begin{multline}
\label{J-def}
J(z, w) := \left( e^{-\frac{nT}{4}(z^2-2i\mu z + w^2-2i\mu w)} \frac{p^{(T,\mu,\tau)}_{n, n}(z)p^{(T,\mu,\tau)}_{n, n-1}(w)-p^{(T,\mu,\tau)}_{n, n-1}(z)p^{(T,\mu,\tau)}_{n, n}(w)}{h_{n,n-1}^{(T,\mu,\tau)}\cdot(z-w)}\right) \\   \times  e^{n^{\frac{2}{3}} \frac{d^2}{2} (\tau_i (z^2-2i\mu z) - \tau_j (w^2-2i\mu w))} e^{-in^{\frac{1}{3}} d (\xi z - \eta w)} 
\frac{e^{\pi i(nz-\tau)} e^{\pi i(nw-\tau)}}{(e^{2\pi i(nz-\tau)}-1)(e^{2\pi i(nw-\tau)}-1)}.
\end{multline}
Equations \eqref{tac8} and \eqref{J-def} are the same as  
\cite[Equations (258) and (259)]{LiechtyW:2016} when $\mu=0$. 
We denote by $\Sigma_{\rm local}$, $\Sigma^{\rm upper}_{\rm local}$, and 
$\Sigma^{\rm lower}_{\rm local}$ the contours
\begin{equation}
\Sigma_{\rm local} := \Sigma \cap \{ z \in \C \mid \lvert z \rvert < n^{-\frac{1}{4}} \}, \quad \Sigma^{\rm upper}_{\rm local} := \Sigma_{\rm local} \cap \C_+, \quad \Sigma^{\rm lower}_{\rm local} := \Sigma_{\rm local} \cap \C_-.
\end{equation}
We claim that for large $n$ the integral \eqref{tac8} becomes localized on $\Sg_{\rm local}$. To justify this claim, we apply Lemma \ref{lemma:asymptotics_of_OPs_critical} and Equation \eqref{eq:rough_estimate_critical_zero} to the function $J(z,w)$ to obtain the estimate that for all $z\in \Sg\setminus \Sg_{\rm local}$ and $w\in  \{\Sg+(i/2d)n^{-2/3}\}\setminus \{\Sg_{\rm local}+(i/2d)n^{-2/3}\}$,
\begin{multline}\label{J_est}
  \lvert J(z, w) \rvert =  \bigO(e^{n(\widetilde{g}(z) - \frac{T}{4} (z^2-2i\mu z) + \pi iz + \widetilde{g}(w) - \frac{T}{4} (w^2-2i\mu w) + \pi iw-\ell)}) \\ \times \bigO(n^{\frac{4k}{3}}e^{2rn^{1/3}})e^{n^{\frac{2}{3}} d^2 (\tau_i (z^2-2i\mu z) - \tau_j (w^2-2i\mu w))} e^{-in^{\frac{1}{3}} d (\xi z - \eta w)} \bigO(n^{\frac{2}{3}}),
\end{multline}
where $\widetilde{g}$ is defined as
\begin{equation} \label{eq:defn_tilde_g}
  \widetilde{g}(z) :=
  \begin{cases}
    g(z) & \text{if $\Im z > \mu$}, \\
    g(z) + i\pi - 2\pi i z & \text{if $\Im z < \mu$}.
  \end{cases}
\end{equation}
The $\bigO(n^{2/3})$ contribution comes from $(z-w)^{-1}$. By direct calculation, we find that for $z \in \Sigma \setminus \Sigma_{\rm local}$, $\Re(\widetilde{g}(z) - T (z^2-2i\mu z)/4 + \pi iz)$ decreases as $z$ moves away from $0$. Hence, by a standard steepest-descent argument we have 
that
\begin{equation}\label{tac17a}
  \begin{split}
    \widetilde{K}_{t_i, t_j}(\theta,\varphi) = & \frac{n}{2\pi} \int_{\Sg} dz \int_{\Sg + \frac{i}{2} d^{-1} n^{-2/3}} dw J(z, w) \\
    = & \frac{n}{2\pi} \int_{\Sg_{\rm local}} dz \oint_{\Sg_{\rm local} + \frac{i}{2} d^{-1} n^{-2/3}} dw J(z, w) + \bigO(e^{-cn^{\frac{1}{4}}}), \\
  \end{split}
\end{equation}
where $c$ is a positive constant. 

Inserting the estimates \eqref{eq:rough_estimate_critical_zero}  
into the integral \eqref{tac17a} and making the change of variables $u=dn^{1/3} z$ and $v = dn^{1/3} w$, we obtain that 
    \begin{multline} \label{tac10}
  \frac{n}{2\pi} \int_{\Sg_{\rm local}} dz \int_{\Sg_{\rm local} + \frac{i}{2} d^{-1} n^{-2/3}} dw J(z, w) = \frac{n^{\frac{2}{3}}}{4\pi^2 i d} \int_{\Sg_T^*} du \int_{\Sg_T^* + \frac{i}{2}} dv \frac{e^{\frac{1}{2}(\tau_i u^2 - \tau_j v^2) - i(\xi u - \eta v)}}{u - v} \\
  \times
  \begin{bmatrix}
    \frac{1}{1 - e^{2\pi i(nz - \tau)}} \\
    \frac{1}{1 - e^{-2\pi i(nz - \tau)}} 
  \end{bmatrix}^\mathsf{T}
\widetilde{  {\bf L}}_k(u ; s)^{-1} \widetilde{{\bf L}}_k(w ; s)
  \begin{bmatrix}
    \frac{1}{1 - e^{-2\pi i(nw - \tau)}} \\
    \frac{-1}{1 - e^{2\pi i(nw - \tau)}} 
  \end{bmatrix}
  \left(1 + \bigO(n^{-\frac{1}{6}})\right),
\end{multline}
  where $\Sg_T^* := \Sg_T \cap \{z : |z|< n^{1/12}\}$ and $\Sg_T$ is defined following \eqref{kernel-non-ess}. Noting that the factors  $ \frac{1}{1 - e^{\pm 2\pi i(nz - \tau)}} $ and $ \frac{1}{1 - e^{\pm 2\pi i(nz - \tau)}} $ are either $\bigO(e^{-2n^{2/3}/d})$ or $(1+\bigO(e^{-2n^{2/3}/d}))$ depending on if $z$ and/or $w$ are in the upper or lower half-plane (see \cite[Equations (264) and (265)]{LiechtyW:2016}), we find
  \begin{equation} \label{eq:local_est_K_tac}
  \begin{split}
    & \frac{n}{2\pi} \int_{\Sg_{\rm local}} dz \int_{\Sg_{\rm local} + \frac{i}{2} d^{-1} n^{-2/3}} dw J(z, w) \\
    = & \frac{n^{\frac{2}{3}}}{4\pi^2 i d} \int_{\Sg_T^*} du \int_{\Sg_T^* + \frac{i}{2}} dv e^{(\frac{\tau_i}{2} u^2 - \frac{\tau_j}{2} v^2) - i(\xi u - \eta v)}\frac{f_k(u; s)g_k(v;s) - g_k(u; s)f_k(v; s)}{(u - v)} \left(1 + \bigO\left(n^{-\frac{1}{6}}\right)\right), \\
  \end{split}
\end{equation}
where $f_k(u;s)$ and $g_k(u;s)$ are defined in \eqref{fk-gk-def}.
Combining \eqref{eq:local_est_K_tac} and \eqref{non-ess-limit} we obtain
\eq 
\lim_{n\to\infty} \frac{\pi}{2^{5/3} n^{2/3}}K_{t_i, t_j}(\varphi,\theta)  = \widetilde{K}_{\tau_i,\tau_j}^{(k)}(\xi,\eta;s) - \phi_{\tau_i,\tau_j}(\xi,\eta),
\endeq
where $\phi_{\tau_i,\tau_j}(\xi,\eta)$ and 
$\widetilde{K}_{\tau_i,\tau_j}^{(k)}(\xi,\eta;s)$ are as defined in 
\eqref{kernel-non-ess} and \eqref{kernel-Kktilde}, respectively, and the contour $\Sg_T+\frac{i}{2}$ is easily deformed to $\Sg_T$ after taking the limit.
\end{proof}

\section{Setup of the Riemann--Hilbert problem}
\label{sec-rhp}

We obtain our asymptotic results on the discrete Gaussian orthogonal 
polynomials by analyzing the following Riemann--Hilbert problem 
\cite{BuckinghamL:2017}.
\begin{rhp}[Discrete Gaussian orthogonal polynomial problem]
\label{rhp-dop}
Fix $n\in\{1,2,3,\ldots\}$ and $\tau\in[0,1]$ and find a $2\times 2$ matrix-valued 
function $\mathbf P_n(z)$ with the 
following properties:
\begin{itemize}
\item[]{\bf Analyticity:} $\mathbf P_n(z)$ is a meromorphic function of $z$ 
and is analytic for $z\in\C\setminus L_{n,\tau}$.
\item[]{\bf Normalization:}  There exists a function $\textfrak{r}(x)>0$ on  $L_{n,\tau}$ 
such that 
\begin{equation} \label{IP2a}
\lim_{x\to\infty} \textfrak{r}(x)=0
\end{equation} 
and such that, as $z\to\infty$, $\mathbf P_n(z)$ admits the asymptotic 
expansion
\begin{equation} \label{IP2}
\mathbf P_n(z) = \left(\mathbb{I}+\frac{\mathbf P_{n,1}}{z}+\frac{\mathbf P_{n,2}}{z^2}+ \mathcal{O}\left(\frac{1}{z^3}\right) \right)
\begin{bmatrix}
z^n & 0 \\
0 & z^{-n}
\end{bmatrix},\quad z\in \C\setminus \left[\bigcup_{x\in \lattice}^\infty D\big(x,\textfrak{r}(x)\big)\right],
\end{equation}
where ${\bf P}_{n,1}$ and ${\bf P}_{n,2}$ are independent of $z$, and 
$D(x,\textfrak{r}(x))$ denotes a disk of radius $\textfrak{r}(x)>0$ centered at $x$.
\item[]{\bf Residues at poles:}  At each node $x\in L_{n,\tau}$, the elements 
$[\mathbf P_n(z)]_{11}$ and $[\mathbf P_n(z)]_{21}$ of the matrix 
$\mathbf P_n(z)$ 
are analytic functions of $z$, and the elements $[\mathbf P_n(z)]_{12}$ and
$[\mathbf P_n(z)]_{22}$ have a simple pole with the residues
\begin{equation} \label{IP1}
\underset{z=x}{\rm Res}\; [\mathbf P_n(z)]_{j2}=\frac{1}{n} e^{-\frac{nT}{2}(x^2-2i\mu x)} [\mathbf P_n(x)]_{j1},\quad j=1,2.
\end{equation}
\end{itemize}
\end{rhp}
Define the weighted discrete Cauchy transform $C$ as 
\eq
\label{eq:def_Cauchy_trans}
Cf(z):=\frac{1}{n}\sum_{x\in L_{n,\tau}}\frac{f(x)e^{-\frac{nT}{2}(x^2-2i\mu x)}}{z-x}.
\endeq
Then the unique solution to Riemann--Hilbert Problem \ref{rhp-dop} 
(see \cite{FokasIK:1991,BleherL:2011}) is 
\begin{equation} \label{IP3}
\mathbf P_n(z) :=
\begin{bmatrix}
p_{n,n}^{(T,\mu,\tau)}(z) & \left(Cp_{n,n}^{(T,\mu,\tau)}\right)(z) \\
(h^{(T,\mu,\tau)}_{n,n-1})^{-1}p_{n,n-1}^{(T,\mu,\tau)}(z) & (h^{(T,\mu,\tau)}_{n,n-1})^{-1}\left(Cp_{n,n-1}^{(T,\mu,\tau)}\right)(z)
\end{bmatrix}.
\end{equation}
In particular, the coefficient $c_{n,n,n-1}^{(T,\mu,\tau)}$ can be 
calculated via
\begin{equation}\label{IP5a}
c^{(T,\mu,\tau)}_{n, n,n-1} = [\mathbf P_{n,1}]_{11}.
\end{equation}
For subcritical drift values, this Riemann--Hilbert problem was transformed 
in \cite{BuckinghamL:2017} via consecutive changes of variables 
${\bf P}_n\to{\bf R}_n\to{\bf T}_n\to{\bf S}_n$ to a controlled problem 
with either constant or near-identity jumps.  As we will see below, using 
exactly the same changes of variables in the $k$-tacnode regime leads to a 
Riemann--Hilbert problem where the jumps are controlled except in a 
neighborhood of the origin.  The jumps near the origin will be controlled 
by an additional change of variables ${\bf S}_n\to{\bf S}_n^\text{crit}$ and 
the use of a local parametrix.  We begin by combining the 
interpolation of poles, introduction of the $g$-function, 
and opening of lenses into one change of variables ${\bf P}_n\to{\bf S}_n$.  
Define 
\begin{equation}
\mathbf D^u_{\pm}(z) := \begin{bmatrix} 1 & \displaystyle \frac{-\pi}{\sin(n\pi z-\tau\pi)}e^{-nT(z^2-2i\mu z)/2}e^{\pm i\pi(n  z-\tau)} \\ 0 & 1 \end{bmatrix} \quad \text{and} \quad {\bf A}:=\begin{bmatrix} 1 & 0 \\ 0 & -2\pi i \end{bmatrix}.
\end{equation}
Recall the $g$-function $g$ and Lagrange multiplier $\ell$ defined in \eqref{eq:def-g-function}. Also define the complex numbers $a$ and $b$ as
 \eq
a\equiv a(T,\mu):=-\frac{2}{\sqrt{T}}+i\mu, \quad b\equiv b(T,\mu):=\frac{2}{\sqrt{T}}+i\mu,
\endeq
and the potential function
\eq
\label{V-def}
V(z)\equiv V(z;T,\mu) :=\frac{Tz^2}{2}-iT\mu z.
\endeq
We will denote the horizontal line segment between $a$ and $b$ by $(a,b)$, and the half-infinite horizontal ray  $\{ x+i\mu \, | \, x\le 2/\sqrt{T}\}$ by $(-\infty +i\mu, b]$.
Here we note that the function $g(z)$ is the one appearing in the asymptotic analysis of Hermite polynomials \cite{DeiftKMVZ:1999} up to rescaling and a shift, and we record some of its analytic properties.
\begin{itemize}
\item  $g(z)$ is also given by the integral formula
\begin{equation}\label{log-def-g}
g(z) = \int_a^b \log(z-w)\rho(w)\,dw,
\end{equation}
where the principal branch of the logarithm is taken, and $\rho(w)$ is the semicircle density on the interval $(a,b)$ as defined in \eqref{def-rho}.
\item $g(z)$ is analytic for $z\in \C \setminus (-\infty +i\mu, b]$.
\item $g(z)$ takes limiting values from above or below the ray $(-\infty +i\mu, b]$, which we denote $g_+(z)$ and $g_-(z)$, respectively. These functions satisfy the variational condition
\begin{equation} \label{eq:variational_condition}
  g_+(z) + g_-(z) - V(z) - \ell
  \begin{cases}
    = 0, & z\in\{\Im z = \mu\}\cap\{|\Re z| \le 2/\sqrt{T}\}, \\
    < 0, & z\in\{\Im z = \mu\}\cap\{|\Re z| > 2/\sqrt{T}\}.
  \end{cases}
\end{equation}
\item As $ z \to\infty$, $g(z)$ has the expansion
\eq
\label{g1}
g(z) = \log(z) - \frac{i\mu}{z} + \mathcal{O}\left(\frac{1}{z^2}\right).
\endeq
\end{itemize}

Now fix small positive constants $\epsilon$ and $\delta$ and define the 
following regions:
\eq
\begin{split}
\Omega_+ & := \left\{z : |\Re z|<\frac{2}{\sqrt{T}} \text{ and } \mu<\Im z<\mu+\epsilon \right\}, \quad \Omega_- := \left\{z : |\Re z|<\frac{2}{\sqrt{T}} \text{ and } -\delta<\Im z<\mu \right\},\\
\widetilde\Omega_+ & := \left\{z : |\Re z|>\frac{2}{\sqrt{T}} \text{ and } \mu<\Im z<\mu+\epsilon \right\}, \quad \widetilde\Omega_- := \left\{z : |\Re z|>\frac{2}{\sqrt{T}} \text{ and } -\delta<\Im z<\mu \right\}.
\end{split}
\endeq
See Figure \ref{fig-Sn-jumps}.  The boundary between $\Omega_+$ and 
$\Omega_-$ is the \emph{band} $(a,b)$.
The jump matrices will decay to the identity as $n\to\infty$ except on 
this interval.  
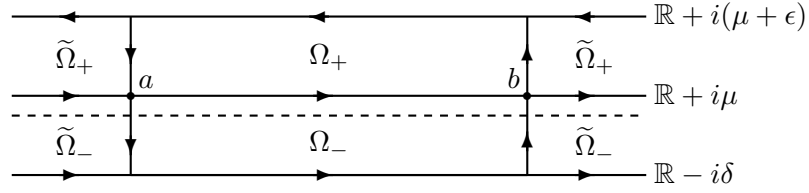
\begin{figure}[h]
\setlength{\unitlength}{1.5pt}
\begin{center}
\begin{picture}(100,60)(-50,5)
\thicklines
\put(-80,55){\line(1,0){160}}
\put(-80,35){\line(1,0){160}}
\put(-50,15){\line(1,0){100}}
\multiput(-80,30)(4,0){40}{\line(1,0){2}}
\put(-80,15){\line(1,0){160}}
\put(-50,15){\line(0,1){40}}
\put(50,15){\line(0,1){40}}
\put(82,33){$\mathbb{R}+i\mu$}
\put(82,53){$\mathbb{R}+i(\mu+\epsilon)$}
\put(82,13){$\mathbb{R}-i\delta$}
\put(-50,35){\circle*{2}}
\put(50,35){\circle*{2}}
\put(-48,37){$a$}
\put(45,37){$b$}
\put(-69,42){$\widetilde\Omega_+$}
\put(62,42){$\widetilde\Omega_+$}
\put(-69,20){$\widetilde\Omega_-$}
\put(62,20){$\widetilde\Omega_-$}
\put(-5,43){$\Omega_+$}
\put(-5,21){$\Omega_-$}
\put(0,55){\vector(-1,0){5}}
\put(67,55){\vector(-1,0){5}}
\put(-63,55){\vector(-1,0){5}}
\put(62,35){\vector(1,0){5}}
\put(-68,35){\vector(1,0){5}}
\put(-4,35){\vector(1,0){5}}
\put(62,15){\vector(1,0){5}}
\put(-68,15){\vector(1,0){5}}
\put(-4,15){\vector(1,0){5}}
\put(50,43){\vector(0,1){5}}
\put(-50,48){\vector(0,-1){5}}
\put(-50,25){\vector(0,-1){5}}
\put(50,21){\vector(0,1){5}}
\end{picture}
\end{center}
\caption{\label{fig-Sn-jumps} The jump contours for ${\bf S}_n$ along with their orientations and the regions $\Omega_\pm$ and $\widetilde{\Omega}_\pm$. The real axis is the dotted horizontal line.}
\end{figure}
Also define the function 
\eq\label{G-def}
G(z) \equiv G(z;T,\mu) := \begin{cases} 2g(z;T,\mu)-V(z;T,\mu)-\ell, & z\in\Omega_+, \\ -2g(z;T,\mu)+V(z;T,\mu)+\ell, & z\in\Omega_-.  \end{cases} 
\endeq
From the equilibrium condition \eqref{eq:variational_condition}, we see that $G(z)$ is analytic in $\Om_+ \cup \Om_-$, and on the band $(a,b)$ it is given by the formula 
\eq\label{diff-def-G}
G(z) = g_+(z) - g_-(z), \quad z\in (a,b).
\endeq
Furthermore, combining \eqref{log-def-g} and \eqref{diff-def-G} gives the 
formula
\eq\label{int-def-G}
G(z) = 2\pi i \int_z^b \rho(w)\,dw, \qquad z\in (a,b),
\endeq
which naturally extends analytically into $\Om_+ \cup \Om_-$.

We are now ready to define the matrix ${\bf S}_n(z)$.  Set
\eq
\label{st2}
{\bf S}_n(z):=\begin{cases} 
\vspace{.02in}
{\bf A} e^{-n\ell\sg_3/2} \mathbf P_n(z) \mathbf D_+^u(z) e^{-n(g(z) - \ell/2)\sg_3} {\bf A}^{-1} \bbm 1 & 0 \\ -e^{-nG(z)} & 1 \ebm, & z\in\Omega_+,\\ 
{\bf A} e^{-n\ell\sg_3/2} \mathbf P_n(z) \mathbf D_-^u(z) e^{-n(g(z) - \ell/2)\sg_3} {\bf A}^{-1} \bbm 1 & 0 \\ e^{nG(z)} & 1 \ebm, & z\in\Omega_-,\\ 
{\bf A} e^{-n\ell\sg_3/2} \mathbf P_n(z) \mathbf D_+^u(z) e^{-n(g(z) - \ell/2)\sg_3} {\bf A}^{-1}, & z\in\widetilde{\Omega}_+,\\ 
{\bf A} e^{-n\ell\sg_3/2} \mathbf P_n(z) \mathbf D_-^u(z) e^{-n(g(z) - \ell/2)\sg_3} {\bf A}^{-1}, & z\in\widetilde{\Omega}_-,\\ 
{\bf A} e^{-n\ell\sg_3/2} \mathbf P_n(z) e^{-n(g(z) - \ell/2)\sg_3} {\bf A}^{-1}, & \text{otherwise}. 
\end{cases}
\endeq
Here the matrices ${\bf D}_\pm^u(z)$ and ${\bf A}$ are for the interpolation 
of poles, the diagonal matrices involving $g(z)$ and $\ell$ are how the 
$g$-function is introduced, and the lower-triangular matrices involving 
$G(z)$ are for the opening of lenses.  As $z\to\infty$, ${\bf S}_n(z)$ satisfies
\eq\label{S-bc}
{\bf S}_n(z)= \mathbb{I} +\bigO(z^{-1}).
\endeq
Also, ${\bf S}_n(z)$ satisfies the 
jump conditions ${\bf S}_{n+}(z)={\bf S}_{n-}(z){\bf V}^{({\bf S})}(z)$, 
where the orientation is given in Figure \ref{fig-Sn-jumps} and the jumps 
are given by 
\begin{equation}\label{st3a}
{\bf V}^{({\bf S})}(z):=
\begin{cases}
\bbm 1 &  \frac{e^{nG(z)}}{1-e^{- 2\pi i(nz-\tau)}} \\ 0 & 1 \ebm, & z\in (\mathbb{R}+i(\mu+\epsilon))\backslash(a+i\epsilon,b+i\epsilon), \vspace{.03in} \\
\bbm 1 &  \frac{e^{nG(z)}}{1-e^{- 2\pi i(nz-\tau)}} \\ 0 & 1 \ebm  \bbm 1 & 0 \\ -e^{-nG(z)} & 1 \ebm, & z\in (a+i\epsilon,b+i\epsilon), \vspace{.03in} \\
\bbm 1 & 0 \\ -e^{-nG(z)} & 1 \ebm, & z\in(a,a+i\epsilon)\cup(b,b+i\epsilon), \vspace{.03in} \\
\bbm 1 & e^{n(g_+(z)+g_-(z)-V(z)-\ell)} \\ 0 & 1 \ebm, & z\in (\mathbb{R}+i\mu)\backslash(a,b), \vspace{.03in} \\
\bbm 0 & 1 \\ -1 & 0 \ebm, & z\in (a,b),  \vspace{.03in} \\
\bbm 1 &  0 \\ e^{nG(z)} & 1 \ebm, & z\in (a,a-i\delta)\cup(b,b-i\delta), \vspace{.03in} \\
\bbm 1 & -\frac{e^{-nG(z)}}{1-e^{2\pi i(nz-\tau)}} \\ 0 & 1 \ebm \bbm 1 & 0 \\ e^{nG(z)} & 1 \ebm, & z\in(a-i\delta,b-i\delta), \vspace{.03in} \\
\bbm 1 & -\frac{e^{-nG(z)}}{1-e^{2\pi i(nz-\tau)}} \\ 0 & 1 \ebm, & z\in (\realR -i\delta)\backslash(a-i\delta,b-i\delta). 
\end{cases}
\end{equation}
The jump conditions and the boundary condition \eqref{S-bc} determine 
${\bf S}_n$ uniquely (see, for example, \cite[Theorem 7.18]{Deift:1998}).

In the $k$-tacnode scaling, the convergence of the jump matrices to the identity matrix fails in a neighborhood of the origin and we need to make a local transformation in a small neighborhood of $z=i\mu$. To that end, first define the regions $\Omega_{\pm}^\Delta$ as
      \begin{equation}
      \begin{aligned}
  \Omega_{+}^\Delta &= \left\{ z =x+iy :   -y+\mu < x< y-\mu, \ \mu < y < \mu+\epsilon \right\}, \\
  \Omega_{-}^\Delta &= \left\{ z =x+iy :   y-\mu < x< -y+\mu, \ -\delta < y < \mu\right\}
  \end{aligned}
  \end{equation}
(see Figure \ref{fig-crit-jumps}). Now define the matrix $\mathbf S_n^{\rm crit}(z)$ as
    \begin{equation}
  \mathbf S_n^{\rm crit}(z) :=  \left\{
   \begin{aligned}
 &   \mathbf S_n(z) \begin{bmatrix} 1 & \pm e^{\pm nG(z)}e^{\pm 2\pi i(nz-\tau)} \\ 0 & 1 \end{bmatrix}, \qquad z\in \Omega_{\pm}^\Delta, \\
&   \mathbf  S_n(z), \qquad {\rm otherwise},
   \end{aligned}\right.
   \end{equation}
where $\mathbf S_n(z)$ is defined in \eqref{st2}. The matrix function $\mathbf S_n^{\rm crit}(z)$ satisfies a Riemann--Hilbert problem similar to $\mathbf S_n(z)$, but with additional jumps in the boundaries of $\Omega_{\pm}^\Delta$, which we denote $\ga_1$, $\ga_2$, $\ga_3$, and $\ga_4$, and orient as shown in Figure \ref{fig-crit-jumps}.

\begin{figure}[h]
\setlength{\unitlength}{1.5pt}
\begin{center}
\begin{picture}(100,60)(-50,5)
\thicklines
\put(-80,55){\line(1,0){160}}
\put(-80,35){\line(1,0){160}}
\put(-50,15){\line(1,0){100}}
\multiput(-80,30)(4,0){40}{\line(1,0){2}}
\put(-80,15){\line(1,0){160}}
\put(-50,15){\line(0,1){40}}
\put(50,15){\line(0,1){40}}
\put(-30,15){\line(3,2){60}}
\put(-30,55){\line(3,-2){60}}
\put(82,33){$\mathbb{R}+i\mu$}
\put(82,53){$\mathbb{R}+i(\mu+\epsilon)$}
\put(82,13){$\mathbb{R}-i\delta$}
\put(-50,35){\circle*{2}}
\put(50,35){\circle*{2}}
\put(0,35){\circle*{2}}
\put(-48,37){$a$}
\put(45,37){$b$}
\put(-5,44){$\Omega_+^\Delta$}
\put(-5,21){$\Omega_-^\Delta$}
\put(-5,59){$\gamma_1$}
\put(20,44){$\gamma_2$}
\put(-26,44){$\gamma_2$}
\put(-26,24){$\gamma_3$}
\put(20,24){$\gamma_3$}
\put(-5,9){$\gamma_4$}
\put(-38,55){\vector(-1,0){5}}
\put(0,55){\vector(-1,0){5}}
\put(42,55){\vector(-1,0){5}}
\put(67,55){\vector(-1,0){5}}
\put(-63,55){\vector(-1,0){5}}
\put(62,35){\vector(1,0){5}}
\put(-68,35){\vector(1,0){5}}
\put(-32,35){\vector(1,0){5}}
\put(24,35){\vector(1,0){5}}
\put(62,15){\vector(1,0){5}}
\put(-68,15){\vector(1,0){5}}
\put(-42,15){\vector(1,0){5}}
\put(-4,15){\vector(1,0){5}}
\put(38,15){\vector(1,0){5}}
\put(50,43){\vector(0,1){5}}
\put(50,21){\vector(0,1){5}}
\put(-50,48){\vector(0,-1){5}}
\put(-50,26){\vector(0,-1){5}}
\put(13.5,44){\vector(-3,-2){1}}
\put(-16.5,46){\vector(-3,2){1}}
\put(-13.5,26){\vector(3,2){1}}
\put(16.5,24){\vector(3,-2){1}}
\end{picture}
\end{center}
\caption{\label{fig-crit-jumps} The jump contours $\Sigma^\text{crit}$ for ${\bf S}_n^\text{crit}$ along with their orientations. The real axis is the dotted horizontal line.}
\end{figure}
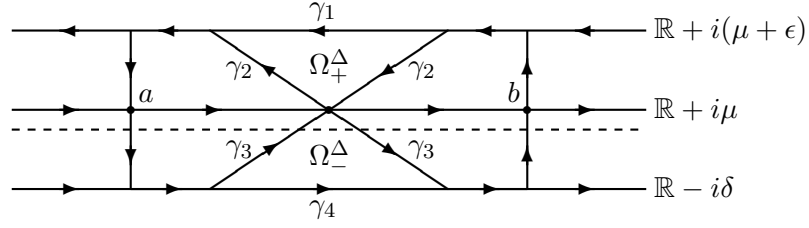

It is straightforward to check that the jump matrices for 
$\mathbf S_n^{\rm crit}(z)$ are the same as those for $  \mathbf S_n(z)$, 
see \eqref{st3a}, except on the contours $\ga_1$, $\ga_2$, $\ga_3$, and 
$\ga_4$.  The jump conditions are 
\begin{equation}
\mathbf S_{n+}^{\rm crit}(z) = \mathbf S_{n-}^{\rm crit}(z){\bf V}^{\rm crit}(z),
\end{equation}
where the orientations are given in Figure \ref{fig-crit-jumps} and 
\begin{equation}
\label{Vcrit-def}
{\bf V}^{\rm crit}(z)= \left\{
\begin{aligned}
&\begin{bmatrix} (1-e^{2\pi i (nz-\tau)})^{-1} & 0 \\ -e^{-nG(z)} & 1-e^{2\pi i (nz-\tau)} \end{bmatrix}, \quad &z\in \ga_1, \\
&\begin{bmatrix} 1 &  -e^{ nG(z)}e^{2\pi i(nz-\tau)} \\ 0 & 1 \end{bmatrix}, \quad &z\in \ga_2, \\
&\begin{bmatrix} 1 &  e^{ -nG(z)}e^{-2\pi i(nz-\tau)} \\ 0 & 1 \end{bmatrix}, \quad &z\in \ga_3, \\
&\begin{bmatrix} (1-e^{-2\pi i (nz-\tau)})^{-1} & 0 \\ e^{nG(z)} & 1-e^{-2\pi i (nz-\tau)} \end{bmatrix}, \quad &z\in \ga_4, \\
&{\bf V}^{({\bf S})}(z), \quad & \text{otherwise}.
\end{aligned}
\right.
\end{equation}

%

We now show that, as $n\to\infty$, the jump matrices for 
${\bf S}_n^\text{crit}(z)$ decay exponentially to 
the identity matrix except in a neighborhood of the band $(a,b)$.  Let 
$\mathbb{D}_a$, $\mathbb{D}_b$, and $\mathbb{D}_{i\mu}$ be small fixed 
circular neighborhoods centered at $a$, $b$, and $i\mu$, respectively, small 
enough so their closures do not intersect $\mathbb{R}+i(\mu+\epsilon)$, 
$\mathbb{R}-i\delta$, or each other.  

\begin{lem}
\label{lemma:Scrit-jumps}
Fix a non-negative integer $k$, choose $\mu$ according to 
\eqref{mu-range},  and suppose $(\pi^2-T)n^{2/3}$ remains bounded as 
$n\to\infty$. Then there exists a constant 
$c>0$ such that
\eq
{\bf V}^{\rm crit}(z) = \mathbb{I} + \mathcal{O}(e^{-cn}) \text{ as } n\to\infty \text{ for } z\in\Sigma^{\rm crit}\cap(\mathbb{D}_a\cup\mathbb{D}_b\cup\mathbb{D}_{i\mu})^\mathsf{C}
\endeq
(here $\mathsf{C}$ denotes the complement).  
\end{lem}
\begin{proof}
Throughout the proof we assume 
$z\notin\mathbb{D}_a\cup\mathbb{D}_b\cup\mathbb{D}_{i\mu}$ in order to avoid 
rewriting this condition.  From \eqref{Vcrit-def}, it is enough to prove the 
following:
\begin{itemize}
\item $\Re G(z)>0$ for $z\in (a+i\epsilon,b+i\epsilon)\cup(a,a+i\epsilon)\cup(b,b+i\epsilon)$,
\item $\Re G(z)<0$ for $z\in (a-i(\mu+\delta),b-i(\mu+\delta))\cup(a,a-i(\mu+\delta))\cup(b,b-i(\mu+\delta))$,
\item $\Re(G(z)+2\pi iz)<0$ for $z\in\gamma_2 \cup [(\mathbb{R}+i(\mu+\epsilon))\backslash\gamma_1]$,
\item $\Re(G(z)+2\pi iz)>0$ for $z\in\gamma_3 \cup [(\mathbb{R}-i\delta)\backslash\gamma_4]$,
\item $\Re(g_+(z)+g_-(z)-V(z)-\ell)<0$ for $z\in(\mathbb{R}+i\mu)\backslash(a,b).$
\end{itemize}
Each of these conditions follow from the properties \eqref{log-def-g}--\eqref{eq:variational_condition}, and was shown in the tacnode scaling regime with $\mu=0$ 
in \cite{Liechty:2012}.  (In fact, the arguments used are the same as in the 
case when $T\in(0,\pi^2)$ ;  the only significant change if $T\approx\pi^2$ 
occurs near $z=i\mu$.)  All of the quantities involved ($G(z;T\mu)$, 
$g(z;T,\mu)$, $V(z;T,\mu)$, $\ell(T,\mu)$, $a(T,\mu)$, $b(T,\mu)$) are 
continuous as functions of $\mu$.  Therefore, if $\mu$ is given by 
\eqref{mu-range}, then these conditions must hold as long as $n$ is 
sufficiently large.  
\end{proof}

\section{Initial construction of parametrices}
\label{sec-parametrices}

We build an approximation, or model, of ${\bf S}^{\rm crit}_n(z)$ in four 
pieces.  Inside the disks $\mathbb{D}_{i\mu}$, $\mathbb{D}_a$, and 
$\mathbb{D}_b$, we construct the inner model solutions or parametrices 
${\bf M}_k^{(i\mu)}(z)$, ${\bf M}_k^{(a)}(z)$, and ${\bf M}_k^{(b)}(z)$ to 
satisfy exactly the same jumps as ${\bf S}^{\rm crit}_n(z)$.  These 
constructions each involve a local conformal change of variables.  When 
necessary we assume without further comment the jump contours for 
${\bf S}^{\rm crit}_n(z)$ have been locally deformed in order to map exactly 
onto the parametrix jump contours.  Outside of the three disks 
${\bf S}^{\rm crit}_n(z)$ is approximated by the outer model solution 
${\bf M}_k^{(\text{out})}(z)$.  It is necessary to closely match the inner 
model solutions to the outer model solution on the disk boundaries.  It 
turns out that the match between ${\bf M}_k^{(\text{out})}(z)$ and 
${\bf M}_k^{(i\mu)}(z)$ is not uniform in $\mu$, and in fact there are certain 
values of $\mu$ for which the matching is insufficient to give controlled 
error bounds.  This issue will be taken care of in \S\ref{sec-error} via the 
construction of a \emph{parametrix for the error}.  

\subsection{The outer model problem near criticality}

It turns out that if we use the same outer model problem used in 
\cite{LiechtyW:2016} in the tacnode case, the inner and outer model solutions 
do not match well on $\partial\mathbb{D}_{i\mu}$.  This difficulty can be 
circumvented by requiring the outer model problem to have a pole of a 
specified order at $i\mu$.  This technique has previously been used to 
analyze the emergence of a spectral cut in unitarily invariant random matrix 
ensembles \cite{BertolaL:2009}, a smooth-to-oscillatory transition for the 
semiclassical Korteweg-de Vries equation \cite{ClaeysG:2010}, a 
librational-to-rotational transition for the semiclassical sine-Gordon 
equation \cite{BuckinghamM:2012}, and at the edge of the pole region for 
rational solutions for the Painlev\'e-II equation \cite{BuckinghamM:2015}.  
With this motivation, we therefore pose the outer model problem as follows.
\begin{rhp}[Outer model problem near criticality]
\label{rhp-outer-crit}
Fix $k\in\mathbb{Z}$ and determine a $2\times 2$ 
matrix-valued function ${\bf M}_k^{(\rm{out})}(z)$ satisfying:
\begin{itemize}
\item[]{\bf Analyticity:} ${\bf M}_k^{(\rm{out})}(z)$ is analytic in $z$ 
off $[a,b]$ and is H\"older continuous up to $(a,b)$ except in 
$\mathbb{D}_{i\mu}$ with at most quarter-root singularities at $a$ and $b$.  
Furthermore, the function 
\eq
\widetilde{{\bf M}}_k^{(\rm{out})}(z):=\begin{cases} {\bf M}_k^{(\rm{out})}(z)(z-i\mu)^{k\sigma_3}, & z\in\mathbb{D}_{i\mu}\cap\{\Im(z)>\mu\}, \\  {\bf M}_k^{(\rm{out})}(z)(z-i\mu)^{-k\sigma_3}, & z\in\mathbb{D}_{i\mu}\cap\{\Im(z)<\mu\}\end{cases}
\endeq
is analytic in its domain of definition.
\item[]{\bf Normalization:}  
\eq
\lim_{z\to\infty}{\bf M}_k^{(\rm{out})}(z) = \mathbb{I}.
\endeq
\item[]{\bf Jump condition:}  Orienting $[a,b]$ left-to-right, the solution 
satisfies 
\eq
{\bf M}_{k+}^{(\rm{out})}(z)={\bf M}_{k-}^{(\rm{out})}(z)\bbm 0 & 1 \\ -1 & 0 \ebm, \quad z\in[a,b].
\endeq
\end{itemize}
\end{rhp}
To solve the outer model problem we begin by defining $R(z)$ to be the function 
satisfying 
\eq
R(z)^2 = (z-a)(z-b)
\endeq
with branch cut $[a,b]$ and asymptotics $R(z)=z+\mathcal{O}(1)$ as $z\to\infty$.  
Next, define the function
\eq
\label{d-def}
d(z):=\frac{R(z)+\frac{2i}{\sqrt{T}}}{z-i\mu}.
\endeq
The following lemma records some properties of $d(z)$ that are easily checked directly.
\begin{lem}
\label{d-properties-lem}
\begin{itemize}
\item[(a)] $d(z)$ is analytic off $[a,b]$.
\item[(b)] $d_+(z)d_-(z)=-1$ for $z\in(a,b)$, where $(a,b)$ is oriented left-to-right.
\item[(c)] $\displaystyle d(z) = \frac{4i}{\sqrt{T}}\frac{1}{z-i\mu} + \mathcal{O}(z-i\mu)$ for $z\in\mathbb{D}_{i\mu}\cap\{\Im z>0\}$.
\item[(d)] $\displaystyle d(z) = \frac{\sqrt{T}}{4i}(z-i\mu) + \mathcal{O}((z-i\mu)^3)$ for $z\in\mathbb{D}_{i\mu}\cap\{\Im z<0\}$.
\item[(e)] $d(z)=-i+\mathcal{O}(\sqrt{a-z})$ for $z\in\mathbb{D}_a$.
\item[(f)] $d(z)=i+\mathcal{O}(\sqrt{z-b})$ for $z\in\mathbb{D}_b$.
\item[(g)] $\displaystyle d(z)=1+\frac{2i}{\sqrt{T}z} + \mathcal{O}\left(\frac{1}{z^2}\right)$ as $z\to\infty$.
\end{itemize}
\end{lem}
Also define
\eq
\label{gamma-def}
\gamma(z):=\left(\frac{z-a}{z-b}\right)^{1/4}
\endeq
with branch cut $[a,b]$ and $\lim_{z\to\infty}\gamma(z)=1$.  
Now Riemann--Hilbert Problem \ref{rhp-outer-crit} is solved by 
\eq
\label{M-k-out}
{\bf M}_k^{(\text{out})}(z):=e^{-ik\pi\sigma_3/2} \bbm \displaystyle \frac{\gamma(z)+\gamma(z)^{-1}}{2} & \displaystyle \frac{\gamma(z)-\gamma(z)^{-1}}{-2i} \\ \displaystyle \frac{\gamma(z)-\gamma(z)^{-1}}{2i} & \displaystyle \frac{\gamma(z)+\gamma(z)^{-1}}{2} \ebm e^{ik\pi\sigma_3/2}d(z)^{k\sigma_3}.
\endeq

\subsection{The inner model problem near $z=i\mu$}

We now construct the function ${\bf M}_k^{(i\mu)}$ that satisfies the same 
jumps as ${\bf S}^{\rm crit}(z)$ for $z\in\mathbb{D}_{i\mu}$ and approximately matches 
${\bf M}_k^{(\text{out})}(z)$ for $z\in\partial\mathbb{D}_{i\mu}$.  It is 
convenient to work in a local variable $\zeta(z)$ in which the jump conditions 
take a particularly nice form.  For $z\in\mathbb{D}_{i\mu}$, the jump exponent 
has the expansion
\eq
\label{nG-expansion}
\begin{split}
n(G(z)&+2\pi iz) - 2\pi i\tau \\
& = (ni\pi-2\pi\mu n - 2\pi i\tau) + 2in(\pi-\sqrt{T})(z-i\mu) + \frac{inT^{3/2}}{12}(z-i\mu)^3 + \mathcal{O}((z-i\mu)^4).
\end{split}
\endeq
Note that at criticality (i.e. $T=\pi^2$ and $\mu=0$) the coefficient of 
$z-i\mu$ vanishes while that of $(z-i\mu)^3$ does not.  It is therefore 
reasonable to expect that the exponent can be modeled by a cubic polynomial.  
Indeed, following Chester, Friedman, and Ursell \cite{ChesterFU:1957} (see also 
\cite{BuckinghamM:2012,Liechty:2012}), there is, for $z$ sufficiently close to 
$i\mu$ and $T$ and $\mu$ sufficiently close to criticality, an invertible 
conformal mapping $\zeta(z)=\zeta(z;\mu,T)$ as well as analytic functions 
$s(\mu,T)$ and $\theta(\mu)$ such that $\zeta(i\mu) = 0$ and
\eq
\label{zeta-conf-map}
n(G(z)+2\pi iz) - 2\pi i\tau  = 2i\left(\frac{4}{3}\zeta(z)^3+s\zeta(z)-\theta\right).
\endeq
If necessary, we shrink the size of $\mathbb{D}_{i\mu}$ to ensure these conditions 
hold for all $z\in\mathbb{D}_{i\mu}$.  By plugging the expansion \eqref{nG-expansion} into \eqref{zeta-conf-map} and matching constant terms we find that 
\eq\label{theta-explicit}
\theta(\mu) = -\frac{n\pi}{2} - in\pi\mu + \pi\tau.
\endeq
The change of variables \eqref{zeta-conf-map} is nearly identical to the one presented in \cite[Section 4.9]{Liechty:2012}, up to a shift by $i\mu$ in the definition of $G(z)$. The parameters $s$ and $\theta$ are defined in terms of the stationary points of the left-hand-side of \eqref{zeta-conf-map}, and the analysis presented in \cite[Section 4.9]{Liechty:2012} applies to \eqref{zeta-conf-map} as well. The result is the formula \eqref{def-s} for $s$, compare \cite[Equation (1.43)]{Liechty:2012}. For a similar application of the Chester--Friedman--Ursell change of variables with more details given, see also \cite[Section 4.3]{BuckinghamM:2012}.

We note that
\eq
\label{zeta-growth-in-n}
\zeta(z) = \mathcal{O}(n^{1/3}) \text{ for } z\in\mathbb{D}_{i\mu}
\endeq
as well as the facts that 
\eq
\frac{\zeta(z)}{z-i\mu}=\mathcal{O}(1) \text{ and } \frac{z-i\mu}{\zeta(z)}=\mathcal{O}(1) \text{ for } z\in\mathbb{D}_{i\mu}.
\endeq
More precisely, inserting 
$\zeta(z)=\zeta'(i\mu)(z-i\mu)+\mathcal{O}\left((z-i\mu)^2\right)$ into 
\eqref{nG-expansion} and \eqref{zeta-conf-map}, and then using 
\eqref{T-scaling}, shows 
\eq
\label{zeta-prime}
\zeta'(i\mu) = \frac{n(\pi-\sqrt{T})}{s} = \frac{\pi n^{1/3}}{2^{5/3}} + \mathcal{O}\left(\frac{1}{n^{2/3}}\right).
\endeq

We are now ready to pose the inner model problem.
\begin{rhp}[The inner model problem in $\mathbb{D}_{i\mu}$ near criticality]
\label{rhp-inner-crit}
Fix $s\in\mathbb{R}$ and $k\in\mathbb{Z}$.  Determine 
a $2\times 2$ matrix-valued function ${\bf M}_k^{(i\mu)}(\zeta(z))$ satisfying:
\begin{itemize}
\item[]{\bf Analyticity:} ${\bf M}_k^{(i\mu)}(\zeta)$ is analytic for 
$\zeta\in\mathbb{D}_{i\mu}$ off the six rays $\arg(\zeta)\in\{0,\pm\frac{\pi}{6}$, 
$\pm\frac{5\pi}{6},\pi\}$.  In each sector the solution can be analytically 
continued into a larger sector, and is H\"older continuous up to the 
boundary in a neighborhood of $\zeta=0$.
\item[]{\bf Normalization:}  
\eq
{\bf M}_k^{(i\mu)}(\zeta(z)) = (\mathbb{I}+o(1)){\bf M}_k^{(\rm{out})}(z) \text{ as }n\to\infty\text{ for }z\in\partial\mathbb{D}_{i\mu}.
\endeq
\item[]{\bf Jump condition:}  The solution satisfies 
${\bf M}_{k+}^{(i\mu)}(\zeta)={\bf M}_{k-}^{(i\mu)}(\zeta){\bf V}^{(i\mu)}(\zeta)$, 
with jumps as shown in Figure \ref{fig-inner-jumps}.
\begin{figure}[h]
\setlength{\unitlength}{1.5pt}
\begin{center}
\begin{picture}(100,100)(-50,-50)
\thicklines
\put(0,0){\line(3,2){40}}
\put(0,0){\line(-3,2){40}}
\put(0,0){\line(-3,-2){40}}
\put(0,0){\line(3,-2){40}}
\put(-40,0){\line(1,0){80}}
\put(0,0){\vector(3,2){22}}
\put(0,0){\vector(-3,2){22}}
\put(0,0){\vector(-3,-2){22}}
\put(0,0){\vector(3,-2){22}}
\put(-40,0){\vector(1,0){20}}
\put(0,0){\vector(1,0){26}}
\put(0,0){\circle*{2}}
\put(-2,4){$0$}
\put(42,25){$\bbm 1 & e^{2i(\frac{4}{3}\zeta^3+s\zeta-\theta)} \\ 0 & 1 \ebm$}
\put(-105,25){$\bbm 1 & -e^{2i(\frac{4}{3}\zeta^3+s\zeta-\theta)} \\ 0 & 1 \ebm$}
\put(-109,-28){$\bbm 1 & -e^{-2i(\frac{4}{3}\zeta^3+s\zeta-\theta)} \\ 0 & 1 \ebm$}
\put(42,-28){$\bbm 1 & e^{-2i(\frac{4}{3}\zeta^3+s\zeta-\theta)} \\ 0 & 1 \ebm$}
\put(42,0){$\bbm 0 & 1 \\ -1 & 0 \ebm$}
\put(-69,0){$\bbm 0 & 1 \\ -1 & 0 \ebm$}
\end{picture}
\end{center}
\caption{\label{fig-inner-jumps} The jump contours $\Sigma^{(i\mu)}$ and jump 
matrices ${\bf V}^{(i\mu)}(\zeta)$.}
\end{figure}
\end{itemize}
\end{rhp}

We now perform a series of changes of variables 
\begin{equation*}
{\bf M}_k^{(i\mu)} \to {\bf Z}_k^{(1)} \to {\bf Z}_k^{(2)} \to {\bf Z}_k
\end{equation*}
to transform the Riemann--Hilbert 
problem for ${\bf M}_k^{(i\mu)}$ into a Riemann--Hilbert problem associated 
with the Painlev\'e-II equation.  Note that
\eq
\label{E-k-def}
{\bf E}_k(z):= \begin{cases} {\bf M}_k^{(\text{out})}(z)\bbm 0 & -1 \\ 1 & 0 \ebm \left(\displaystyle\frac{n^{1/3}}{\zeta(z)}\right)^{k\sigma_3}, & z\in\mathbb{D}_{i\mu}\cap\{\Im z>\mu\}, \\ {\bf M}_k^{(\text{out})}(z) \left(\displaystyle\frac{n^{1/3}}{\zeta(z)}\right)^{k\sigma_3}, & z\in\mathbb{D}_{i\mu}\cap\{\Im z<\mu\}  \end{cases}
\endeq
is analytic and invertible for $z\in\mathbb{D}_{i\mu}$.  Given the required 
normalization for ${\bf M}_k^{(i\mu)}(\zeta(z))$, we have 
\eq
{\bf M}_k^{(i\mu)}(\zeta(z)) = \begin{cases} (\mathbb{I}+o(1)){\bf E}_k(z)\left(\frac{\zeta(z)}{n^{1/3}}\right)^{k\sigma_3}\bbm 0 & 1 \\ -1 & 0 \ebm, & z\in\partial\mathbb{D}_{i\mu}\cap\{\Im z>\mu\}, \\ (\mathbb{I}+o(1)){\bf E}_k(z)\left(\frac{\zeta(z)}{n^{1/3}}\right)^{k\sigma_3}, & z\in\partial\mathbb{D}_{i\mu}\cap\{\Im z<\mu\}. \end{cases}
\endeq
We now pull out this analytic factor from ${\bf M}_k^{(i\mu)}$:
\eq
\label{Zk1-def}
{\bf Z}_k^{(1)}(\zeta(z)):={\bf E}_k(z)^{-1}{\bf M}_k^{(i\mu)}(\zeta(z)) \text{ for }z\in\mathbb{D}_{i\mu}.
\endeq
Now ${\bf Z}_k^{(1)}$ has the same jumps as ${\bf M}_k^{(i\mu)}$, but the 
normalization changes to 
\eq
{\bf Z}_k^{(1)}(\zeta(z)) = \begin{cases} (\mathbb{I}+o(1))\left(\frac{\zeta(z)}{n^{1/3}}\right)^{k\sigma_3}\bbm 0 & 1 \\ -1 & 0 \ebm, & z\in\partial\mathbb{D}_{i\mu}\cap\{\Im z>\mu\}, \\ (\mathbb{I}+o(1))\left(\frac{\zeta(z)}{n^{1/3}}\right)^{k\sigma_3}, & z\in\partial\mathbb{D}_{i\mu}\cap\{\Im z<\mu\}. \end{cases}
\endeq
The next transformation removes the 
jump on the real axis and switches the triangularity of the jump matrices in 
the upper half-plane.
\eq
\label{Zk2-def}
{\bf Z}_k^{(2)}(\zeta):=\begin{cases} \bbm 1 & 0 \\ 0 & -1 \ebm {\bf Z}_k^{(1)}(\zeta)\bbm 0 & 1 \\ 1 & 0 \ebm, \quad \Im\zeta>0, \vspace{.05in} \\ \bbm 1 & 0 \\ 0 & -1 \ebm {\bf Z}_k^{(1)}(\zeta)\bbm 1 & 0 \\ 0 & -1 \ebm, \quad \Im\zeta<0.\end{cases}
\endeq
The jumps for ${\bf Z}_k^{(2)}(\zeta)$ are shown in Figure 
\ref{fig-Z2-jumps}.  
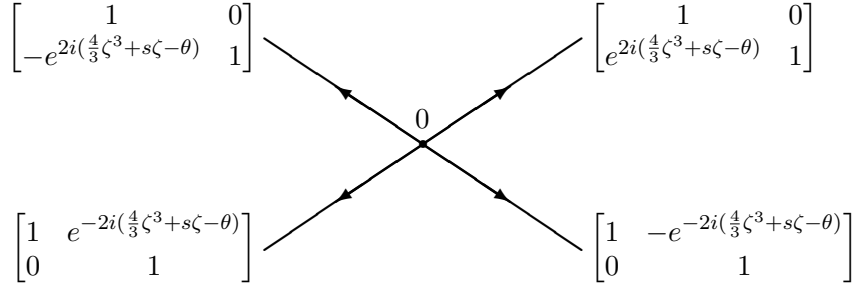
\begin{figure}[h]
\setlength{\unitlength}{1.5pt}
\begin{center}
\begin{picture}(100,100)(-50,-50)
\thicklines
\put(0,0){\line(3,2){40}}
\put(0,0){\line(-3,2){40}}
\put(0,0){\line(-3,-2){40}}
\put(0,0){\line(3,-2){40}}
\put(0,0){\vector(3,2){22}}
\put(0,0){\vector(-3,2){22}}
\put(0,0){\vector(-3,-2){22}}
\put(0,0){\vector(3,-2){22}}
\put(0,0){\circle*{2}}
\put(-2,4){$0$}
\put(42,25){$\bbm 1 & 0 \\ e^{2i(\frac{4}{3}\zeta^3+s\zeta-\theta)} & 1 \ebm$}
\put(-105,25){$\bbm 1 & 0 \\ -e^{2i(\frac{4}{3}\zeta^3+s\zeta-\theta)} & 1 \ebm$}
\put(-104,-28){$\bbm 1 & e^{-2i(\frac{4}{3}\zeta^3+s\zeta-\theta)} \\ 0 & 1 \ebm$}
\put(42,-28){$\bbm 1 & -e^{-2i(\frac{4}{3}\zeta^3+s\zeta-\theta)} \\ 0 & 1 \ebm$}
\end{picture}
\end{center}
\caption{\label{fig-Z2-jumps}  The jump contours and matrices for  
${\bf Z}_k^{(2)}(\zeta)$.}
\end{figure}
This also has the effect of simplifying the normalization:
\eq
{\bf Z}_k^{(2)}(\zeta) = (\mathbb{I}+o(1))\left(\frac{\zeta}{n^{1/3}}\right)^{k\sigma_3} \text{ as }\zeta\to\infty.
\endeq
Here $o(1)$ refers to growth in $n$, but because $\zeta(z)$ grows like 
$n^{1/3}$, we could just as well think of it as referring to growth in 
$\zeta$.  This is advantageous as we want to pose a problem in the $\zeta$ 
plane with no reference to $z$ or $n$.  
The final transformation removes the dependence on $n$, $\mu$, and $\tau$ 
from the Riemann--Hilbert problem:
\eq
\label{Zk-def}
{\bf Z}_k(\zeta):=n^{k\sigma_3/3}e^{-i\theta\sigma_3}{\bf Z}_k^{(2)}(\zeta)e^{i\theta\sigma_3}.
\endeq
For future reference, we note that combining \eqref{Zk1-def}, \eqref{Zk2-def}, 
and \eqref{Zk-def} gives
\eq
\label{M-k-imu}
{\bf M}_k^{(i\mu)}(\zeta(z)) = \begin{cases} {\bf E}_k(z)\sigma_3 e^{i\theta\sigma_3}n^{-k\sigma_3/3}{\bf Z}_k(\zeta(z))e^{-i\theta\sigma_3} \sigma_1, & z\in\mathbb{D}_{i\mu}\cap\{\Im z>\mu\}, \\  {\bf E}_k(z)\sigma_3 e^{i\theta\sigma_3}n^{-k\sigma_3/3}{\bf Z}_k(\zeta(z))e^{-i\theta\sigma_3} \sigma_3, & z\in\mathbb{D}_{i\mu}\cap\{\Im z<\mu\}. \end{cases}
\endeq
We state the Riemann--Hilbert problem for ${\bf Z}_k(\zeta)$ and relate its 
solution to the generalized Hastings--McLeod functions in the following 
subsection.

\subsection{The Jimbo--Miwa RHP for generalized Hastings--McLeod functions}

The function ${\bf Z}_k(\zeta)$ defined in \eqref{Zk-def} is the unique 
solution to the following Riemann--Hilbert problem.  
\begin{rhp}[Jimbo--Miwa problem for generalized Hastings--McLeod functions]
\label{rhp-jm}
Fix $s\in\mathbb{R}$ and $k\in\mathbb{Z}$ and determine a $2\times 2$ 
matrix-valued function ${\bf Z}_k(\zeta;s)$ satisfying:
\begin{itemize}
\item[]{\bf Analyticity:} ${\bf Z}_k(\zeta;s)$ is analytic in $\zeta$ 
off the four rays $\arg(\zeta)\in\{\pm\frac{\pi}{6}$, 
$\pm\frac{5\pi}{6}$\}.  In each sector the solution can be analytically 
continued into a larger sector, and is H\"older continuous up to the 
boundary in a neighborhood of $\zeta=0$.
\item[]{\bf Normalization:}  Uniformly with respect to $\arg(\zeta)$ 
in each sector of analyticity,
\eq
\lim_{\zeta\to\infty}{\bf Z}_k(\zeta;s)\zeta^{-k\sigma_3} = \mathbb{I}.
\endeq
\item[]{\bf Jump condition:}  Orienting the four jump rays towards 
infinity, the solution satisfies 
${\bf Z}_{k+}(\zeta;s)={\bf Z}_{k-}(\zeta;s){\bf V}^{({\bf Z})}(\zeta;s)$, as 
shown in Figure \ref{fig-jm-jumps}.
\begin{figure}[h]
\setlength{\unitlength}{1.5pt}
\begin{center}
\begin{picture}(100,100)(-50,-50)
\thicklines
\put(0,0){\line(3,2){40}}
\put(0,0){\line(-3,2){40}}
\put(0,0){\line(-3,-2){40}}
\put(0,0){\line(3,-2){40}}
\put(0,0){\vector(3,2){22}}
\put(0,0){\vector(-3,2){22}}
\put(0,0){\vector(-3,-2){22}}
\put(0,0){\vector(3,-2){22}}
\put(0,0){\circle*{2}}
\put(-2,4){$0$}
\put(42,25){$\bbm 1 & 0 \\ e^{2i(\frac{4}{3}\zeta^3+s\zeta)} & 1 \ebm$}
\put(-98,25){$\bbm 1 & 0 \\ -e^{2i(\frac{4}{3}\zeta^3+s\zeta)} & 1 \ebm$}
\put(-97,-28){$\bbm 1 & e^{-2i(\frac{4}{3}\zeta^3+s\zeta)} \\ 0 & 1 \ebm$}
\put(42,-28){$\bbm 1 & -e^{-2i(\frac{4}{3}\zeta^3+s\zeta)} \\ 0 & 1 \ebm$}
\end{picture}
\end{center}
\caption{\label{fig-jm-jumps}  The jump contours $\Sigma^{({\bf Z})}$ and jump 
matrices ${\bf V^{({\bf Z})}}(\zeta;s)$.}
\end{figure}
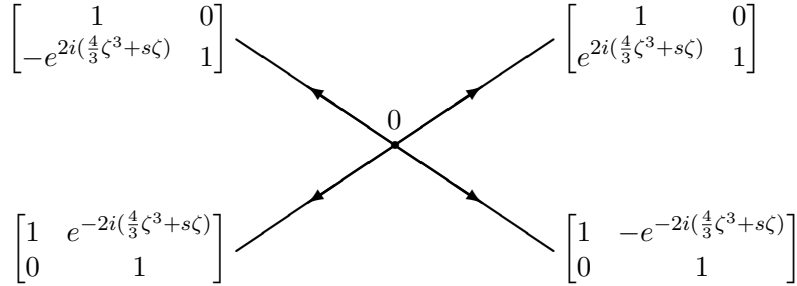
\end{itemize}
\end{rhp}

This is a special case of the Riemann--Hilbert problem posed in 
Fokas et al., Chapter 3, \S3.2 \cite{FokasIKN:2006}.  Specifically, 
${\bf Y}(\lambda):={\bf Z}_\nu(\lambda;x)\lambda^{-\nu\sigma_3}$ satisfies 
that Riemann--Hilbert problem with the specific choice of Stokes data 
$s_1=s_4=1$, $s_2=s_5=0$, and $s_3=s_6=-1$.  Fokas et al. show that the 
solution to this problem exists and is unique.  They furthermore 
(see \cite[Chapter 4, \S2.5]{FokasIKN:2006}) show that the function ${\bf L}_k(\zeta;s)$ defined as
\eq
\label{L-def}
{\bf L}_k(\zeta;s):={\bf Z}_k(\zeta;s)e^{-i(\frac{4}{3}\zeta^3+s\zeta)\sigma_3}
\endeq
satisfies the Jimbo--Miwa (or Jimbo--Miwa--Garnier) Lax pair \cite{JimboM:1981}
\eq
\label{Jimbo-Miwa-Lax-pair}
\begin{split}
\frac{\partial{\bf L}_k}{\partial s} & = \left( -i\sigma_3\zeta + i\bbm 0 & \mathcal{U}_k(s) \\ -\mathcal{V}_k(s) & 0 \ebm \right) {\bf L}_k, \\
\frac{\partial{\bf L}_k}{\partial\zeta} & = \left( -4i\sigma_3\zeta^2 + 4i \bbm 0 & \mathcal{U}_k(s) \\ -\mathcal{V}_k(s) & 0 \ebm \zeta + \bbm -2i\mathcal{U}_k(s)\mathcal{V}_k(s) - is & -2\mathcal{U}_k'(s) \\ -2\mathcal{V}_k'(s) & 2i\mathcal{U}_k(s)\mathcal{V}_k(s) + is \ebm \right) {\bf L}_k,
\end{split}
\endeq
and that this overdetermined system is equivalent to 
the coupled Painlev\'e-II system \eqref{coupled-PII}.  They also 
show that the logarithmic derivatives of $\mathcal{U}_k(s)$ and 
$\mathcal{V}_k(s)$ satisfy (uncoupled) Painlev\'e-II equations, a calculation 
we carry out below.  This implies that $\mathcal{U}_k(s)$ and 
$\mathcal{V}_k(s)$ are tau functions for certain Painlev\'e-II 
transcendents.  The remainder of the section is devoted to identifying these 
transcendents as certain generalized Hastings--McLeod functions.  We begin by 
expressing certain terms in the large-$\zeta$ expansion of ${\bf Z}_k(\zeta)$ 
in terms of $\mathcal{U}_k(s)$ and $\mathcal{V}_k(s)$.  With the help of this 
expansion, we then identify the Painlev\'e-II equations satisfied by the 
logarithmic derivatives of $\mathcal{U}_k(s)$ and $\mathcal{V}_k(s)$.  Next, 
we obtain the B\"acklund transformations relating $\mathcal{U}_{k+1}(s)$ and
$\mathcal{V}_{k+1}(s)$ to $\mathcal{U}_k(s)$.  Combining the B\"acklund 
transformations with the fact that $\mathcal{U}_0(s)$ and $\mathcal{V}_0(s)$ 
can be identified as Hastings--McLeod functions from the Riemann--Hilbert 
problem they satisfy allows us to identify the logarithmic derivatives of 
$\mathcal{U}_k(s)$ and $\mathcal{V}_k(s)$.

\subsubsection{Expansion of ${\bf Z}_k$}
We need the large-$\zeta$ expansion of ${\bf Z}_k(\zeta;s)$: 
\eq
\label{Z-expansion}
{\bf Z}_k(\zeta;s)\zeta^{-k\sigma_3} = \mathbb{I} + \frac{{\bf A}_k(s)}{\zeta} + \frac{{\bf B}_k(s)}{\zeta^2} + \mathcal{O}\left(\frac{1}{\zeta^3}\right).
\endeq
To compute this expansion, we begin by assuming ${\bf L}_k(\zeta;s)$ 
has the expansion 
\eq
\label{L-expansion}
\begin{split}
{\bf L}_k(\zeta) = & \left(\mathbb{I}+\frac{{\bf Y}_k^{(1)}}{\zeta}+\frac{{\bf Y}_k^{(2)}}{\zeta^2}+\frac{{\bf Y}_k^{(3)}}{\zeta^3}+\mathcal{O}\left(\frac{1}{\zeta^4}\right)\right) \\
  & \times \exp\left({\bf D}_k^{(3)}\zeta^3+{\bf D}_k^{(2)}\zeta^2+{\bf D}_k^{(1)}\zeta+{\bf D}_k^{(0)}\log\zeta+{\bf D}_k^{(-1)}\zeta^{-1}+\mathcal{O}\left(\frac{1}{\zeta^2}\right) \right),
\end{split}
\endeq
where ${\bf Y}_k^{(j)}$, $j=1,2,3$ are independent of $\zeta$ and off-diagonal, 
while ${\bf D}_k^{(j)}$, $j=-3,2,1,0,-1$ are independent of $\zeta$ and 
diagonal.  We insert this ansatz into the $\zeta$-derivative equation in the 
Lax pair \eqref{Jimbo-Miwa-Lax-pair}  and group diagonal and off-diagonal terms 
in each power of $\zeta$.
\eq
\begin{split}
\mathcal{O}(\zeta^2)\text{ diagonal}: & \quad 3{\bf D}_k^{(3)} = -4i\sigma_3,\\
\mathcal{O}(\zeta)\text{ diagonal}: & \quad  {\bf D}_k^{(2)}={\bf 0},\\
\mathcal{O}(\zeta)\text{ off-diagonal}: & \quad 3{\bf Y}_k^{(1)}{\bf D}_k^{(3)} = -4i\sigma_3{\bf Y}_k^{(1)} + 4i\bbm 0 & \mathcal{U}_k \\ -\mathcal{V}_k & 0 \ebm, \\
\mathcal{O}(1)\text{ diagonal}: & \quad {\bf D}_k^{(1)} = 4i\bbm 0 & \mathcal{U}_k \\ -\mathcal{V}_k & 0 \ebm{\bf Y}_k^{(1)} + \bbm -2i\mathcal{U}_k\mathcal{V}_k-is & 0 \\ 0 & 2i\mathcal{U}_k\mathcal{V}_k+is \ebm,\\
\mathcal{O}(1)\text{ off-diagonal}: & \quad 2{\bf Y}_k^{(1)}{\bf D}_k^{(2)} + 3{\bf Y}_k^{(2)}{\bf D}_k^{(3)} = -4i\sigma_3{\bf Y}_k^{(2)} + \bbm 0 & -2\mathcal{U}_k' \\ -2\mathcal{V}_k' & 0 \ebm, \\ 
\mathcal{O}(\zeta^{-1})\text{ diagonal}: & \quad {\bf D}_k^{(0)} = 4i\bbm 0 & \mathcal{U}_k \\ -\mathcal{V}_k & 0 \ebm{\bf Y}_k^{(2)} + \bbm 0 & -2\mathcal{U}_k' \\ -2\mathcal{V}_k' & 0 \ebm{\bf Y}_k^{(1)}, \\ 
\mathcal{O}(\zeta^{-1})\text{ off-diagonal}: & \quad {\bf Y}_k^{(1)}{\bf D}_k^{(1)} + 2{\bf Y}_k^{(2)}{\bf D}_k^{(2)} + 3{\bf Y}_k^{(3)}{\bf D}_k^{(3)} \\ 
  & \hspace{.5in}= -4i\sigma_3{\bf Y}_k^{(3)} - (2i\mathcal{U}_k\mathcal{V}_k+is)\sigma_3{\bf Y}_k^{(1)}, \\ 
\mathcal{O}(\zeta^{-2})\text{ diagonal}: & \quad -{\bf D}_k^{(-1)} = 4i\bbm 0 & \mathcal{U}_k \\ -\mathcal{V}_k & 0 \ebm{\bf Y}_k^{(3)} + \bbm 0 & -2\mathcal{U}_k' \\ -2\mathcal{V}_k' & 0 \ebm{\bf Y}_k^{(2)}.
\end{split}
\endeq
Solving these equations sequentially yields
\eq
\label{L-coefficients}
\begin{split}
{\bf D}_k^{(3)} = -\frac{4}{3}i\sigma_3, \quad {\bf D}_k^{(2)}={\bf 0}&, \quad  {\bf D}_k^{(1)} = -is\sigma_3, \\ {\bf D}_k^{(0)} = (\mathcal{U}_k\mathcal{V}_k'-\mathcal{V}_k\mathcal{U}_k')\sigma_3, \quad {\bf D}_k^{(-1)} = &\frac{i}{2}(\mathcal{U}_k^2\mathcal{V}_k^2+s\mathcal{U}_k\mathcal{V}_k + \mathcal{U}_k'\mathcal{V}_k')\sigma_3, \\ 
{\bf Y}_k^{(1)} = \bbm 0 & \frac{1}{2}\mathcal{U}_k \\ \frac{1}{2}\mathcal{V}_k & 0 \ebm, \quad {\bf Y}_k^{(2)} = \bbm 0 & \frac{i}{4}\mathcal{U}_k' \\ -\frac{i}{4}\mathcal{V}_k' & 0 \ebm, & \quad {\bf Y}_k^{(3)} = -\frac{1}{8} \bbm 0 & \mathcal{U}_k^2\mathcal{V}_k+s\mathcal{U}_k \\ \mathcal{U}_k\mathcal{V}_k^2+s\mathcal{V}_k & 0 \ebm.
\end{split}
\endeq
Combining \eqref{L-def} and \eqref{Z-expansion} gives 
\eq
{\bf L}_k(\zeta;s)e^{i(\frac{4}{3}\zeta^3+s\zeta)\sigma_3}\zeta^{-k\sigma_3} = \mathbb{I} + \frac{{\bf A}_k(s)}{\zeta} + \frac{{\bf B}_k(s)}{\zeta^2} + \mathcal{O}\left(\frac{1}{\zeta^3}\right).
\endeq
Using \eqref{L-expansion},  we see 
\eq
{\bf A}_k = {\bf Y}_k^{(1)} + {\bf D}_k^{(-1)}, \quad {\bf B}_k = {\bf Y}_k^{(1)}{\bf D}_k^{(-1)} + {\bf Y}_k^{(2)} + \frac{1}{2}\left({\bf D}_k^{(-1)}\right)^2 + {\bf D}_k^{(-2)}.
\endeq
Applying \eqref{L-coefficients}, 
\eq
\label{Ak-Bk}
\begin{split}
{\bf A}_k = & \frac{1}{2}\bbm i(\mathcal{U}_k^2\mathcal{V}_k^2+s\mathcal{U}_k\mathcal{V}_k + \mathcal{U}_k'\mathcal{V}_k') & \mathcal{U}_k \\ \mathcal{V}_k & -i(\mathcal{U}_k^2\mathcal{V}_k^2+s\mathcal{U}_k\mathcal{V}_k + \mathcal{U}_k'\mathcal{V}_k') \ebm,\\
{\bf B}_k = & \frac{i}{4}\bbm 0 & \mathcal{U}_k'-\mathcal{U}_k(\mathcal{U}_k^2\mathcal{V}_k^2+s\mathcal{U}_k\mathcal{V}_k+\mathcal{U}_k'\mathcal{V}_k') \\ -\mathcal{V}_k'+\mathcal{V}_k(\mathcal{U}_k^2\mathcal{V}_k^2+s\mathcal{U}_k\mathcal{V}_k+\mathcal{U}_k'\mathcal{V}_k')  & 0 \ebm \\ 
  &  + \text{(diagonal)}.
\end{split}
\endeq
We will not need the diagonal terms in ${\bf B}_k$.

\subsubsection{Differential equations for the logarithmic derivatives}
We note that 
\eq
\label{lambda-def}
\lambda_k:=\mathcal{U}_k(s)\mathcal{V}_k'(s) - \mathcal{V}_k(s)\mathcal{U}_k'(s)
\endeq
is an $s$-independent quantity.  To see this, simply multiply the second  
equation in \eqref{coupled-PII} by $\mathcal{U}_k(s)$, multiply the first 
equation by $\mathcal{V}_k(s)$, and subtract:
\eq
\mathcal{U}_k(s)\mathcal{V}_k''(s) - \mathcal{V}_k(s)\mathcal{U}_k''(s) = 0.
\endeq
This is equivalent to 
\eq
\frac{d}{ds}(\mathcal{U}_k(s)\mathcal{V}_k'(s) - \mathcal{V}_k(s)\mathcal{U}_k'(s)) = 0.
\endeq
To see exactly what $\lambda_k$ is in terms of $k$, we recall from 
\eqref{L-coefficients} that the same combination 
$\mathcal{U}_k\mathcal{V}_k'-\mathcal{V}_k\mathcal{U}_k'$ appears in 
${\bf D}_k^{(0)}$:
\eq
{\bf D}_k^{(0)} = \lambda_k\sigma_3.
\endeq
From the expansion \eqref{L-expansion}  for ${\bf L}_k(\zeta)$ and the 
expressions for the coefficients \eqref{L-coefficients}, we see
\eq
{\bf L}_k(\zeta) e^{(\frac{4}{3}i\zeta^3 + is)\sigma_3}\zeta^{-\lambda_k\sigma_3} = \mathbb{I} + \mathcal{O}\left(\frac{1}{\zeta}\right).
\endeq
Using the definition of ${\bf L}_k(\zeta)$ in \eqref{L-def}, this gives 
\eq
{\bf Z}_k(\zeta)\zeta^{-\lambda_k\sigma_3} = \mathbb{I} + \mathcal{O}\left(\frac{1}{\zeta}\right).
\endeq
Comparing with the expansion \eqref{Z-expansion} for ${\bf Z}_k(\zeta)$, we 
obtain
\eq
\label{lambda-equals-k}
\lambda_k \equiv k.
\endeq
Next, by using the first equation in \eqref{coupled-PII} to express 
$\mathcal{V}_k(s)$ in terms of $\mathcal{U}_k(s)$, a direct calculation shows
\eq
\frac{1}{2}\frac{d^2}{ds^2}\left(\frac{\mathcal{U}_k'(s)}{\mathcal{U}_k(s)}\right) = \left(\frac{\mathcal{U}_k'(s)}{\mathcal{U}_k(s)}\right)^3 - s \frac{\mathcal{U}_k'(s)}{\mathcal{U}_k(s)} + \frac{1}{2} +\mathcal{U}_k(s)\mathcal{V}_k'(s) - \mathcal{V}_k(s)\mathcal{U}_k'(s).
\endeq
Similarly, using the second equation in \eqref{coupled-PII} to express 
$\mathcal{U}_k(s)$ in terms of $\mathcal{V}_k(s)$, 
\eq
\frac{1}{2}\frac{d^2}{ds^2}\left(\frac{\mathcal{V}_k'(s)}{\mathcal{V}_k(s)}\right) = \left(\frac{\mathcal{V}_k'(s)}{\mathcal{V}_k(s)}\right)^3 - s \frac{\mathcal{V}_k'(s)}{\mathcal{V}_k(s)} + \frac{1}{2} - (\mathcal{U}_k(s)\mathcal{V}_k'(s) - \mathcal{V}_k(s)\mathcal{U}_k'(s)).
\endeq
Thus, from \eqref{lambda-def} and \eqref{lambda-equals-k}, the logarithmic 
derivatives $p_k(s):=\mathcal{U}_k'(s)/\mathcal{U}_k(s)$ and 
$q_k(s):=\mathcal{V}_k'(s)/\mathcal{V}_k(s)$ (cf. \eqref{log-derivatives}) 
satisfy the \emph{uncoupled} inhomogeneous Painlev\'e-II equations 
\eq
\frac{1}{2}p_k''(s) = p_k(s)^3 - sp_k(s) + \frac{1}{2} + k, \quad \frac{1}{2}q_k''(s) = q_k(s)^3 - sq_k(s) + \frac{1}{2} - k.
\endeq
By scaling $P_k(x):=2^{-1/3}p_k(-2^{-1/3}x)$ and 
$Q_k(x):=2^{-1/3}q_k(-2^{-1/3}x)$ (cf. \eqref{scaled-P-Q}),
we can bring these into the standard Painlev\'e-II form matching \eqref{PII}:
\eq
\label{uncoupled-PII}
P_k''(x) = 2P_k(x)^3 + xP_k(x) + \frac{1}{2} + k, \quad Q_k''(x) = 2Q_k(x)^3 + xQ_k(x) + \frac{1}{2} - k.
\endeq

\subsubsection{Schlesinger and B\"acklund transformations}
Fix $k$ and assume the functions ${\bf Z}_k(\zeta;s)$, $\mathcal{U}_k(s)$, 
and $\mathcal{V}_k(s)$ are known.  Then it is possible to obtain 
${\bf Z}_{k\pm 1}(\zeta;s)$, $\mathcal{U}_{k\pm 1}(s)$, and 
$\mathcal{V}_{k\pm 1}(s)$.  The maps 
${\bf Z}_k(\zeta;s)\to{\bf Z}_{k\pm 1}(\zeta;s)$ are called \emph{Schlesinger 
transformations}, while the maps 
$\{\mathcal{U}_k(s),\mathcal{V}_k(s)\}\to\{\mathcal{U}_{k\pm 1}(s),\mathcal{V}_{k\pm 1}(s)\}$
are called \emph{B\"acklund transformations}.  We begin with the ansatz
\eq
\label{Schlesinger-ansatz}
{\bf Z}_{k+1}(\zeta;s) = ({\mathfrak Q}_k(s)\zeta+{\mathfrak R}_k(s)){\bf Z}_k(\zeta;s)
\endeq
and determine the matrices ${\mathfrak Q}_k(s)$ and ${\mathfrak R}_k(s)$.  
Inserting \eqref{Z-expansion} into \eqref{Schlesinger-ansatz} gives
\eq
(\mathfrak{Q}_k\zeta+\mathfrak{R}_k)\left(\mathbb{I} + \frac{{\bf A}_k}{\zeta} + \frac{{\bf B}_k}{\zeta^2} + \mathcal{O}\left(\frac{1}{\zeta^3}\right)\right)\zeta^{-\sigma_3} = \mathbb{I} + \mathcal{O}\left(\frac{1}{\zeta}\right).
\endeq
Grouping terms in each power of $\zeta$ gives
\eq
\begin{split}
\mathcal{O}(\zeta^2): & \quad \mathfrak{Q}_k\bbm 0 & 0 \\ 0 & 1 \ebm = {\bf 0},\\
\mathcal{O}(\zeta): & \quad (\mathfrak{Q}_k{\bf A}_k+\mathfrak{R}_k)\bbm 0 & 0 \\ 0 & 1 \ebm = {\bf 0},\\
\mathcal{O}(1): & \quad \mathfrak{Q}_k\bbm 1 & 0 \\ 0 & 0 \ebm + (\mathfrak{Q}_k{\bf B_k}+\mathfrak{R}_k{\bf A_k})\bbm 0 & 0 \\ 0 & 1 \ebm = \mathbb{I}. 
\end{split}
\endeq
The $\mathcal{O}(\zeta^2)$ equation gives 
$[\mathfrak{Q}_k]_{12}=[\mathfrak{Q}_k]_{22}=0$, while the first column of the 
$\mathcal{O}(1)$ equation gives $[\mathfrak{Q}_k]_{11}=1$ and 
$[\mathfrak{Q}_k]_{21}=0$.  Using this in the $\mathcal{O}(\zeta)$ equation 
shows $[\mathfrak{R}_k]_{12}=-[{\bf A}_k]_{12}$ and $[\mathfrak{R}_k]_{22}=0$.  
Then the second column of the $\mathcal{O}(1)$ equations shows 
$[\mathfrak{R}_k]_{21}=[{\bf A}_k]_{12}^{-1}$ and 
$[\mathfrak{R}_k]_{11} = [{\bf A}_k]_{22} - [{\bf A}_k]_{12}^{-1}[{\bf B_k}]_{12}$.  
Along with the expressions for ${\bf A}_k$ and ${\bf B}_k$ in 
\eqref{Ak-Bk}, we have
\eq
\mathfrak{Q}_k = \bbm 1 & 0 \\ 0 & 0 \ebm, \quad \mathfrak{R}_k = \bbm \displaystyle -\frac{i}{2}\mathcal{U}_k'\mathcal{U}_k^{-1} & \displaystyle -\frac{1}{2}\mathcal{U}_k \\ 2\mathcal{U}_k^{-1} & 0 \ebm.
\endeq
From the Schlesinger transformation we can now obtain the B\"acklund 
transformations.  Using the expansion \eqref{Z-expansion} in the Schlesinger 
transformation \eqref{Schlesinger-ansatz} gives the equation 
\eq
(\mathfrak{Q}_k\zeta+\mathfrak{R}_k)\left(\mathbb{I} + \frac{{\bf A}_k}{\zeta} + \frac{{\bf B}_k}{\zeta^2} + \mathcal{O}\left(\frac{1}{\zeta^3}\right)\right)\zeta^{-\sigma_3} = \mathbb{I} + \frac{{\bf A}_{k+1}}{\zeta} + \frac{{\bf B}_{k+1}}{\zeta^2} + \mathcal{O}\left(\frac{1}{\zeta^3}\right).
\endeq
Reading off the (21)-entry of the $\mathcal{O}(\zeta^{-1})$ term gives 
$[\mathfrak{R}_k]_{21} = [{\bf A}_{k+1}]_{21}$, or 
$2\mathcal{U}_k^{-1}=\frac{1}{2}\mathcal{V}_{k+1}$.  Then we can take the 
second equation in \eqref{coupled-PII} with $k\to k+1$, i.e. 
$\mathcal{V}_{k+1}''=2\mathcal{U}_{k+1}\mathcal{V}_{k+1}^2+s\mathcal{V}_{k+1}$, 
solve for $\mathcal{U}_{k+1}$ in terms of $\mathcal{V}_{k+1}$, and use 
$\mathcal{V}_{k+1}=4\mathcal{U}_k^{-1}$.  We therefore obtain the B\"acklund 
transformation 
\eq
\label{Backlund}
\mathcal{U}_{k+1} = \frac{1}{4}\frac{(\mathcal{U}_k')^2}{\mathcal{U}_k} - \frac{1}{8}\mathcal{U}_k'' - \frac{s}{8}\mathcal{U}_k, \quad \mathcal{V}_{k+1} = \frac{4}{\mathcal{U}_k}. 
\endeq

\subsubsection{Identification of $\mathcal{U}_k$ and $\mathcal{V}_k$}\label{Id-Uk-Vk}
Our next objective is to identify the associated Painlev\'e functions 
$\mathcal{U}_k(s)$ and $\mathcal{V}_k(s)$.  The uncoupled Painlev\'e-II equation
\eqref{PII} is the compatibility condition for the well-studied Flaschka--Newell 
Lax pair 
\cite{FlaschkaN:1980} (see also, for example, 
\cite{DeiftZ:1995,FokasIKN:2006,Liechty:2012,ClaeysKV:2008})
\eq
\label{Flaschka-Newell-Lax-pair}
\begin{split}
\frac{\partial{\bf \Psi}}{\partial s} & = \bbm -i\zeta & u(s) \\ u(s) & i\zeta \ebm {\bf \Psi}, \\
\frac{\partial{\bf \Psi}}{\partial \zeta} & = \bbm -4i\zeta^2-is-2iu(s)^2 & 4\zeta u(s) + 2iu'(s) + \alpha/\zeta \\ 4\zeta u(s) - 2iu'(s) + \alpha/\zeta & 4i\zeta^2 + is + 2iu(s) \ebm {\bf\Psi}.
\end{split}
\endeq
The Riemann--Hilbert problem for this Lax pair has jumps on six semi-infinite rays 
and a pole at the origin of order $\alpha$.  However, at $\alpha=0$ this 
Riemann--Hilbert problem reduces to the Riemann--Hilbert problem for the Jimbo--Miwa 
Lax pair with $k=0$.  We will not need the full Flaschka--Newell Riemann--Hilbert 
problem, but simply note that Riemann--Hilbert Problem \ref{rhp-jm} with $k=0$ 
is a special case.  Indeed, it is well known 
\cite{DeiftZ:1995,FokasIKN:2006,ClaeysKV:2008,Liechty:2012} that the Painlev\'e 
function associated to Riemann--Hilbert Problem \ref{rhp-jm} with $k=0$ is the 
Hastings--McLeod function $u_{_\textnormal{HM}}^{(0)}(s)$ satisfying 
\eqref{PII-hom} with asymptotics \eqref{uhm-gen-plus-inf}.
Now, matching the 
Jimbo--Miwa Lax pair \eqref{Jimbo-Miwa-Lax-pair}  with the Flaschka--Newell Lax pair 
\eqref{Flaschka-Newell-Lax-pair} (with $\alpha=0$) we arrive at 
\eqref{U0-V0-ito-HM}, i.e. 
$\mathcal{U}_0(s) = -iu_{_\textnormal{HM}}^{(0)}(s)$ and 
$\mathcal{V}_0(s) = iu_{_\textnormal{HM}}^{(0)}(s)$.

For $k>0$ we solve the problem iteratively by Schlesinger and B\"acklund 
transformations, relating the functions $\mathcal{U}_k(s)$ and $\mathcal{V}_k(s)$ 
to the generalized Hastings--McLeod functions.  
First, we use the uniformly convergent expansions \eqref{uhm-minus-inf} and 
\eqref{uhm-gen-plus-inf}, the relationship \eqref{U0-V0-ito-HM}, and the B\"acklund 
transformations \eqref{Backlund} to find the asymptotic behavior of 
$\mathcal{U}_1(s)$ and $\mathcal{V}_1(s)$.  
\eq
\label{U1-V1}
\begin{split}
\mathcal{U}_1(s) & = \begin{cases} \displaystyle \frac{-i}{16\sqrt{\pi}s^{3/4}}e^{-\frac{2}{3}s^{3/2}}\left(1+\mathcal{O}\left(\frac{1}{s^{3/4}}\right)\right), & s\to+\infty, \\ \displaystyle -\frac{i(-s)^{3/2}}{8\sqrt{2}}\left(1+\mathcal{O}\left(\frac{1}{(-s)^{3/2}}\right)\right), & s\to-\infty, \end{cases} \\
\mathcal{V}_1(s) & = \begin{cases} \displaystyle 8i\sqrt{\pi}s^{1/4}e^{\frac{2}{3}s^{3/2}}\left(1+\mathcal{O}\left(\frac{1}{s^{3/4}}\right)\right), & s\to+\infty, \\ \displaystyle 4i\sqrt{\frac{2}{-s}}\left(1+\mathcal{O}\left(\frac{1}{(-s)^{3/2}}\right)\right), & s\to-\infty.
\end{cases}
\end{split}
\endeq
A straightforward inductive argument applying the B\"acklund transformations 
\eqref{Backlund} repeatedly to \eqref{U1-V1} shows the asymptotics 
\eqref{Uk-Vk} for $\mathcal{U}_k$ and $\mathcal{V}_k$.  
Then \eqref{log-derivatives} and \eqref{scaled-P-Q} give 
\eq
\label{Pk-Qk-asymptotics}
\begin{split}
P_k(x) & = \begin{cases} \displaystyle -\sqrt{\frac{-x}{2}}\left(1 + \mathcal{O}\left(\frac{1}{(-x)^{3/2}}\right)\right), & x\to-\infty, \\ \displaystyle -\frac{k+\frac{1}{2}}{x}\left(1 + \mathcal{O}\left(\frac{1}{x^{3/2}}\right)\right), & x\to+\infty,\end{cases} \\
Q_k(x) & = \begin{cases} \displaystyle \sqrt{\frac{-x}{2}}\left(1 + \mathcal{O}\left(\frac{1}{(-x)^{3/2}}\right)\right), & x\to-\infty, \\ \displaystyle \frac{k-\frac{1}{2}}{x}\left(1 + \mathcal{O}\left(\frac{1}{x^{3/2}}\right)\right), & x\to+\infty.\end{cases}
\end{split}
\endeq
From \eqref{uncoupled-PII} and \eqref{Pk-Qk-asymptotics}, $(-P_k(x))$ satisfies 
the inhomogeneous Painlev\'e-II equation \eqref{PII} and both boundary 
conditions \eqref{uhm-minus-inf}--\eqref{uhm-plus-inf} with 
$\alpha=k+\frac{1}{2}$.  
Similarly, $Q_k(x)$ 
satisfies \eqref{PII} and \eqref{uhm-minus-inf}--\eqref{uhm-plus-inf} with 
$\alpha=k-\frac{1}{2}$.  Therefore, we can identify $P_k(x)$ and $Q_k(x)$ 
in terms of generalized Hastings--McLeod solutions as 
$P_k(x) \equiv -u_{_\textnormal{HM}}^{(k+\frac{1}{2})}(x)$ and 
$Q_k(x) \equiv u_{_\textnormal{HM}}^{(k-\frac{1}{2})}(x)$ (cf. 
\eqref{Pk-Qk-ito-hm}).

\subsection{Airy parametrices}
It remains to build the inner parametrices ${\bf M}_k^{(a)}(z)$ and 
${\bf M}_k^{(b)}(z)$ in the open disks $\mathbb{D}_a$ and $\mathbb{D}_b$.  This 
construction is standard (see, for instance, 
\cite{DeiftKMVZ:1999,BleherL:2011,BertolaB:2014}) and is done in terms of 
Airy functions.  The important fact for us is that the Airy parametrices can 
be matched to the outer parametrix with an $\mathcal{O}(n^{-1})$ error (as 
opposed to the parametrix at $i\mu$, where the error is 
$\mathcal{O}(n^{-1/3})$).  This means the explicit form of the parametrix 
only comes into play when computing quantities of $\mathcal{O}(n^{-1})$.  
Since we limit ourselves to computing $\mathcal{O}(n^{-1/3})$ terms, it is 
sufficient for our purposes to note the existence and uniqueness of functions 
satisfying the following problem.  
\begin{rhp}[The Airy parametrix Riemann--Hilbert problems]
\label{rhp-airy}
Fix a positive integer $k$.  For either $*=a$ or $*=b$, determine a 
$2\times 2$ matrix-valued function ${\bf M}_k^{(*)}(z)$ satisfying:
\begin{itemize}
\item[]{\bf Analyticity:} ${\bf M}_k^{(*)}(z)$ is analytic in 
$\mathbb{D}_*\backslash\Sigma^{\rm crit}$.  In each wedge the solution can be 
analytically continued into a larger wedge, and is H\"older continuous up to 
the boundary in a neighborhood of $z=*$.
\item[]{\bf Normalization:}  
\eq
\label{airy-normalization}
{\bf M}_k^{(*)}(z) = {\bf M}_k^{(\rm{out})}(z)(\mathbb{I}+\mathcal{O}(n^{-1})), \quad z\in\partial\mathbb{D}_*.
\endeq
\item[]{\bf Jump condition:}  For $z\in\mathbb{D}_*\cap\Sigma^{\rm crit}$, 
${\bf M}_k^{(*)}(z)$ satisfies 
${\bf M}_{k+}^{(*)}(z) = {\bf M}_{k-}^{(*)}(z){\bf V}^{\rm crit}(z)$.
\end{itemize}
\end{rhp}

\section{Error analysis}
\label{sec-error}

\subsection{The error problem}

We can now define
\eq
\label{Mk-def}
{\bf M}^{(k)}(z):=\begin{cases} {\bf M}_k^\text{(out)}(z), & z\in\mathbb{C}\backslash\{\mathbb{D}_{i\mu}\cup\mathbb{D}_a\cup\mathbb{D}_b\}, \\ {\bf M}_k^{(i\mu)}(\zeta(z)), & z\in \mathbb{D}_{i\mu}, \\ {\bf M}_k^{(a)}(z), & z\in \mathbb{D}_a, \\ {\bf M}_k^{(b)}(z), & z\in \mathbb{D}_b.  \end{cases}
\endeq
Then the error function is 
\eq
\label{X-crit-def}
{\bf X}_n^{(k)}(z):={\bf S}_n^\text{crit}(z){\bf M}^{(k)}(z)^{-1},
\endeq
and it satisfies the jump condition
\eq
\label{X-crit-jump}
{\bf X}_{n+}^{(k)}(z) = {\bf X}_{n-}^{(k)}(z) {\bf V}^{({\bf X})}(z)
\endeq
on the jump contours $\Sigma^{({\bf X})}$ shown in 
Figure \ref{fig-crit-error-jumps}.
From Lemma \ref{lemma:Scrit-jumps} and \eqref{airy-normalization}, we have 
\eq
\label{VX-bounds}
{\bf V}^{({\bf X})}(z) = \mathbb{I}+\mathcal{O}(n^{-1}), \quad z\in\Sigma^{({\bf X})}\backslash\partial\mathbb{D}_{i\mu}.
\endeq
The jump on $\partial\mathbb{D}_{i\mu}$ will be analyzed in 
\S\ref{subsec-error-parametrix}.
\begin{figure}[h]
\setlength{\unitlength}{1.5pt}
\begin{center}
\begin{picture}(100,35)(-50,28)
\thicklines
\put(-80,65){\line(1,0){160}}
\put(-80,45){\line(1,0){22}}
\put(58,45){\line(1,0){22}}
\put(-80,25){\line(1,0){160}} 
\put(30,65){\line(-3,-2){23}}
\put(30,25){\line(-3,2){23}}
\put(-30,65){\line(3,-2){23}}
\put(-30,25){\line(3,2){23}}
\multiput(-80,41)(4,0){40}{\line(1,0){2}}
\put(-50,25){\line(0,1){12}}
\put(-50,53){\line(0,1){12}}
\put(50,25){\line(0,1){12}}
\put(50,53){\line(0,1){12}}
\put(82,38.5){$\mathbb{R}$}
\put(82,63){$\mathbb{R}+i(\mu+\epsilon)$}
\put(82,23){$\mathbb{R}-i\delta$}
\put(-50,45){\circle{15}}
\put(0,45){\circle{15}}
\put(50,45){\circle{15}}
\put(-70,52){$\partial\mathbb{D}_a$}
\put(56,52){$\partial\mathbb{D}_b$}
\put(-7,55){$\partial\mathbb{D}_{i\mu}$}
\put(1,65){\vector(-1,0){5}}
\put(66,45){\vector(1,0){5}}
\put(-70,45){\vector(1,0){5}}
\put(67,65){\vector(-1,0){5}}
\put(62,25){\vector(1,0){5}}
\put(-64,65){\vector(-1,0){5}}
\put(-68,25){\vector(1,0){5}}
\put(-4,25){\vector(1,0){5}}
\put(50,56){\vector(0,1){5}}
\put(50,28){\vector(0,1){5}}
\put(-50,61){\vector(0,-1){5}}
\put(-50,33){\vector(0,-1){5}}
\put(-44,50){\vector(1,-1){1}}
\put(45.5,51.5){\vector(1,1){1}}
\put(1,52.8){\vector(1,0){1}}
\put(-38,65){\vector(-1,0){5}}
\put(42,65){\vector(-1,0){5}}
\put(-42,25){\vector(1,0){5}}
\put(38,25){\vector(1,0){5}}
\put(30,65){\vector(-3,-2){15}}
\put(-30,25){\vector(3,2){12}}
\put(-15,55){\vector(-3,2){5}}
\put(15,35){\vector(3,-2){8}}
\end{picture}
\end{center}
\caption{\label{fig-crit-error-jumps} The jump contours $\Sigma^{({\bf X})}$ for the error problem along with their orientations.  The dotted line is the real axis.}
\end{figure}
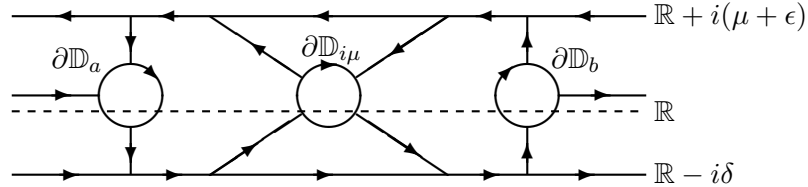
The error function has the expansion
\eq
{\bf X}_n^{(k)}(z) = \mathbb{I} + \frac{{\bf X}_{n,1}^{(k)}}{z} + \mathcal{O}\left(\frac{1}{z^2}\right), \quad z\to\infty,
\endeq
where ${\bf X}_{n,1}^{(k)}$ is independent of $z$.

\subsection{Parametrix for the error}
\label{subsec-error-parametrix}
The jump matrix ${\bf V}^{({\bf X})}(z)$ is well controlled (i.e. 
close to the identity asymptotically as $n\to\infty$) for 
$z\in\Sigma^{({\bf X})}\backslash\partial\mathbb{D}_{i\mu}$.  However, 
we will now see that, for certain $\mu$, the jump on $\partial\mathbb{D}_{i\mu}$ is 
not asymptotically small.  This will necessitate approximating ${\bf X}_n^{(k)}(z)$ 
by its own parametrix ${\bf Y}_n^{(k)}(z)$.  We begin by calculating the jump on 
this circle.  
From \eqref{X-crit-jump} and \eqref{X-crit-def}, and using the fact that 
${\bf S}_n^\text{crit}(z)$ has no jump on $\partial\mathbb{D}_{i\mu}$,
\eq
\left.{\bf V}^{({\bf X})}(z)\right\vert_{\partial\mathbb{D}_{i\mu}} = {\bf M}_k^{(i\mu)}(\zeta(z)){\bf M}_k^\text{(out)}(z)^{-1}.
\endeq
Now, using \eqref{E-k-def} and \eqref{M-k-imu}, and then \eqref{Z-expansion}, 
\eq
\label{VX-expansion}
\begin{split}
\left.{\bf V}^{({\bf X})}(z)\right\vert_{\partial\mathbb{D}_{i\mu}} & = {\bf E}_k(z)\sigma_3e^{i\theta\sigma_3}n^{-k\sigma_3/3}{\bf Z}_k(\zeta(z))\zeta(z)^{-k\sigma_3}n^{k\sigma_3/3}e^{-i\theta\sigma_3}\sigma_3{\bf E}_k(z)^{-1} \\
  & = {\bf E}_k(z)\sigma_3e^{i\theta\sigma_3}n^{-k\sigma_3/3} \left(\mathbb{I} + \frac{{\bf A}_k(s)}{\zeta(z)} + \mathcal{O}\left(\frac{1}{\zeta(z)^2}\right)\right) n^{k\sigma_3/3}e^{-i\theta\sigma_3}\sigma_3{\bf E}_k(z)^{-1}.
\end{split}
\endeq
Using the fact that ${\bf E}_k(z)$ is bounded as $n\to\infty$ and the fact that 
$\zeta=\mathcal{O}(n^{1/3})$ (see \eqref{zeta-growth-in-n}) shows 
\eq\label{D-crit-jump-est}
\left.{\bf V}^{({\bf X})}(z)\right\vert_{\partial\mathbb{D}_{i\mu}} = \mathbb{I} + \bbm 0 & \mathcal{O}(e^{2n\pi\mu}n^{-(2k+1)/3}) \\ \mathcal{O}(e^{-2n\pi\mu}n^{(2k-1)/3}) & 0 \ebm + \mathcal{O}\left(\frac{1}{n^{1/3}}\right).
\endeq
To make these jumps bounded as $n\to\infty$ we require $k$ to be chosen so 
\eq
\label{k-condition}
\frac{1}{3\pi}\left(k-\frac{1}{2}\right)\frac{\log n}{n} < \mu \le \frac{1}{3\pi}\left(k+\frac{1}{2}\right)\frac{\log n}{n}.
\endeq
Note that if $\mu=\frac{k\log n}{3\pi n}$, i.e., it is in the center of the interval \eqref{k-condition}, then the jump \eqref{D-crit-jump-est} is $\bigO(n^{-1/3})$. The error increases as $\mu$ moves away from the center, and at the endpoints $\mu =\frac{1}{3\pi}\left(k\pm \frac{1}{2}\right) \frac{\log n}{n}$, the error function ${\bf X}_n^{(k)}$ no longer satisfies a 
small-norm Riemann--Hilbert problem.  To circumvent this problem and to give a uniform error for $\mu$ throughout the interval \eqref{k-condition} we build a 
\emph{parametrix for the error}.  

Define
\eq
\label{Qpm-def}
Q_+(z) \equiv Q_+(z;s,\mu,\tau,n) = \frac{\mathcal{U}_k(s)}{2\zeta(z)}\cdot\frac{e^{2i\theta}}{n^{2k/3}}, \quad Q_-(z) \equiv Q_-(z;s,\mu,\tau,n) = \frac{\mathcal{V}_k(s)}{2\zeta(z)}\cdot\frac{n^{2k/3}}{e^{2i\theta}}
\endeq
for $z\in\mathbb{D}_{i\mu}$.  We record that, at worst,
\eq
\label{size-of-Q}
\begin{split}
Q_+(z) & = \begin{cases} 
\mathcal{O}(n^{-1/3}), & \displaystyle \left(k-\frac{1}{2}\right)\frac{\log n}{3\pi n} \le \mu \le k\frac{\log n}{3\pi n}, 
\\ 
\mathcal{O}(1), & \displaystyle k\frac{\log n}{3\pi n} < \mu \le \left(k+\frac{1}{2}\right)\frac{\log n}{3\pi n}, 
\end{cases} \\
Q_-(z) & = \begin{cases} 
\mathcal{O}(1), & \displaystyle \left(k-\frac{1}{2}\right)\frac{\log n}{3\pi n} \le \mu < k\frac{\log n}{3\pi n}, 
\\ 
\mathcal{O}(n^{-1/3}), & \displaystyle k\frac{\log n}{3\pi n} \le \mu \le \left(k+\frac{1}{2}\right)\frac{\log n}{3\pi n}.
\end{cases}
\end{split}
\endeq
We can now rewrite \eqref{VX-expansion} as
\eq
\label{VX-ito-Q}
\left.{\bf V}^{({\bf X})}(z)\right\vert_{\partial\mathbb{D}_{i\mu}} = \mathbb{I} - {\bf E}_k(z) \bbm 0 & Q_+(z) \\ Q_-(z) & 0 \ebm {\bf E}_k(z)^{-1} + \mathcal{O}\left(\frac{1}{n^{1/3}}\right).
\endeq
Here we have used that $k$ is given so \eqref{k-condition} is satisfied in order to 
bound the terms proportional to $\zeta^{-2}$.  If we discard all 
$\mathcal{O}(n^{-1/3})$ terms we are led to the following model problem.
\begin{rhp}[The parametrix for the error]
\label{rhp-error-parametrix}
Determine the $2\times 2$ matrix-valued function ${\bf Y}_n^{(k)}(z)$ satisfying:
\begin{itemize}
\item[]{\bf Analyticity:} ${\bf Y}_n^{(k)}(z)$ is analytic for 
$z\in\mathbb{C}\backslash\partial\mathbb{D}_{i\mu}$ with 
H\"older-continuous boundary values on $\partial\mathbb{D}_{i\mu}$.
\item[]{\bf Normalization:}  
\eq
{\bf Y}_n^{(k)}(z) = \mathbb{I}+\mathcal{O}\left(\frac{1}{z}\right) \text{ as } z\to\infty.
\endeq
\item[]{\bf Jump condition:}  Orienting the jump contour negatively, the 
solution satisfies 
\eq
\label{Y-jump}
{\bf Y}_{n+}^{(k)}(z)={\bf Y}_{n-}^{(k)}(z){\bf V}^{({\bf Y})}(z) \equiv {\bf Y}_{n-}^{(k)}(z) \left(\mathbb{I} - {\bf E}_k(z) {\bf Q}(z) {\bf E}_k(z)^{-1} \right), \quad z\in\partial\mathbb{D}_{i\mu},
\endeq
where
\eq
\label{Q-def}
{\bf Q}(z):= \begin{cases} \bbm 0 & 0 \\ Q_-(z) & 0 \ebm, & \displaystyle \left(k-\frac{1}{2}\right)\frac{\log n}{3\pi n} < \mu \le k\frac{\log n}{3\pi n}, \vspace{.05in} \\ \bbm 0 & Q_+(z) \\ 0 & 0 \ebm, & \displaystyle k\frac{\log n}{3\pi n} < \mu \le \left(k+\frac{1}{2}\right)\frac{\log n}{3\pi n}. \end{cases}
\endeq
\end{itemize}
\end{rhp}
Before we solve this Riemann--Hilbert problem explicitly, we define the ratio 
\eq
{\bf Z}_n^{(k)}(z) := {\bf X}_n^{(k)}(z){\bf Y}_n^{(k)}(z)^{-1}
\endeq
and note that the jump ${\bf V}^{({\bf Z})}(z)$ for this function is uniformly 
close to the identity for large $n$ on its jump contour $\Sigma^{({\bf Z})}=\Sigma^{({\bf X})}$.  
Indeed, the jumps for $z\notin\partial\mathbb{D}_{i\mu}$ are 
controlled since the jumps for ${\bf X}_n^{(k)}(z)$ are (and ${\bf Y}_n^{(k)}(z)$ 
is continuous there), while at the same time \eqref{VX-ito-Q} and \eqref{Y-jump} 
give
\eq
\label{VZ-formula}
\begin{split}
\left.{\bf V}^{({\bf Z})}(z)\right\vert_{\partial\mathbb{D}_{i\mu}} & = {\bf Z}_{n-}^{(k)}(z)^{-1}{\bf Z}_{n+}^{(k)}(z) \\ 
  & = {\bf Y}_{n-}^{(k)}(z){\bf X}_{n-}^{(k)}(z)^{-1}{\bf X}_{n+}(z){\bf Y}_{n+}^{(k)}(z)^{-1} \\ 
  & = {\bf Y}_{n-}^{(k)}(z){\bf V}^{({\bf X})}(z){\bf V}^{({\bf Y})}(z)^{-1}{\bf Y}_{n-}^{(k)}(z)^{-1} \\ 
  & = \mathbb{I} - {\bf Y}_{n-}^{(k)}(z){\bf E}_k(z) \widehat{\bf Q}(z){\bf E}_k(z)^{-1}{\bf Y}_{n-}^{(k)}(z)^{-1} + \mathcal{O}\left(\frac{1}{n}\right),
\end{split}
\endeq
wherein
\eq
\label{Qhat-def}
\widehat{\bf Q}(z) := \begin{cases} \bbm 0 & Q_+(z) \\ 0 & 0 \ebm, & \displaystyle \left(k-\frac{1}{2}\right)\frac{\log n}{3\pi n} < \mu \le k\frac{\log n}{3\pi n}, \vspace{.05in} \\ \bbm 0 & 0 \\ Q_-(z) & 0 \ebm, & \displaystyle k\frac{\log n}{3\pi n} < \mu \le \left(k+\frac{1}{2}\right)\frac{\log n}{3\pi n}. \end{cases}
\endeq
To invert ${\bf V}^{({\bf Y})}(z)$ we used the fact that ${\bf Q}(z)$ is 
nilpotent. Equation \eqref{size-of-Q} shows us that 
$\widehat{\bf Q}(z)=\mathcal{O}(n^{-1/3})$ (assuming the correct choice of $k$), 
and so 
\eq
\label{VZ-on-Dimu}
{\bf V}^{({\bf Z})}(z) = \mathcal{O}(n^{-1/3}), \quad z\in\partial\mathbb{D}_{i\mu}.
\endeq
Away from $\partial\mathbb{D}_{i\mu}$, \eqref{VX-bounds} gives the stronger 
bound 
\eq
\label{VZ-off-Dimu}
{\bf V}^{({\bf Z})}(z) = \mathcal{O}(n^{-1}), \quad z\in\Sigma^{({\bf Z})}\backslash\partial\mathbb{D}_{i\mu},
\endeq
with the largest contribution coming from $\partial\mathbb{D}_a$ and 
$\partial\mathbb{D}_b$.

\subsubsection{Solution of Riemann--Hilbert Problem \ref{rhp-error-parametrix} for 
the error parametrix ${\bf Y}_n^{(k)}(z)$.}
Guided by the calculations in \cite[\S3.5.2]{BuckinghamM:2015}, we now solve 
Riemann--Hilbert Problem \ref{rhp-error-parametrix} exactly.  The meromorphic 
continuation of ${\bf Y}_n^{(k)}(z)$ from the exterior to the interior of 
$\mathbb{D}_{i\mu}$ is 
\eq
\label{Ytilde}
{\bf \widetilde{Y}}_n^{(k)}(z) := \begin{cases} {\bf Y}_n^{(k)}(z), & z\in\mathbb{C}\backslash\mathbb{D}_{i\mu}, \\ \displaystyle {\bf Y}_n^{(k)}(z)\left(\mathbb{I} - {\bf E}_k(z) {\bf Q}(z) {\bf E}_k(z)^{-1} \right), & z\in\mathbb{D}_{i\mu}. \end{cases} 
\endeq
Equivalently, we can use the nilpotency of ${\bf Q}$ to write
\eq
\label{Yk}
{\bf Y}_n^{(k)}(z) := \begin{cases} {\bf \widetilde{Y}}_n^{(k)}(z), & z\in\mathbb{C}\backslash\mathbb{D}_{i\mu}, \\ \displaystyle {\bf \widetilde{Y}}_n^{(k)}(z)\left(\mathbb{I} + {\bf E}_k(z) {\bf Q}(z) {\bf E}_k(z)^{-1} \right), & z\in\mathbb{D}_{i\mu}. \end{cases} 
\endeq
This function tends to the identity matrix as $z\to\infty$ (since 
${\bf Y}_n^{(k)}(z)$ does), and is analytic for all $z$ except for a simple pole 
at $z=i\mu$.  Therefore, ${\bf \widetilde{Y}}_n^{(k)}(z)$ necessarily has the 
form
\eq
\label{Ytilde-ansatz}
{\bf \widetilde{Y}}_n^{(k)}(z) = \mathbb{I} + \frac{1}{z-i\mu}{\bf B},
\endeq
where ${\bf B}\equiv{\bf B}(s,\mu,\tau,n)$ is independent of $z$.  It only remains 
to determine ${\bf B}$, which can be done using the fact that ${\bf Y}_n^{(k)}(z)$ 
is analytic at $z=i\mu$.  This analyticity implies 
\eq
\label{Yn-explicit}
{\bf Y}_n^{(k)}(z) = \left( \mathbb{I} + \frac{1}{z-i\mu}{\bf B} \right) \left(\mathbb{I} + {\bf E}_k(z) {\bf Q}(z) {\bf E}_k(z)^{-1} \right) = \mathcal{O}(1) \text{ for } z\to i\mu
\endeq
via \eqref{Ytilde} and \eqref{Ytilde-ansatz}.  
We expand the constituent functions around $z=i\mu$, keeping in mind that 
${\bf Q}(z)$ has a simple pole at $z=i\mu$:
\eq
\label{E-Q-expansions}
{\bf E}_k(z) = {\bf F} + (z-i\mu){\bf G} + \mathcal{O}\left((z-i\mu)^2\right), \quad {\bf Q}(z) = \frac{1}{z-i\mu}{\bf R} + {\bf S} + \mathcal{O}\left((z-i\mu)^2\right).
\endeq
Here the matrices ${\bf F}$, ${\bf G}$, ${\bf R}$, and ${\bf S}$ are independent of 
$z$.  
Plugging \eqref{E-Q-expansions} into \eqref{Yn-explicit} and isolating the Laurent 
terms gives
\eq
\begin{split}
\mathcal{O}\left(\frac{1}{(z-i\mu)^2}\right): \quad & {\bf B}{\bf F}{\bf R}{\bf F}^{-1} = {\bf 0}, \\
\mathcal{O}\left(\frac{1}{z-i\mu}\right): \quad & {\bf B} + {\bf F}{\bf R}{\bf F}^{-1} - {\bf B}{\bf F}{\bf R}{\bf F}^{-1}{\bf G}{\bf F}^{-1} + {\bf B}{\bf F}{\bf S}{\bf F}^{-1} + {\bf B}{\bf G}{\bf R}{\bf F}^{-1} = 0
\end{split}
\endeq
The invertibility of ${\bf F}$ follows from the invertibility of 
${\bf F}_k^{(\text{out})}(z)$ at $z=i\mu$.  The first equation is equivalent to 
${\bf B}{\bf F}{\bf R}={\bf 0}$.  This also implies that 
${\bf B}{\bf F}{\bf S}={\bf 0}$ since ${\bf S}$ is a (nonzero) constant multiple 
of ${\bf R}$.  Using these facts in the second equation gives the simplified system
\eq
{\bf B}{\bf F}{\bf R} = {\bf 0}, \quad {\bf B}{\bf F} + {\bf F}{\bf R} + {\bf B}{\bf G}{\bf R} = 0.
\endeq
Looking at the first equation, we recall that ${\bf R}$ is either strictly 
upper-triangular (in which case the first column of ${\bf B}{\bf F}$ is zero) or 
strictly lower-triangular (in which case the second column of ${\bf B}{\bf F}$ is 
zero).  This can be used along with the second equation and the fact that 
${\bf R}$ is strictly triangular to solve for 
${\bf B}{\bf F}$.  After multiplying on the right by ${\bf F}^{-1}$, we find 
\eq
\label{B-formula}
{\bf B} = \begin{cases} \displaystyle \frac{-{\bf F}{\bf R}{\bf F}^{-1}}{1 + [{\bf F}^{-1}{\bf G}]_{12}[{\bf R}]_{21}}, & \displaystyle \left(k-\frac{1}{2}\right)\frac{\log n}{3\pi n} < \mu \le k\frac{\log n}{3\pi n}, \vspace{.05in} \\ \displaystyle \frac{-{\bf F}{\bf R}{\bf F}^{-1}}{1 + [{\bf F}^{-1}{\bf G}]_{21}[{\bf R}]_{12}}, & \displaystyle k\frac{\log n}{3\pi n} < \mu \le \left(k+\frac{1}{2}\right)\frac{\log n}{3\pi n}. \end{cases}
\endeq
Only the (11)-entry will be necessary to compute the winding number probabilities:
\eq
\label{B11}
[{\bf B}]_{11} = \begin{cases} \displaystyle \frac{-[{\bf F}]_{12}[{\bf F}]_{22}[{\bf R}]_{21}}{1 + [{\bf F}^{-1}{\bf G}]_{12}[{\bf R}]_{21}}, & \displaystyle \left(k-\frac{1}{2}\right)\frac{\log n}{3\pi n} < \mu \le k\frac{\log n}{3\pi n}, \vspace{.05in} \\ \displaystyle \frac{[{\bf F}]_{11}[{\bf F}]_{21}[{\bf R}]_{12}}{1 + [{\bf F}^{-1}{\bf G}]_{21}[{\bf R}]_{12}}, & \displaystyle k\frac{\log n}{3\pi n} < \mu \le \left(k+\frac{1}{2}\right)\frac{\log n}{3\pi n}. \end{cases}
\endeq
We note immediately from \eqref{E-Q-expansions}, \eqref{Q-def}, 
\eqref{Qpm-def}, and \eqref{zeta-prime} that 
\eq
\label{R-expansion}
{\bf R} = \begin{cases} \bbm 0 & 0 \\ \displaystyle \frac{2^{2/3}\mathcal{V}_k(s)}{\pi }\cdot\frac{n^{(2k-1)/3}}{e^{2i\theta}} + \mathcal{O}\left(\frac{1}{n}\right) & 0 \ebm, & \displaystyle \left(k-\frac{1}{2}\right)\frac{\log n}{3\pi n} < \mu \le k\frac{\log n}{3\pi n}, \vspace{.05in} \\ \bbm 0 & \displaystyle \frac{2^{2/3}\mathcal{U}_k(s)}{\pi}\cdot\frac{e^{2i\theta}}{n^{(2k+1)/3}} + \mathcal{O}\left(\frac{1}{n}\right) \\ 0 & 0 \ebm, & \displaystyle k\frac{\log n}{3\pi n} < \mu \le \left(k+\frac{1}{2}\right)\frac{\log n}{3\pi n}. \end{cases}
\endeq
Later we will also need the expansion of $\widehat{\bf Q}$, so we note 
\eq
\label{Qhat-expansion}
\widehat{\bf Q}(z) = \frac{1}{z-i\mu}\widehat{\bf R} + \mathcal{O}(1),
\endeq
where 
\eq
\label{Rhat-expansion}
\widehat{\bf R} := \begin{cases} 
\bbm 0 & \displaystyle \frac{2^{2/3}\mathcal{U}_k(s)}{\pi}\cdot\frac{e^{2i\theta}}{n^{(2k+1)/3}} + \mathcal{O}\left(\frac{1}{n}\right) \\ 0 & 0 \ebm, & \displaystyle \left(k-\frac{1}{2}\right)\frac{\log n}{3\pi n} < \mu \le k\frac{\log n}{3\pi n}, \vspace{.05in} \\ \bbm 0 & 0 \\ \displaystyle \frac{2^{2/3}\mathcal{V}_k(s)}{\pi }\cdot\frac{n^{(2k-1)/3}}{e^{2i\theta}} + \mathcal{O}\left(\frac{1}{n}\right) & 0 \ebm, & \displaystyle k\frac{\log n}{3\pi n} < \mu \le \left(k+\frac{1}{2}\right)\frac{\log n}{3\pi n}. \end{cases}
\endeq
Next we compute $[{\bf F}]_{12}[{\bf F}]_{22}$ and $[{\bf F}]_{11}[{\bf F}]_{21}$.
From 
\eqref{E-k-def} and \eqref{M-k-out},
\eq
\label{E-explicit}
{\bf E}_k(z) = \bbm \displaystyle \frac{\widetilde{\gamma}(z)+\widetilde{\gamma}(z)^{-1}}{2} & \displaystyle \frac{\widetilde{\gamma}(z)-\widetilde{\gamma}(z)^{-1}}{-2i}e^{-ik\pi} \\ \displaystyle \frac{\widetilde{\gamma}(z)-\widetilde{\gamma}(z)^{-1}}{2i}e^{ik\pi} & \displaystyle \frac{\widetilde{\gamma}(z)+\widetilde{\gamma}(z)^{-1}}{2} \ebm \bbm 0 & \displaystyle -\left(\frac{\widetilde{d}(z)\zeta(z)}{n^{1/3}}\right)^k \\ \displaystyle \left(\frac{n^{1/3}}{\widetilde{d}(z)\zeta(z)}\right)^k & 0 \ebm,
\endeq 
where $\widetilde{\gamma}(z)$ and $\widetilde{d}(z)$ are defined for 
$z\in\mathbb{D}_{i\mu}$ to be the analytic continuations of $\gamma(z)$ and 
$d(z)$ from the upper half-plane.  
From Lemma \ref{d-properties-lem} (c), 
\eq
\label{dtilde-zeta}
\widetilde{d}(z)\zeta(z) = \frac{4i\zeta'(i\mu)}{\sqrt{T}} + \mathcal{O}(z-i\mu).
\endeq 
Then \eqref{zeta-prime} and \eqref{T-scaling} show
\eq
\label{zetaprime-ito-n}
\frac{4\zeta'(i\mu)}{\sqrt{T}} = (2n)^{1/3} + \mathcal{O}\left(\frac{1}{n^{1/3}}\right).
\endeq
Combining the previous three equations along with ${\bf F}\equiv{\bf E}_k(i\mu)$ 
and $\widetilde{\gamma}(i\mu)^2=-i$ gives
\eq
\label{F12F22-F11F21}
\begin{split}
[{\bf F}]_{12}[{\bf F}]_{22} & = \frac{\widetilde{\gamma}(i\mu)^2 - \widetilde{\gamma}(i\mu)^{-2}}{4i} 2^{2k/3} + \mathcal{O}\left(\frac{1}{n^{2/3}}\right) = -2^{(2k-3)/3} + \mathcal{O}\left(\frac{1}{n^{2/3}}\right),\\ 
[{\bf F}]_{11}[{\bf F}]_{21} & = \frac{\widetilde{\gamma}(i\mu)^2 - \widetilde{\gamma}(i\mu)^{-2}}{-4i} 2^{-2k/3} + \mathcal{O}\left(\frac{1}{n^{2/3}}\right) = 2^{-(2k+3)/3} + \mathcal{O}\left(\frac{1}{n^{2/3}}\right).
\end{split}
\endeq
Recall the definitions of $R_\mathcal{V}$ and $R_\mathcal{U}$ in 
\eqref{RV-RU-def}.  Then, from \eqref{R-expansion} and \eqref{F12F22-F11F21}, 
\eq
\label{FFR}
\begin{split}
[{\bf F}]_{12}[{\bf F}]_{22}[{\bf R}]_{21} & = -R_\mathcal{V}e^{-2i\pi\tau} + \mathcal{O}\left(\frac{1}{n^{2/3}}\right), \quad \left(k-\frac{1}{2}\right)\frac{\log n}{3\pi n} < \mu \le k\frac{\log n}{3\pi n}, \\
[{\bf F}]_{11}[{\bf F}]_{21}[{\bf R}]_{12} & = R_\mathcal{U}e^{2i\pi\tau} + \mathcal{O}\left(\frac{1}{n^{2/3}}\right), \quad k\frac{\log n}{3\pi n} < \mu \le \left(k+\frac{1}{2}\right)\frac{\log n}{3\pi n}.
\end{split} 
\endeq
To find $[{\bf F}^{-1}{\bf G}]_{12}$ and $[{\bf F}^{-1}{\bf G}]_{21}$, we 
differentiate \eqref{E-explicit} and use 
\eqref{dtilde-zeta}--\eqref{zetaprime-ito-n} to write 
\eq
\begin{split}
{\bf G} \equiv & \left.\frac{d{\bf E}_k(z)}{dz}\right|_{z=i\mu} \\
   = & \frac{\widetilde{\gamma}'(i\mu)}{\widetilde{\gamma}(i\mu)}\bbm \displaystyle \frac{\widetilde{\gamma}(i\mu)-\widetilde{\gamma}(i\mu)^{-1}}{2} & \displaystyle \frac{\widetilde{\gamma}(i\mu)+\widetilde{\gamma}(i\mu)^{-1}}{-2ie^{ik\pi}} \\ \displaystyle \frac{\widetilde{\gamma}(i\mu)+\widetilde{\gamma}(i\mu)^{-1}}{2ie^{-ik\pi}} & \displaystyle \frac{\widetilde{\gamma}(i\mu)-\widetilde{\gamma}(i\mu)^{-1}}{2} \ebm \bbm 0 & \displaystyle -\left(\frac{4i\zeta'(i\mu)}{n^{1/3}\sqrt{T}}\right)^k \\ \displaystyle \left(\frac{n^{1/3}\sqrt{T}}{4i\zeta'(i\mu)}\right)^k & 0 \ebm \\
   & + \bbm \displaystyle \frac{\widetilde{\gamma}(i\mu)+\widetilde{\gamma}(i\mu)^{-1}}{2} & \displaystyle \frac{\widetilde{\gamma}(i\mu)-\widetilde{\gamma}(i\mu)^{-1}}{-2ie^{ik\pi}} \\ \displaystyle \frac{\widetilde{\gamma}(i\mu)-\widetilde{\gamma}(i\mu)^{-1}}{2ie^{-ik\pi}} & \displaystyle \frac{\widetilde{\gamma}(i\mu)+\widetilde{\gamma}(i\mu)^{-1}}{2} \ebm  \frac{d}{dz}\left.\bbm 0 & \displaystyle -\left(\frac{\widetilde{d}(z)\zeta(z)}{n^{1/3}}\right)^k \\ \displaystyle \left(\frac{n^{1/3}}{\widetilde{d}(z)\zeta(z)}\right)^k & 0 \ebm \right|_{z=i\mu}.
\end{split}
\endeq
We will not need the explicit form of the final matrix, only that it is 
off-diagonal.  Explicit calculation along with \eqref{T-scaling} gives
\eq
\label{log-deriv-gamma}
\frac{\widetilde{\gamma}'(i\mu)}{\widetilde{\gamma}(i\mu)} = \frac{\sqrt{T}}{4} = \frac{\pi}{4} + \mathcal{O}\left(\frac{1}{n^{2/3}}\right).
\endeq
Using both \eqref{zetaprime-ito-n} and \eqref{log-deriv-gamma} then gives
\eq
\begin{split}
[{\bf F}^{-1}{\bf G}]_{12} & = i(-1)^k\left(\frac{4i\zeta'(i\mu)}{n^{1/3}\sqrt{T}}\right)^{2k}\frac{\widetilde{\gamma}'(i\mu)}{\widetilde{\gamma}(i\mu)} = -i2^{(2k-6)/3}\pi + \mathcal{O}\left(\frac{1}{n^{2/3}}\right),\\
[{\bf F}^{-1}{\bf G}]_{21} & = -i(-1)^k\left(\frac{4i\zeta'(i\mu)}{n^{1/3}\sqrt{T}}\right)^{-2k}\frac{\widetilde{\gamma}'(i\mu)}{\widetilde{\gamma}(i\mu)} = -i2^{(-2k-6)/3}\pi + \mathcal{O}\left(\frac{1}{n^{2/3}}\right).
\end{split}
\endeq
Therefore
\eq
\label{FinvGR}
\begin{split}
[{\bf F}^{-1}{\bf G}]_{12}[{\bf R}]_{21} & = -\frac{\pi}{2i}R_\mathcal{V}e^{-2\pi i\tau} + \mathcal{O}\left(\frac{1}{n^{2/3}}\right), \quad \left(k-\frac{1}{2}\right)\frac{\log n}{3\pi n} < \mu \le k\frac{\log n}{3\pi n}, \\
[{\bf F}^{-1}{\bf G}]_{21}[{\bf R}]_{12} & = \frac{\pi}{2i}R_\mathcal{U}e^{2\pi i\tau} + \mathcal{O}\left(\frac{1}{n^{2/3}}\right), \quad k\frac{\log n}{3\pi n} < \mu \le \left(k+\frac{1}{2}\right)\frac{\log n}{3\pi n}.
\end{split}
\endeq
We now use \eqref{B11}, \eqref{FFR}, and \eqref{FinvGR} to find 
\eq
\label{B11-simplified}
[{\bf B}]_{11} = \begin{cases} \displaystyle \frac{ R_\mathcal{V}e^{-2\pi i\tau}  }{1 - \frac{\pi}{2i}R_\mathcal{V}e^{-2\pi i\tau}} + \mathcal{O}\left(\frac{1}{n^{2/3}}\right), & \displaystyle \left(k-\frac{1}{2}\right)\frac{\log n}{3\pi n} < \mu \le k\frac{\log n}{3\pi n}, \vspace{.05in} \\ \displaystyle \frac{ R_\mathcal{U}e^{2\pi i\tau} }{1 + \frac{\pi}{2i}R_\mathcal{U}e^{2\pi i\tau} } + \mathcal{O}\left(\frac{1}{n^{2/3}}\right), & \displaystyle k\frac{\log n}{3\pi n} < \mu \le \left(k+\frac{1}{2}\right)\frac{\log n}{3\pi n}. \end{cases}
\endeq

\subsection{Proofs of Lemmas \ref{lemma-cnnnm1}, \ref{lemma:asymptotics_of_OPs_critical}, and \ref{lemma:asymptotics_of_OPs_origin}}
\label{subsec-lemma-proofs}
We are now ready to prove the three results on discrete orthogonal 
polynomials needed to establish Theorems \ref{thm-tacnode-winding} and 
\ref{thm:k-tacnode-kernel}.  

\begin{proof}[Proof of Lemma \ref{lemma-cnnnm1}]
From \eqref{winding-probs-ito-cnnnm1} and \eqref{IP5a}, we merely need $[{\bf P}_{n,1}]_{11}$ in 
order to compute the winding probabilities.  
Reversing the changes of variables used in the nonlinear steepest-descent 
analysis, 
\begin{equation}
\mathbf P_n(z) =\bbm 1 & 0 \\ 0 & -2\pi i \ebm^{-1} e^{n\ell\sg_3/2} {\bf Z}_n^{(k)}(z) {\bf Y}_n^{(k)}(z) \mathbf M^{(k)}(z) e^{n(g(z) - \ell/2)\sg_3} \bbm 1 & 0 \\ 0 & -2\pi i \ebm
\end{equation}
for $|z|$ sufficiently large.  The matrices ${\bf M}^{(k)}(z)$, 
${\bf Y}_n^{(k)}(z)$, and ${\bf Z}_n^{(k)}(z)$ can be expanded as 
\eq
\begin{split}
{\bf M}^{(k)}(z) = \mathbb{I}+\frac{\mathbf M_1^{(k)}}{z}+\mathcal{O}\left(\frac{1}{z^2}\right), \quad 
{\bf Y}_n^{(k)}(z) = \mathbb{I}+\frac{\mathbf B}{z}+\mathcal{O}\left(\frac{1}{z^2}\right), \quad {\bf Z}_n^{(k)}(z) = \mathbb{I}+\frac{\mathbf Z_{n,1}^{(k)}}{z}+\mathcal{O}\left(\frac{1}{z^2}\right), 
\end{split}
\endeq
where ${\bf M}_1^{(k)}$, ${\bf B}$, and ${\bf Z}_{n,1}^{(k)}$ are independent of $z$ and the 
expansion for ${\bf Y}_n^{(k)}(z)$ follows from \eqref{Yk} and 
\eqref{Ytilde-ansatz}.  Thus, using \eqref{g1} to expand $g(z)$,
\eq
\begin{split}
\mathbb{I}+\frac{\mathbf P_{n,1}}{z}+\mathcal{O}\left(\frac{1}{z^2}\right) = 
 & \bbm 1 & 0 \\ 0 & -2\pi i \ebm^{-1} e^{n\ell\sg_3/2} \left(\mathbb{I}+\frac{ {\bf M}_1^{(k)} + {\bf B} + {\bf Z}_{n,1}^{(k)}  }{z}+\mathcal{O}\left(\frac{1}{z^2}\right)\right) \\
 &  \times \exp\left(-n\left(\frac{i\mu}{z}+\mathcal{O}\left(\frac{1}{z^2}\right)\right)\sg_3\right) e^{-n\ell\sg_3/2} \bbm 1 & 0 \\ 0 & -2\pi i \ebm.
\end{split}
\endeq
Thus
\eq
\label{Pn1-11}
[\mathbf P_{n,1}]_{11} = \left[\mathbf M_1^{(k)}\right]_{11} + \left[{\bf B}\right]_{11} + \left[\mathbf Z_{n,1}^{(k)}\right]_{11} - in\mu.
\endeq
For $[{\bf B}]_{11}$ refer to \eqref{B11-simplified}.  

We continue by calculating $\left[{\bf M}_1^{(k)}\right]_{11}$.  For $|z|$ 
large, 
we have that ${\bf M}^{(k)}(z)$ is given by ${\bf M}_k^\text{(out)}(z)$ (see 
\eqref{Mk-def}).  By using first \eqref{M-k-out} and then 
\eqref{gamma-def} and Lemma \ref{d-properties-lem} (g), 
\eq
\left[{\bf M}_k^\text{(out)}(z)\right]_{11} = \frac{\gamma(z)+\gamma(z)^{-1}}{2}d(z)^k = \left(1+\mathcal{O}\left(\frac{1}{z^2}\right)\right)\left(1+\frac{2ik}{\sqrt{T}z}+\mathcal{O}\left(\frac{1}{z^2}\right)\right).
\endeq
Therefore, using \eqref{T-scaling} in the second equation,
\eq
\label{Mk1-11}
\left[{\bf M}_1^{(k)}\right]_{11} = \frac{2ik}{\sqrt{T}} = \frac{2ik}{\pi} + \mathcal{O}\left(\frac{1}{n^{2/3}}\right).
\endeq

We now calculate $\left[{\bf Z}_{n,1}^{(k)}\right]_{11}$.  From 
\eqref{VZ-on-Dimu} and \eqref{VZ-off-Dimu}, the small-norm theory of 
Riemann--Hilbert problems (see, for example, \cite{DeiftZ:1993} and 
\cite[\S 5]{Liechty:2012}) shows that ${\bf Z}_n^{(k)}(z)$ can be computed 
via a convergent Neumann series, with
\eq
\label{Zk-integral-form}
{\bf Z}_{n,1}^{(k)} = \frac{-1}{2\pi i}\int_{\Sigma^{({\bf Z})}}\left({\bf V}^{({\bf Z})}(u)-\mathbb{I}\right) du + \mathcal{O}\left(\frac{1}{n^{2/3}}\right) = \frac{-1}{2\pi i}\int_{\partial\mathbb{D}_{i\mu}}\left({\bf V}^{({\bf Z})}(u)-\mathbb{I}\right) du + \mathcal{O}\left(\frac{1}{n^{2/3}}\right).
\endeq
The last integral can be computed by Cauchy's integral formula using 
\eqref{VZ-formula}, \eqref{E-Q-expansions}, and \eqref{Qhat-expansion} (recall 
$\partial\mathbb{D}_{i\mu}$ is oriented in the negative direction):
\eq
\label{Z-ito-YFRFY}
{\bf Z}_{n,1}^{(k)} = -{\bf Y}_n^{(k)}(i\mu){\bf F}{\bf \widehat{R}}{\bf F}^{-1}{\bf Y}_n^{(k)}(i\mu)^{-1} + \mathcal{O}\left(\frac{1}{n^{2/3}}\right).
\endeq
We now show ${\bf Y}_n^{(k)}(i\mu)$ and its inverse can be neglected without 
increasing the error.  From the formula \eqref{Yn-explicit} for 
${\bf Y}_n^{(k)}(z)$ and the formulas \eqref{Q-def} and \eqref{B-formula} for 
${\bf Q}(z)$ and ${\bf B}$ (along with the boundedness of ${\bf E}_k(z)$ in $n$),
\eq
{\bf Y}_n^{(k)}(i\mu) = \begin{cases} \displaystyle \mathbb{I} + \mathcal{O}\left(\frac{n^{(2k-1)/3}}{e^{2n\pi\mu}}\right), & \displaystyle \left(k-\frac{1}{2}\right)\frac{\log n}{3\pi n} < \mu \le k\frac{\log n}{3\pi n}, \vspace{.05in} \\ \displaystyle \mathbb{I} + \mathcal{O}\left(\frac{e^{2n\pi\mu}}{n^{(2k+1)/3}}\right), & \displaystyle k\frac{\log n}{3\pi n} < \mu \le \left(k+\frac{1}{2}\right)\frac{\log n}{3\pi n}. \end{cases}
\endeq
The same bounds hold for ${\bf Y}_n^{(k)}(i\mu)^{-1}$.  On the other hand, 
\eq
\label{Rhat-bounds}
\widehat{\bf R} = \begin{cases} \displaystyle \mathcal{O}\left(\frac{e^{2n\pi\mu}}{n^{(2k+1)/3}}\right), & \displaystyle \left(k-\frac{1}{2}\right)\frac{\log n}{3\pi n} < \mu \le k\frac{\log n}{3\pi n}, \vspace{.05in} \\ \displaystyle \mathcal{O}\left(\frac{n^{(2k-1)/3}}{e^{2n\pi\mu}}\right), & \displaystyle k\frac{\log n}{3\pi n} < \mu \le \left(k+\frac{1}{2}\right)\frac{\log n}{3\pi n}. \end{cases}
\endeq
Plugging these relations into \eqref{Z-ito-YFRFY} and expanding shows all but the 
identity terms from ${\bf Y}_n^{(k)}(i\mu)$ and ${\bf Y}_n^{(k)}(i\mu)^{-1}$ can be 
absorbed into the error term, leaving
\eq
\label{Z-ito-FRF}
{\bf Z}_{n,1}^{(k)} = -{\bf F}{\bf \widehat{R}}{\bf F}^{-1} + \mathcal{O}\left(\frac{1}{n^{2/3}}\right).
\endeq
We note that $\det{\bf F}=\det{\bf E}_k(i\mu)=1$ since $\det{\bf E}_k(z)\equiv 1$.  
Thus 
\eq
\left[{\bf Z}_{n,1}^{(k)}\right]_{11} = \begin{cases} \displaystyle 
[{\bf F}]_{11}[{\bf F}]_{21}\left[\widehat{\bf R}\right]_{12} + \mathcal{O}\left(\frac{1}{n^{2/3}}\right), 
& \displaystyle \left(k-\frac{1}{2}\right)\frac{\log n}{3\pi n} < \mu \le k\frac{\log n}{3\pi n}, \vspace{.05in} \\ \displaystyle 
-[{\bf F}]_{12}[{\bf F}]_{22}\left[\widehat{\bf R}\right]_{21} + \mathcal{O}\left(\frac{1}{n^{2/3}}\right), 
& \displaystyle k\frac{\log n}{3\pi n} < \mu \le \left(k+\frac{1}{2}\right)\frac{\log n}{3\pi n}. \end{cases}
\endeq
From \eqref{Rhat-expansion}, \eqref{F12F22-F11F21}, and \eqref{RV-RU-def}, 
\eq
\left[{\bf Z}_{n,1}^{(k)}\right]_{11} = \begin{cases} \displaystyle R_\mathcal{U}e^{2i\pi\tau} + \mathcal{O}\left(\frac{1}{n^{2/3}}\right), & \displaystyle \left(k-\frac{1}{2}\right)\frac{\log n}{3\pi n} < \mu \le k\frac{\log n}{3\pi n}, \vspace{.05in} \\ \displaystyle R_\mathcal{V}e^{-2i\pi\tau} + \mathcal{O}\left(\frac{1}{n^{2/3}}\right), & \displaystyle k\frac{\log n}{3\pi n} < \mu \le \left(k+\frac{1}{2}\right)\frac{\log n}{3\pi n}. \end{cases}
\endeq
Comparing this to \eqref{B11-simplified}, we see that it is advantageous to instead 
write 
\eq
\label{Zk1-11}
\left[{\bf Z}_{n,1}^{(k)}\right]_{11} = \begin{cases} \displaystyle \frac{R_\mathcal{U}e^{2i\pi\tau}}{1+\frac{\pi}{2i}R_\mathcal{U}e^{2\pi i\tau}} + \mathcal{O}\left(\frac{1}{n^{2/3}}\right), & \displaystyle \left(k-\frac{1}{2}\right)\frac{\log n}{3\pi n} < \mu \le k\frac{\log n}{3\pi n}, \vspace{.05in} \\ \displaystyle \frac{R_\mathcal{V}e^{-2i\pi\tau}}{1-\frac{\pi}{2i}R_\mathcal{V}e^{-2\pi i\tau}} + \mathcal{O}\left(\frac{1}{n^{2/3}}\right), & \displaystyle k\frac{\log n}{3\pi n} < \mu \le \left(k+\frac{1}{2}\right)\frac{\log n}{3\pi n}, \end{cases}
\endeq
which is legal by \eqref{Rhat-bounds}.  Now, combining \eqref{Pn1-11}, 
\eqref{B11-simplified}, \eqref{Mk1-11}, and \eqref{Zk1-11}, $[{\bf P}_{n,1}]_{11}$ 
can be written in the unified form 
\eq
\label{Pn1-final}
[{\bf P}_{n,1}]_{11} = \frac{2ik}{\pi} + \frac{R_\mathcal{V}e^{-2i\pi\tau}}{1-\frac{\pi}{2i}R_\mathcal{V}e^{-2\pi i\tau}} + \frac{R_\mathcal{U}e^{2i\pi\tau}}{1+\frac{\pi}{2i}R_\mathcal{U}e^{2\pi i\tau}} - in\mu + \mathcal{O}\left(\frac{1}{n^{2/3}}\right).
\endeq
Since $[{\bf P}_{n,1}]_{11}=c_{n,n,n-1}^{(T,\mu,\tau)}$, this completes the 
proof of Lemma \ref{lemma-cnnnm1}.
\end{proof}

\begin{proof}[Proof of Lemma \ref{lemma:asymptotics_of_OPs_critical}]
In the domain $\{z \mid \lvert \Im z\rvert >\max(\delta,\mu+\epsilon)\}$ the transformations of the Riemann--Hilbert problem give that
\eq\label{Pn_out_expansion}
\mathbf P_n(z) =\bbm 1 & 0 \\ 0 & -2\pi i \ebm^{-1} e^{n\ell\sg_3/2} {\bf X}_n^{(k)}(z) \mathbf M_k^{({\rm out})}(z) e^{n(g(z) - \ell/2)\sg_3} \bbm 1 & 0 \\ 0 & -2\pi i \ebm.
\endeq
With the choice $\mu = \frac{k\log n}{3\pi n}$, the matrix ${\bf X}_n^{(k)}(z)$ is uniformly close to the identity matrix:
\eq
{\bf X}_n^{(k)}(z)= \mathbb{I}+\bigO(n^{-1/3}).
\endeq
Thus, expanding \eqref{Pn_out_expansion} and recalling that the entries of $\mathbf P_n(z)$ are given in \eqref{IP3}, we immediately obtain the results of Lemma \ref{lemma:asymptotics_of_OPs_critical}.
\end{proof}

\begin{proof}[Proof of Lemma \ref{lemma:asymptotics_of_OPs_origin}]
We first prove the rough estimate \eqref{eq:rough_estimate_critical_zero}. From \eqref{st2}, \eqref{Mk-def}, and \eqref{X-crit-def} we have
\eq\label{Pn-first-crit-expansion}
 \mathbf P_n(z) = e^{\frac{n\ell}{2}\sg_3} {\bf A}^{-1} {\bf X}_n^{(k)}(z){\bf M}_k^{(i\mu)}(z) {\bf H}(z) \bbm 1 & 0 \\ \pm e^{\mp nG(z)} & 1 \ebm {\bf A} e^{n(g(z)-\ell/2)\sg_3} {\bf D}^u_\pm(z)^{-1} \quad {\rm for } \ \pm (\Im z - \mu)>0, \
\endeq
where the matrix ${\bf H}(z)$
is defined piecewise as
\eq
{\bf H}(z) := \begin{cases}
\bbm 1 & \mp e^{\pm nG(z)}e^{\pm 2\pi i(nz-\tau)} \\ 0 & 1 \ebm, & z\in \Om_\pm^\Delta, \\
\mathbb{I}, &{\rm otherwise}.
\end{cases}
\endeq 
Note that both ${\bf M}_k^{(i\mu)}(z)$ and ${\bf X}_n^{(k)}(z)$ are uniformly bounded in $n$. Furthermore, from the formula \eqref{int-def-G}
we find that for $x$ close to zero and small $y>0$, 
\eq\label{G-est}
\Re G(x+i\mu\pm iy) = \pm yT\sqrt{\frac{4}{T} - x^2}+\bigO(y^2).
\endeq
This implies that $\ep$ may be chosen small enough so that $e^{\mp nG(z)}=\bigO(1)$ for all $z$ such that $|z-i\mu|<\ep$ and $\pm (\Im z - \mu)>0$.  Thus, we have
\eq\label{O1-est}
\bbm 1 & 0 \\ \pm e^{\mp nG(z)} & 1 \ebm = \bigO(1).
\endeq
 Also, using \eqref{G-est} in the scalings $T=\pi^2(1+\bigO(n^{-2/3}))$ and $\mu = \frac{k\log n}{3\pi n}$, we find that 
 \eq\label{eG-rough-est}
 e^{\pm nG(z)}e^{\pm 2\pi i(nz-\tau)}= \bigO(n^{\frac{2k}{3}} e^{rn^{1/3}}),
 \endeq 
 where  $r:= \epsilon \sup_n |\sqrt{T}-\pi | n^{2/3}$ is an explicit constant. This estimate is very rough and comes from the fact that the density $\frac{T}{2\pi}\sqrt{\frac{4}{T}-x^2}$ may be slightly bigger than one in the $k$-tacnode scaling, so the sum $G(z)+2\pi i z$ could be positive. If $T<\pi^2$ then $ e^{\pm nG(z)}e^{\pm 2\pi i(nz-\tau)}$ is in fact small throughout a neighborhood of $i\mu$, but the rough estimate \eqref{eG-rough-est} is good enough for our purposes. Thus, we have
 \eq\label{bigG-est}
 {\bf H}(z)  = \bigO(n^{\frac{2k}{3}} e^{rn^{1/3}}).
 \endeq
Plugging \eqref{O1-est} and \eqref{bigG-est} into \eqref{Pn-first-crit-expansion} we find
\eq\label{Pn-est-expansion}
 \mathbf P_n(z) = e^{\frac{n\ell}{2}\sg_3}\bigO(n^{\frac{2k}{3}} e^{rn^{1/3}})e^{n(g(z)-\ell/2)\sg_3} {\bf D}^u_\pm(z)^{-1} \qquad {\rm for} \ \pm (\Im z - \mu)>0. \
\endeq
Since the quantities estimated in \eqref{eq:rough_estimate_critical_zero} are given by the entries of the first column of $ \mathbf P_n(z) $, expanding \eqref{Pn-est-expansion} immediately gives \eqref{eq:rough_estimate_critical_zero}.

Now let us give a more precise formula for  $\mathbf P_n(z)$ in a neighborhood of $z=i\mu$. Using \eqref{M-k-imu} we can write \eqref{Pn-first-crit-expansion} as
\begin{multline}
 \mathbf P_n(z) = e^{\frac{n\ell}{2}\sg_3} {\bf A}^{-1} {\bf X}_n^{(k)}(z){\bf E}_k(z) \sg_3e^{i\theta\sg_3} n^{-k\sg_3/3} {\bf Z}_k(\zeta(z))e^{-i\theta\sg_3} \sg_{2\mp 1} {\bf H}(z) \\
 \times \bbm 1 & 0 \\ \pm e^{\mp nG(z)} & 1 \ebm {\bf A} e^{n(g(z)-\ell/2)\sg_3} {\bf D}^u_\pm(z)^{-1} \qquad {\rm for} \ \pm (\Im z - \mu)>0. \
\end{multline}
From \eqref{L-def} and \eqref{zeta-conf-map} we find that 
\eq 
{\bf Z}_k(\zeta(z))e^{-i\theta\sg_3} = {\bf L}_k(\zeta(z))e^{ \frac{nG(z)}{2}\sg_3}e^{i\pi(nz-\tau)\sg_3}.
\endeq 

If we denote by $\widetilde{\bf L}_k(\zeta(z))$ the analytic continuation of  ${\bf L}_k(\zeta(z))$ from the sectors that include the real line, then the jump conditions for ${\bf L}_k(\zeta(z))$ imply that 
\eq
{\bf L}_k(\zeta(z)) = 
\begin{cases}
\vspace{.05in} \widetilde{\bf L}_k(\zeta(z))\bbm 1 & 0 \\ 1 & 1 \ebm, & z\in \Om_+^\Delta, \\
\widetilde{\bf L}_k(\zeta(z))\bbm 1 & 1 \\ 0 & 1 \ebm, & z\in \Om_-^\Delta, \\
\widetilde{\bf L}_k(\zeta(z)), &{\rm otherwise}.
\end{cases}
\endeq
Combining the three preceding equations, we can write the more uniform expression
\begin{multline}\label{Pn_unif}
 \mathbf P_n(z) = e^{\frac{n\ell}{2}\sg_3} {\bf A}^{-1} {\bf X}_n^{(k)}(z){\bf E}_k(z)e^{i\theta\sg_3}  n^{-k\sg_3/3}  \widetilde{\bf L}_k(\zeta(z)) e^{i\pi(nz-\tau)\sg_3} \bbm 1 & \vspace{.03in} \frac{1\pm 1}{2} \\ 1 &  -\frac{1\mp 1}{2}  \ebm\\
 \times e^{\mp\frac{nG(z)}{2}\sg_3}  {\bf A} e^{n(g(z)-\ell/2)\sg_3} {\bf D}^u_\pm(z)^{-1} \qquad {\rm for} \ \pm (\Im z - \mu)>0.
\end{multline}
Note that for $\mu = \frac{k\log n}{3\pi n}$, the error $ {\bf X}_n^{(k)}(z)$ is uniformly $\mathbb{I}+\bigO(n^{-1/3})$. Also using \eqref{theta-explicit}, we see that 
\eq
e^{i\theta\sg_3}  n^{-k\sg_3/3} = e^{-\frac{in\pi}{2}\sg_3}e^{i\pi\tau\sg_3}.
\endeq

From \eqref{IP3} and the fact that $\det \mathbf P_n(z) \equiv 1$, we have the formula
\begin{multline}
      e^{-\frac{nT}{4}(z^2-2i\mu z+w^2-2i\mu w)} \frac{p^{(T,\mu, \tau)}_{n, n}(z)p^{(T,\mu, \tau)}_{n, n-1}(w)-p^{(T\mu, \tau)}_{n, n-1}(z)p^{(T,\mu, \tau)}_{n, n}(w)}{h^{(T,\mu, \tau)}_{n, n-1}(z-w)} = \\ \frac{  e^{-n(V(z)+V(w))/2} }{z-w}\bbm 0 &1\ebm \mathbf P_n(z)^{-1} \mathbf P_n(w) \bbm 1 \\ 0 \ebm,
\end{multline}
where $V(z)$ is defined in \eqref{V-def}. Plugging \eqref{Pn_unif} into this formula, we find that
\begin{multline}\label{ker-expansion-exact}
\frac{  e^{-n(V(z)+V(w))/2} }{z-w}\bbm 0 & 1 \ebm \mathbf P_n(z)^{-1} \mathbf P_n(w) \bbm 1 \\ 0 \ebm = 
\frac{1}{2\pi i(z-w)} \bbm -e^{-i\pi(nz-\tau)} & e^{i\pi(nz-\tau)}  \ebm  \\ 
\times \widetilde{\bf L}_k(\zeta(z))^{-1}e^{-i\pi\tau\sg_3}e^{\frac{in\pi}{2}\sg_3}{\bf E}_k(z)^{-1}{\bf X}_n^{(k)}(z)^{-1}{\bf X}_n^{(k)}(w) {\bf E}_k(w)e^{-\frac{in\pi}{2}\sg_3}e^{i\pi\tau\sg_3} \widetilde{\bf L}_k(\zeta(w))  \bbm  e^{i\pi(nw-\tau)} \\ e^{-i\pi(nw-\tau)}  \ebm,
\end{multline}
which is obtained by direct calculation and the relation \eqref{G-def} between $g(z)$, $G(z)$, and $V(z)$.
Recall that ${\bf X}_n^{(k)}(z)=\mathbb{I}+\bigO(n^{-1/3})$.  Since ${\bf E}_k(z)$ is analytic at $z=i\mu$, we find that ${\bf E}_k(z)^{-1} {\bf E}_k(w)= \mathbb{I}+\bigO(n^{-\delta})$ for $|z-i\mu|<\ep n^{-\delta}$ and $|w-i\mu|<\ep n^{-\delta}$. Finally, for  $|z-i\mu|<\ep n^{-\delta}$,
\eq
\zeta(z) = dn^{1/3}z(1+\bigO(n^{1/3- 2\delta})),
\endeq
which implies
\eq
 \widetilde{\bf L}_k(\zeta(z))= \widetilde{\bf L}_k(dn^{1/3} z)\left(1+\bigO(n^{1/3- 2\delta})\right).
\endeq
The equation \eqref{ker-expansion-exact} reduces to \eqref{critical_CD_kernel}, and the proof of Lemma \ref{lemma:asymptotics_of_OPs_origin} is complete.
\end{proof}

\end{document}